\documentclass[10pt,a4paper]{article}
\usepackage{amsfonts,amssymb,amsmath,amsthm,mathtools}
\usepackage{aliascnt}
\usepackage{mathrsfs}
\usepackage{enumitem}
\usepackage[margin=25mm]{geometry}
\usepackage[english]{babel}
\usepackage[T1]{fontenc}
\usepackage{lmodern}
\usepackage{microtype}
\usepackage{hyperref}
\usepackage[nameinlink,capitalize]{cleveref}
\usepackage{url}
\hypersetup{
  colorlinks=true,
  linkcolor=blue,
  citecolor=blue,
  urlcolor=blue,
  pdftitle={Quantitative propagation of chaos in total variation for unregularized Vlasov--Riesz--Fokker--Planck systems},
  pdfauthor={Ning Jiang, Juntao Wu},
  pdfsubject={Quantitative mean-field limits for unregularized repulsive Riesz interactions and the three-dimensional Vlasov--Poisson--Fokker--Planck system},
  pdfkeywords={Vlasov--Poisson--Fokker--Planck, Riesz interaction, Coulomb interaction, propagation of chaos, total variation, BBGKY hierarchy, mean-field limit},
  pdfdisplaydoctitle=true
}
\numberwithin{equation}{section}

\newcommand{\R}{\mathbb R}
\newcommand{\T}{\mathbb T}
\newcommand{\N}{\mathbb N}
\newcommand{\Pp}{\mathbb P}
\newcommand{\E}{\mathbb E}
\newcommand{\dd}{\,\mathrm d}
\newcommand{\ind}{\mathbf 1}
\newcommand{\TV}{\mathrm{TV}}
\newcommand{\esssupp}{\operatorname*{ess\,supp}}
\DeclareMathOperator{\Div}{div}

\newtheoremstyle{plainnormalnote}%
  {3pt}%
  {3pt}%
  {\itshape}%
  {0pt}%
  {\bfseries}%
  {.}%
  {.5em}%
  {\thmname{#1}\thmnumber{ #2}\thmnote{ \normalfont(#3)}}%
\theoremstyle{plainnormalnote}
\newtheorem{theorem}{Theorem}[section]
\newaliascnt{corollary}{theorem}
\newtheorem{corollary}[corollary]{Corollary}
\aliascntresetthe{corollary}
\newaliascnt{proposition}{theorem}
\newtheorem{proposition}[proposition]{Proposition}
\aliascntresetthe{proposition}
\newaliascnt{lemma}{theorem}
\newtheorem{lemma}[lemma]{Lemma}
\aliascntresetthe{lemma}
\theoremstyle{definition}

\newaliascnt{assumption}{theorem}
\newtheorem{assumption}[assumption]{Assumption}
\aliascntresetthe{assumption}
\newtheoremstyle{boldremark}%
  {3pt}%
  {3pt}%
  {\normalfont}%
  {0pt}%
  {\bfseries}%
  {.}%
  {.5em}%
  {\thmname{#1}\thmnumber{ #2}\thmnote{ \normalfont(#3)}}%
\theoremstyle{boldremark}
\newtheorem{remark}{Remark}[section]
\theoremstyle{plainnormalnote}

\crefname{theorem}{Theorem}{Theorems}
\Crefname{theorem}{Theorem}{Theorems}
\crefname{proposition}{Proposition}{Propositions}
\Crefname{proposition}{Proposition}{Propositions}
\crefname{lemma}{Lemma}{Lemmas}
\Crefname{lemma}{Lemma}{Lemmas}
\crefname{corollary}{Corollary}{Corollaries}
\Crefname{corollary}{Corollary}{Corollaries}
\crefname{definition}{Definition}{Definitions}
\Crefname{definition}{Definition}{Definitions}
\crefname{assumption}{Assumption}{Assumptions}
\Crefname{assumption}{Assumption}{Assumptions}
\crefname{remark}{Remark}{Remarks}
\Crefname{remark}{Remark}{Remarks}
\crefname{equation}{equation}{equations}
\Crefname{equation}{Equation}{Equations}

\begin{document}

\title{\texorpdfstring{\Large\bfseries Quantitative propagation of chaos in total variation for unregularized Vlasov--Riesz--Fokker--Planck systems}{Quantitative propagation of chaos in total variation for unregularized Vlasov--Riesz--Fokker--Planck systems}}
\author{Ning Jiang$^\mathrm{a}$,\quad Juntao Wu$^\mathrm{a}$\\
{\small\itshape $^\mathrm{a}$ School of Mathematics and Statistics, Wuhan University, Wuhan, Hubei 430072, P. R. China}\\[1mm]
{\small Corresponding author: Juntao Wu, \texttt{00036371@whu.edu.cn}}\\[1mm]
}
\date{}

\maketitle

\begin{abstract}
We prove quantitative propagation of chaos in total variation for the
unregularized repulsive Vlasov--Riesz--Fokker--Planck system on $\T^d$
and $\R^d$, $d\ge2$.  For product initial data and potentials with singularity
$|x|^{-s}$, $0<s<d/2$, we obtain an algebraic convergence rate for the
$k$-particle marginals on a time interval independent of the particle
number.  The constants grow at most exponentially in $k$.
The three-dimensional Coulomb case gives every rate $N^{-b}$ with
$0<b<1/3$.  The assumptions on its limiting solution hold for weighted
Sobolev data with a Gaussian velocity moment and, on $\R^3$, a polynomial spatial
moment.  The proof combines conditioning of the initial product law,
auxiliary particle equations, and weighted BBGKY estimates.
An $L^1$ estimate and pointwise exponential energy bounds control the
remainder flux in weighted $L^q$.
For correlated initial data satisfying pointwise energy and marginal
error bounds, the comparison extends to $0<s<d-1$.
\end{abstract}



\section{Introduction}\label{sec:intro}

We study quantitative propagation of chaos for stochastic particle
systems with unregularized repulsive Riesz interactions on the torus
and in the whole space.  Our main result gives an algebraic rate of
convergence in total variation for the particle marginals, including
the three-dimensional Coulomb interaction.

\subsection{Model and main results}\label{subsec:intro-model}

Let $d\ge2$, $0<s<d-1$, and $\mathcal D\in\{\T^d,\R^d\}$, where
$\T^d=(\R/\mathbb Z)^d$ has unit volume.  The phase space is
$\Omega_{\mathcal D}=\mathcal D\times\R^d$.
Fix $\sigma>0$ and $c_{d,s}>0$.  On $\R^d$, set
\begin{equation}\label{eq:intro-whole-riesz-potential}
 \phi_{\R,s}(x)=1+c_{d,s}|x|^{-s},\qquad
 K_{\R,s}(x)=-\nabla\phi_{\R,s}(x)
 =c_{d,s}s\frac{x}{|x|^{s+2}}.
\end{equation}
On the torus, let $\phi_{\T,s}\ge1$ be even and smooth away from the
origin, with lower semicontinuous extension $\phi_{\T,s}(0)=+\infty$.
In a coordinate neighborhood of the origin, assume that
\begin{equation}\label{eq:intro-periodic-riesz-potential}
 \phi_{\T,s}(x)=c_{d,s}|x|^{-s}+H_{\T,s}(x),\qquad
 H_{\T,s}\in C^\infty.
\end{equation}
Set $K_{\T,s}=-\nabla\phi_{\T,s}$ away from the origin.  Then
$K_{\T,s}(x)=c_{d,s}s x/|x|^{s+2}+O(1)$ near zero.  This class includes
smooth periodic extensions of the Euclidean singularity and, after an
additive normalization, the periodic Riesz Green kernel; see
\cite{ChodronRosenzweigSerfaty2025}.  Additive constants in the potential
leave the force unchanged.

We consider the mean-field particle system
\begin{equation}\label{eq:intro-riesz-particle}
 \dd X_i=V_i\dd t,\qquad
 \dd V_i=\frac1N\sum_{j\ne i}K_{\mathcal D,s}(X_i-X_j)\dd t
 +\sqrt{2\sigma}\dd B_i,
 \qquad i=1,\ldots,N,
\end{equation}
where $B_1,\ldots,B_N$ are independent standard $\R^d$-valued Brownian
motions, independent of the initial configuration.  Positions are understood
modulo the torus when $\mathcal D=\T^d$.  The corresponding mean-field equation
is the Vlasov--Riesz--Fokker--Planck equation
\begin{equation}\label{eq:intro-riesz-vpfp}
 \left\{
 \begin{aligned}
  &\partial_t f+v\cdot\nabla_x f+(K_{\mathcal D,s}*\rho_f)\cdot\nabla_v f
  =\sigma\Delta_v f,\\
  &\rho_f(t,x)=\int_{\R^d}f(t,x,v)\dd v,\qquad f(0)=f_0.
 \end{aligned}
 \right.
\end{equation}
Let $F_N(t)$ denote the law of \eqref{eq:intro-riesz-particle} and let
$F_{N,k}(t)$ be its $k$-particle marginal.  Write
$z_i=(x_i,v_i)$, $Z_k=(z_1,\ldots,z_k)$,
$X_k=(x_1,\ldots,x_k)$, $V_k=(v_1,\ldots,v_k)$ and
$\Delta_{V_k}=\sum_{i=1}^k\Delta_{v_i}$.
For multi-indices $\xi,\eta\in\N_0^d$, where $\N_0=\{0,1,\ldots\}$,
we write $D^{(\xi,\eta)}=\partial_x^\xi\partial_v^\eta$ and
$|(\xi,\eta)|=|\xi|+|\eta|$.
We also write $f_t=f(t,\cdot)$ and
$f_t^{\otimes k}(Z_k)=\prod_{i=1}^k f(t,z_i)$.  For any integrable phase-space
function $g$, write $\rho_g(x):=\int_{\R^d}g(x,v)\dd v$ and
$\rho_0:=\rho_{f_0}$.

We identify absolutely continuous measures with their densities.  For
probability measures on a measurable space $S$, our convention is
\begin{equation}\label{eq:TV-convention}
 \|\mu-\nu\|_{\TV}=|\mu-\nu|(S)
 =\sup_{\|\psi\|_\infty\le1}\left|\int_S\psi\dd(\mu-\nu)\right|;
 \qquad \|g-h\|_{\TV}=\|g-h\|_1
\end{equation}
for densities $g,h$.  For a Markov kernel $P$, a measurable map $\pi$,
and a coupling $(X,Y)$ of $\mu,\nu$, we have
\begin{equation}\label{eq:TV-contraction-conventions}
 \|P\mu-P\nu\|_{\TV},\ \|\pi_\#\mu-\pi_\#\nu\|_{\TV}
 \le\|\mu-\nu\|_{\TV}\le2\mathbb P(X\ne Y).
\end{equation}
We write $q'=q/(q-1)$, $\langle x\rangle=(1+|x|^2)^{1/2}$, and
$C_tX=C([0,T];X)$, $L_t^pX=L^p(0,T;X)$ on a fixed interval $[0,T]$.
The space $C_b^r$ consists of functions with bounded continuous derivatives
up to order $r$.  For $1\le q<\infty$ and a positive weight $W$ on
$\Omega_{\mathcal D}^k$, we use
\[
 \|h\|_{L^q(W)}^q=\int_{\Omega_{\mathcal D}^k}|h(Z_k)|^qW(Z_k)\dd Z_k.
\]
  The finite particle dynamics and absence of collisions are
proved in Section~\ref{sec:whole-space-dynamics}.
The letter $C$ denotes a positive constant that may change from line to line.
Subscripts record relevant parameter dependence.  Numbered constants
$C_1,C_2,\ldots$ are fixed within the argument in which they are introduced.

By a \emph{classical solution} of \eqref{eq:intro-riesz-vpfp} on $[0,T]$ we
mean a function $f\in C([0,T];L^1)$ attaining $f_0$ in $L^1$, which is $C^1$
in $(t,x)$ and $C^2$ in $v$ on $(0,T]\times\Omega_{\mathcal D}$ and satisfies
\eqref{eq:intro-riesz-vpfp} pointwise there.  Nonnegativity and unit mass are
stated separately in the results below.

We first state the assumptions on the limiting solution, common to the
comparisons for product and correlated initial laws.

\begin{assumption}[Limiting solution and initial decay]
\label{ass:limiting-solution}
Let $d\ge2$, $0<s<d-1$, $\sigma>0$, and
$\mathcal D\in\{\T^d,\R^d\}$.  The initial datum $f_0\ge0$ has unit mass,
and, for some $T>0$,
\eqref{eq:intro-riesz-vpfp} admits a nonnegative classical solution $f$ on
$[0,T]$ with initial datum $f_0$ and unit mass such that
\begin{equation}\label{eq:limiting-field-regularity}
 E_f=K_{\mathcal D,s}*\rho_f\in C([0,T];C_b^1)
\end{equation}
and
\begin{equation}\label{eq:intro-velocity-derivative}
 \int_0^T\|\nabla_v f(t)\|_1\dd t<\infty.
\end{equation}
The initial datum satisfies the following bound in the relevant domain.
\begin{enumerate}[label=\textup{(\roman*)}]
 \item If $\mathcal D=\T^d$, there exist $a_0>0$ and $C_0<\infty$ such that
 \begin{equation}\label{eq:initial-pointwise-torus}
 0\le f_0(x,v)\le C_0e^{-a_0|v|^2}.
 \end{equation}
 \item If $\mathcal D=\R^d$, there exist $M>d$, $a_0>0$ and $C_0<\infty$ such that
 \begin{equation}\label{eq:initial-pointwise-whole}
 0\le f_0(x,v)\le C_0\langle x\rangle^{-M}e^{-a_0|v|^2}.
 \end{equation}
\end{enumerate}
\end{assumption}

The main result concerns the particle system started from the product law.
For $d=3$ and $s=1$, it gives the rate $C_b^kN^{-b}$ for every $0<b<1/3$;
\Cref{cor:intro-quantitative-coulomb} supplies a weighted Sobolev class of
initial data for this application.

\begin{theorem}[Quantitative propagation of chaos]
\label{thm:quantitative-comparison}
Assume \Cref{ass:limiting-solution} and $0<s<d/2$.
Consider the unregularized particle system with initial law $f_0^{\otimes N}$, independent
of the Brownian motions.  For any $q$ satisfying
\begin{equation}\label{eq:intro-q-range}
 2\le q<\infty,\qquad q>\frac d{d-s-1},
\end{equation}
there exist $0<\tau\le T$ and $C\ge1$, independent of $N$ and $k$,
such that the marginals satisfy
\begin{equation}\label{eq:intro-TV-general}
 \sup_{0\le t\le\tau}\|F_{N,k}(t)-f_t^{\otimes k}\|_{\TV}
 \le C^k\left(N^{-1/q}+N^{-(d/s-2)}\right)
\end{equation}
for every $N\ge1$ and $1\le k\le N$.
For each $N$, the unregularized particle system admits a pathwise unique
global strong solution, and collisions occur with probability zero.
The constants $\tau,C$ depend only on $d,s,q,\sigma,T$, the domain and
potential, and the norms and constants in \Cref{ass:limiting-solution}.
\end{theorem}
\begin{remark}
The exponents in \eqref{eq:intro-TV-general} correspond to the two
integrability conditions
\begin{equation}\label{eq:intro-kernel-integrability}
 \int_0^1 r^{d-1-2s}\dd r<\infty
 \quad\Longleftrightarrow\quad s<\frac d2,
 \qquad
 \int_0^1 r^{d-1-(s+1)q'}\dd r<\infty
 \quad\Longleftrightarrow\quad q>\frac d{d-s-1}.
\end{equation}
The first enters the conditioning estimate, the second the weighted
H\"older estimate.  Velocity diffusion gives the $L^q$ energy estimate
for $q\ge2$.  Polynomial spatial weights give the whole-space kernel
bound in \Cref{lem:common-riesz-regularization}.
\end{remark}
\begin{remark}[Time interval and parameter dependence]
	The quantitative estimate holds on an $N$-independent short-time interval
	chosen to satisfy the summability condition \eqref{eq:TV-time-smallness}.
	The constants and the interval depend on the fixed admissible exponent
	$q$ and the diffusion coefficient $\sigma>0$.  In the Coulomb case,
	they depend on the choice of $0<b<1/3$ and $\sigma>0$.
\end{remark}

\subsection{Initial data and the three-dimensional VPFP equation}

For periodic data, we define
\begin{equation}\label{eq:periodic-classical-data-space}
 C^m_{\T,a}
 :=\{h\in C^m(\T^d\times\R^d):\|h\|_{C^m_{\T,a}}<\infty\},
 \quad
 \|h\|_{C^m_{\T,a}}
 :=\max_{|\xi|+|\eta|\le m}\sup_{x,v}
 e^{a|v|^2}|\partial_x^\xi\partial_v^\eta h(x,v)|.
\end{equation}
For data on $\R^d$, set
\begin{equation}\label{eq:whole-space-classical-data-space}
 C^m_{M,a}(\R^{2d})
 :=\{h\in C^m(\R^{2d}):\|h\|_{C^m_{M,a}}<\infty\},
\end{equation}
with
\begin{equation}\label{eq:whole-space-classical-data-norm}
 \|h\|_{C^m_{M,a}}
 :=\max_{|\xi|+|\eta|\le m}\sup_{x,v}
 \langle x\rangle^M e^{a|v|^2}
 |\partial_x^\xi\partial_v^\eta h(x,v)|.
\end{equation}
Weighted $C^m$ initial data yield local classical solutions satisfying
\Cref{ass:limiting-solution}, by \Cref{prop:common-riesz-classical}.
Thus \Cref{thm:quantitative-comparison} gives the following Riesz estimate.

\begin{corollary}[Quantitative propagation of chaos for Riesz interactions]
\label{cor:intro-quantitative-riesz}
Let $d\ge2$, $0<s<d/2$, $\sigma>0$, $m\in\N$ with $m\ge3$, and $a_0>0$.
Let $f_0\ge0$ have unit mass and assume that one of the following conditions holds.
\begin{enumerate}[label=\textup{(\roman*)}]
 \item On the torus $\mathcal D=\T^d$, the initial datum satisfies
 \begin{equation}\label{eq:intro-periodic-riesz-data}
  \|f_0\|_{C^m_{\T,a_0}}<\infty.
 \end{equation}
 \item In the whole space $\mathcal D=\R^d$, there is $M>d$ such that
 \begin{equation}\label{eq:intro-riesz-data}
  \|f_0\|_{C^m_{M,a_0}}<\infty.
 \end{equation}
\end{enumerate}
Then there exist $T_f>0$ and a nonnegative classical solution $f$ of
\eqref{eq:intro-riesz-vpfp} on $[0,T_f]$ with unit mass satisfying
\Cref{ass:limiting-solution}.  For the particle system with initial law
$f_0^{\otimes N}$ independent of the Brownian motions, the conclusions of
\Cref{thm:quantitative-comparison} hold for every $q$ in
\eqref{eq:intro-q-range}.  The time interval and constants depend only on
$d,s,q,\sigma,m,a_0$, the domain and potential, the displayed initial norm,
and, on $\R^d$, $M$.
\end{corollary}

\paragraph{The three-dimensional VPFP equation.}
For $d=3$ and $s=1$, the condition \eqref{eq:intro-q-range} becomes
$q>3$, and \Cref{cor:intro-quantitative-riesz} gives
$C_q^kN^{-1/q}$ in both geometries.  On $\R^3$, write
$(\phi_{\R},K_{\R})=(\phi_{\R,1},K_{\R,1})$ for the Coulomb
potential and force in \eqref{eq:intro-whole-riesz-potential}, with
$c_{3,1}=c>0$.
On $\T^3$ we take the Coulomb Green function $G_{\T}$ defined by
\[
 -\Delta G_{\T}=4\pi(\delta_0-1),\qquad \int_{\T^3}G_{\T}=0,
\]
and choose $C_{\T}$ so that
$\phi_{\T}=C_{\T}+cG_{\T}\ge1$, $K_{\T}=-\nabla\phi_{\T}$.
Then $G_{\T}(x)=|x|^{-1}+H_{\T}(x)$ near zero with $H_{\T}$ smooth, so this
choice belongs to the periodic class above.

\paragraph{Weighted Sobolev notation.}
For the VPFP result below, define
\begin{equation}\label{eq:intro-sobolev-torus-norm}
 \|h\|_{H^m_{\ell}(\T^3\times\R^3)}^2
 :=\sum_{|\xi|+|\eta|\le m}\int_{\T^3\times\R^3}
 \langle v\rangle^{2\ell}
 |\partial_x^\xi\partial_v^\eta h|^2\dd x\dd v,
\end{equation}
and
\begin{equation}\label{eq:intro-sobolev-whole-norm}
 \|h\|_{H^m_{\mu,\ell}(\R^6)}^2
 :=\sum_{|\xi|+|\eta|\le m}\int_{\R^6}
 \langle x\rangle^{2\mu}\langle v\rangle^{2\ell}
 |\partial_x^\xi\partial_v^\eta h|^2\dd x\dd v.
\end{equation}
These norms define $H^m_\ell(\T^3\times\R^3)$ and
$H^m_{\mu,\ell}(\R^6)$ through distributional derivatives.
The whole-space norm agrees at $t=0$ with the time-dependent norm in
Appendix~\ref{app:coulomb-estimates}.  The local interpolation estimates
in that appendix give \Cref{ass:limiting-solution} under the following
hypotheses.

\begin{corollary}[Quantitative propagation of chaos for 3D VPFP]
\label{cor:intro-quantitative-coulomb}
Let $m\in\N$ with $m\ge6$, $\sigma,c>0$, and let $f_0\ge0$ have unit mass.
Use the Coulomb force $K_{\R}$ or $K_{\T}$ defined above.
Assume that, for some $a_0>0$, one of the following conditions holds:
\begin{enumerate}[label=\textup{(\roman*)}]
 \item On the torus $\mathcal D=\T^3$, there is $\ell>3/2$ such that
 \begin{equation}\label{eq:intro-coulomb-torus-data}
  f_0\in H^m_{\ell}(\T^3\times\R^3),\qquad
  \int_{\T^3\times\R^3}e^{a_0|v|^2}f_0\dd x\dd v<\infty.
 \end{equation}
 \item In the whole space $\mathcal D=\R^3$, there are $\mu,\ell>3/2$ and $P\ge0$ such that
 \begin{equation}\label{eq:intro-coulomb-spatial-condition}
  \mu>3\quad\text{or}\quad6\mu+(m-3)P>3(m+3),
 \end{equation}
 and the initial datum satisfies
 \begin{equation}\label{eq:intro-coulomb-whole-data}
  f_0\in H^m_{\mu,\ell}(\R^6),\qquad
  \int_{\R^6}\langle x\rangle^P e^{a_0|v|^2}f_0\dd x\dd v<\infty.
 \end{equation}
\end{enumerate}
Then there exist $T_f>0$ and a nonnegative classical solution $f$ of the
corresponding VPFP equation on $[0,T_f]$ with unit mass.  Consider the particle
system with initial law $f_0^{\otimes N}$, independent of the Brownian
motions.  For every fixed $0<b<1/3$ there exist $0<\tau_b\le T_f$ and
$C_b\ge1$, independent of $N$ and $k$, such that the marginals satisfy
\begin{equation}\label{eq:intro-TV-coulomb-b}
 \sup_{0\le t\le\tau_b}\|F_{N,k}(t)-f_t^{\otimes k}\|_{\TV}
 \le C_b^kN^{-b}
\end{equation}
for every $N\ge1$ and $1\le k\le N$.
For each $N$, the unregularized particle system admits a pathwise unique
global strong solution, and collisions occur with probability zero.  The
constants $\tau_b,C_b$ depend only on the domain, $b$, $\sigma,c,m,a_0$,
the relevant Sobolev exponents and norms, the displayed moment, and,
on $\R^3$, $P$.
\end{corollary}

\begin{remark}
The solution and its decay estimates follow from
\Cref{prop:gauss-coulomb-sobolev-core,prop:gaussian-coulomb-solution}.
For example, we may take
\begin{equation}\label{eq:intro-polynomial-data-example}
 m=6,\qquad \mu=2,\qquad P=6,
 \qquad \frac{6\mu+(m-3)P}{m+3}=\frac{10}{3}>3.
\end{equation}
Thus \eqref{eq:intro-coulomb-whole-data} includes normalized densities
$f_0(x,v)=C\langle x\rangle^{-10}e^{-\beta|v|^2}$, with
$0<a_0<\beta$ and any fixed $\ell>3/2$.
For $\mu>3$, one may take $P=0$.
\end{remark}
\begin{remark}
For
\[
 0<b\le\min\left\{\frac12,\frac ds-2\right\},
 \qquad b<\frac{d-s-1}{d},
\]
the choice $q=1/b$ satisfies \eqref{eq:intro-q-range} and gives
\[
 N^{-1/q}+N^{-(d/s-2)}\le2N^{-b}.
\]
Thus \eqref{eq:intro-TV-general} gives $C_b^kN^{-b}$ on
$[0,\tau_b]$.  If $s<(d-2)/2$, one may take $q=2$ and obtain
$C^k(N^{-1/2}+N^{-(d/s-2)})$.
\end{remark}

\begin{remark}[Dependence on $k$]\label{rem:intro-growing-k}
Fix $0<b<1/3$ as in \Cref{cor:intro-quantitative-coulomb}.
The estimate gives total variation convergence, uniformly on $[0,\tau_b]$,
for every sequence $1\le k_N\le N$ satisfying
$k_N\log C_b-b\log N\to-\infty$.
In particular, when $C_b>1$, this holds for
$k_N\le c\log N$ with any $c<b/\log C_b$.
\end{remark}

\subsection{Correlated initial data}\label{subsec:intro-correlated}

Under the initial marginal error and pointwise bounds below, the
quantitative comparison holds for $0<s<d-1$.
We use the $k$-particle energy
\begin{equation}\label{eq:intro-particle-energy}
 e_{N,k}(Z_k)=\sum_{i=1}^k(1+|v_i|^2)
 +\frac1N\sum_{\substack{1\le i,j\le k\\i\ne j}}\phi_{\mathcal D,s}(x_i-x_j).
\end{equation}

\begin{theorem}[Correlated initial data]\label{thm:correlated-comparison}
Let $d\ge2$, $0<s<d-1$, $\sigma>0$, and
$\mathcal D\in\{\T^d,\R^d\}$.  Let $f_0$ and $f$ satisfy
\Cref{ass:limiting-solution}, and fix $q$ as in \eqref{eq:intro-q-range}.
For each $N$, let $Q_N^0$ be a symmetric
probability density on $\Omega_{\mathcal D}^N$, with marginals $Q_{N,k}^0$.
Assume that there are constants $A_0\ge1$ and $\beta_Q>0$, independent
of $N,k$, and numbers $0\le\delta_N\le1$ such that the following estimate holds
for every $1\le k\le N$:
\begin{equation}\label{eq:intro-correlated-marginals}
 \|Q_{N,k}^0-f_0^{\otimes k}\|_1\le A_0^k\delta_N.
\end{equation}
On $\T^d$, assume
\begin{equation}\label{eq:intro-correlated-periodic-envelope}
 Q_{N,k}^0(Z_k)\le A_0^k e^{-\beta_Q e_{N,k}(Z_k)}.
\end{equation}
On $\R^d$, assume that there is $M_Q>d$, independent of $N,k$, such that
\begin{equation}\label{eq:intro-correlated-whole-envelope}
 Q_{N,k}^0(Z_k)\le A_0^k e^{-\beta_Q e_{N,k}(Z_k)}
                  \prod_{i=1}^k\langle x_i\rangle^{-M_Q}.
\end{equation}
Let $Q_N(t)$ be the law of \eqref{eq:intro-riesz-particle} with initial law
$Q_N^0$, independent of the Brownian motions.  There exist $0<\tau\le T$
and $C\ge1$, independent of $N$, $k$, and $(\delta_N)_N$, such that
\begin{equation}\label{eq:intro-correlated-TV}
 \sup_{0\le t\le\tau}\|Q_{N,k}(t)-f_t^{\otimes k}\|_{\TV}
 \le C^k(\delta_N+N^{-1})^{1/q},
 \qquad N\ge1,\quad 1\le k\le N.
\end{equation}
The constants depend on $d,s,q,\sigma,T$, the domain and potential,
the stated norms and decay constants of $f_0,f$, $A_0$, $\beta_Q$,
and, on $\R^d$, $M_Q$.  For each $N$, the particle
system has a pathwise unique global strong solution and has no collisions
almost surely.
\end{theorem}

For weighted $C^m$ initial data with $m\ge3$,
\Cref{prop:common-riesz-classical} verifies
\Cref{ass:limiting-solution} for $0<s<d-1$.
For $0<s<d/2$, \Cref{prop:common-riesz-conditioning} provides initial
laws satisfying the displayed marginal error and pointwise energy bounds.  For
$d/2\le s<d-1$, \Cref{thm:correlated-comparison} applies to sequences
of initial laws satisfying these bounds and gives propagation of
chaos when $\delta_N\to0$.

\subsection{Relation to previous work}
\label{subsec:intro-comparison}

The classical mean-field theory for regular interactions includes the
work of Braun and Hepp \cite{BraunHepp1977} and the stability estimates
of Dobrushin \cite{Dobrushin1979}.  McKean \cite{McKean1966} introduced
a probabilistic formulation of nonlinear parabolic equations, and
Sznitman \cite{Sznitman1991} developed the theory of propagation of
chaos.  We refer to Chaintron and Diez
\cite{ChaintronDiez2022I,ChaintronDiez2022II}, Golse \cite{Golse2016},
and Jabin \cite{Jabin2014} for surveys.

For singular first-order systems, Serfaty \cite{Serfaty2020} used
modulated energy to establish mean-field convergence for Coulomb and
super-Coulombic Riesz flows.  Nguyen, Rosenzweig, and Serfaty
\cite{NguyenRosenzweigSerfaty2022} extended the method to Riesz-type
interactions with $0<s<d$, and Rosenzweig and Serfaty
\cite{RosenzweigSerfaty2023} obtained global-in-time estimates for
diffusive flows with $0<s<d-2$.  Nguyen and Serfaty
\cite{NguyenSerfaty2026} obtained quantitative mean-field limits using
heat-kernel mollification at multiple scales.

The relative entropy method of Jabin and Wang \cite{JabinWang2018}
gives quantitative propagation of chaos for stochastic systems with
$W^{-1,\infty}$ interaction kernels, including the two-dimensional
viscous vortex model.  Estimates for individual marginals were developed
by Lacker \cite{Lacker2023} using local relative entropy, and by
Hess-Childs and Rowan \cite{HessChildsRowan2025} using $\chi^2$
divergence for bounded interactions.  Songbo Wang \cite{SongboWang2026}
used a $\chi^2$ hierarchy for singular $W^{-1,\infty}$ interactions.

Bresch, Jabin, and Wang \cite{BreschJabinWang2019} introduced the
modulated free energy, combining relative entropy with the expected
modulated  energy.  They treated singular repulsive interactions in
\cite{BreschJabinWang2020} and singular attractive interactions,
including the subcritical Patlak--Keller--Segel model, in
\cite{BreschJabinWang2023}.  Chodron de Courcel, Rosenzweig, and Serfaty
\cite{ChodronRosenzweigSerfaty2025} used this method to obtain sharp
uniform-in-time convergence for periodic Riesz gradient flows.
Further developments concern modulated logarithmic Sobolev inequalities
\cite{RosenzweigSerfaty2025} and partition-function estimates for
logarithmic and Riesz interactions \cite{DelgadinoGvalaniRosenzweig2026}.

For second-order systems, Hauray and Jabin
\cite{HaurayJabin2007,HaurayJabin2015} established mean-field limits
for mildly singular Vlasov interactions, and Jabin and Wang
\cite{JabinWang2016} treated bounded forces by relative entropy.
Quantitative convergence for VPFP was proved in one dimension by
Hauray and Salem \cite{HauraySalem2019}.  In three dimensions,
results for regularized Coulomb interactions were obtained in
\cite{CarrilloChoiSalem2019,HuangLiuPickl2020}, and Chen, Jung, Pickl,
and Wang \cite{ChenJungPicklWang2026} combined trajectory and entropy
estimates to obtain $L^1$ convergence of the marginals for regularized
Coulomb systems.

Duality methods for singular interactions were developed by Bresch,
Duerinckx, and Jabin \cite{BreschDuerinckxJabin2024} and Khoury and
Jabin \cite{KhouryJabin2026}.  Duerinckx and Jabin derived the
unregularized two-dimensional Vlasov--Poisson equation
\cite{DuerinckxJabin2026} and studied fluctuations around Gibbs
equilibrium \cite{DuerinckxJabinFluctuations2026}.
Hao, Zhang, and Zhao \cite{HaoZhangZhao2026} obtained quantitative
results for kinetic systems with interactions in a kinetic Kato class.

For the limiting equations, classical solutions of three-dimensional
VPFP were studied in \cite{VictoryODwyer1990,Bouchut1993,OnoStrauss2000}.
Well-posedness results for Vlasov--Riesz and
Vlasov--Riesz--Fokker--Planck equations include
\cite{ChoiJeong2024,WangZhang2025}.  For related equations,
solutions near a global Maxwellian were constructed in
\cite{ChoiJeongKang2024}.  Our appendices establish the
weighted estimates used in
\Cref{cor:intro-quantitative-riesz,cor:intro-quantitative-coulomb}.

Bresch, Jabin, and Soler \cite{BJS2025} developed a BBGKY approach
for singular kinetic systems using velocity diffusion and exponentially
weighted $L^q$ estimates based on the internal particle energy.
Their qualitative theorem assumes exponential integrability of the
potential and includes the two-dimensional repulsive Coulomb interaction.
Under the hypotheses of their qualitative theorem, with $K\in L^2$
and quantitative weighted initial bounds, they also obtain quantitative
weighted $L^2$ estimates.  For the
three-dimensional Coulomb potential, the corresponding exponential
interaction weights are not integrable against continuous tensorized
data; see \Cref{rem:product-exponential-divergence}.

Feng and Wang \cite[Theorem~1.3]{FengWang2026} combined modulated Gibbs measures
with the partition-function estimates of Wang and Zhao
\cite{WangZhao2026} to apply the weighted BBGKY estimates to product
initial data.  For repulsive Riesz potentials with $0<s<d/2$, they
obtained weak convergence of each fixed marginal on an
$N$-independent short-time interval, on both $\T^d$ and $\R^d$.
Their initial data are smooth and have a Gaussian velocity moment on
$\T^d$, and a Gaussian moment in $(x,v)$ on $\R^d$.

The present paper proves the algebraic total variation estimate
\eqref{eq:intro-TV-general} for the unregularized system with product
initial data.  In the
three-dimensional Coulomb case, the rate is given by
\eqref{eq:intro-TV-coulomb-b}.  We combine conditioning of the initial law with auxiliary
particle equations to estimate the remainder in the weighted BBGKY
hierarchy.  The comparison also applies to correlated initial laws
satisfying the bounds of \Cref{thm:correlated-comparison}, and the
whole-space result allows polynomial spatial weights together with
Gaussian velocity moments.

\subsection{Outline of the proof}\label{subsec:intro-proof-outline}

Let $G_N$ be the evolution of the conditioned product law constructed in
\Cref{prop:common-riesz-conditioning}, and let $U_{N,k}$ solve
\eqref{eq:TV-U-PDE}.  We use the decomposition
\begin{equation}\label{eq:intro-error-decomposition}
 F_{N,k}-f_t^{\otimes k}
 =(F_{N,k}-G_{N,k})+(G_{N,k}-U_{N,k})
 +(U_{N,k}-f_t^{\otimes k}).
\end{equation}
The estimates are obtained in four steps.

\begin{enumerate}[label=\textup{(\roman*)},leftmargin=*,itemsep=5pt]
\item \emph{Conditioning of the initial product law.}
The estimates \eqref{eq:common-riesz-conditioning-cost} and
\eqref{eq:common-riesz-conditioning-marginals}, together with Markov
contraction \eqref{eq:TV-contraction-conventions}, control the first term
in \eqref{eq:intro-error-decomposition}.  In the notation of
\eqref{eq:comparison-initial-error}, we take $\delta_N=N^{1-d/s}$,
so that $\Delta_N\le2/N$.

\item \emph{Auxiliary particle equations.}
The auxiliary equation has the same internal force as the $k$th BBGKY
equation.  Its comparison with the tensorized solution is given by
\eqref{eq:TV-U-product}, with the initial error controlled by
\eqref{eq:common-riesz-conditioning-marginals}.  The equality at level
$N$ is \eqref{eq:TV-level-N-equality}.

\item \emph{Estimate of the remainder.}
For the remainder $h_{N,k+1}$ of \Cref{prop:TV-small-residual}, use the
reduced energy \eqref{eq:TV-remainder-weight} and the spatial weights
of Section~\ref{subsec:comparison-weights}, with $\alpha$ chosen in
\eqref{eq:TV-alpha-choice}.  The pointwise bounds in
\Cref{lem:TV-U-pointwise} and the preceding $L^1$ estimate give the
interpolation bound \eqref{eq:TV-h-interpolation-explicit}.
With $\Delta_N\le2/N$, this yields the required remainder estimate.
The weighted H\"older inequality \eqref{eq:common-weighted-holder}
then bounds its flux under $(s+1)q'<d$.

\item \emph{Weighted estimate and iteration of the hierarchy.}
For the error and weighted energy defined in
\eqref{eq:TV-error-definition} and \eqref{eq:comparison-weighted-error},
choose the decreasing exponent \eqref{eq:TV-global-lambda}.
The identity \eqref{eq:common-energy-identities} and velocity diffusion
give \eqref{eq:TV-error-energy}, with zero initial and terminal-level
errors as in \eqref{eq:TV-error-identities}.  Iteration terminates at
$k=N$.  The binomial series in \Cref{lem:TV-iteration} gives
$Y_{N,k}\le C^k/N$ on an interval independent of $N,k$.
A final weighted H\"older estimate gives
$\|H_{N,k}\|_1\le C^kN^{-1/q}$, and
\eqref{eq:intro-error-decomposition} proves
\eqref{eq:intro-TV-general}.
\end{enumerate}

For the initial data of \Cref{thm:correlated-comparison}, the same
calculation replaces $N^{-1}$ by $\delta_N+N^{-1}$.

Section~\ref{sec:conditioning} gives the conditioning estimates.
Section~\ref{sec:common-comparison} proves the auxiliary estimates and
iterates the hierarchy.  Section~\ref{sec:kernel-verification} gives the
kernel estimates, the particle dynamics, and the removal of the force
regularization.  The approximation limits in the energy inequality are
proved in Appendix~\ref{app:fixed-N-passage}.  Appendices
\ref{app:whole-space-classical} and \ref{app:coulomb-estimates} establish
the weighted regularity and decay estimates for the limiting equation.

\section{Conditioning of the initial product measure}\label{sec:conditioning}

We condition the initial product measure to obtain the energy bounds
required in \Cref{thm:correlated-comparison}.

\subsection{Tail estimates for the interaction sums}\label{subsec:abstract-conditioning-tail}

Let $\mathcal D\in\{\T^d,\R^d\}$, let $\rho$ be a probability density
on $\mathcal D$, and let $h:\mathcal D\times\mathcal D\to[0,+\infty]$
be symmetric and lower semicontinuous.  For $X$ with density $\rho$, assume
that there are $p>2$ and $C_h,L_h>0$ such that
\begin{equation}\label{eq:poly-tail}
 \Psi_h(L):=\sup_{x\in\mathcal D}\Pp(h(x,X)>L)
 \le C_hL^{-p},\qquad L\ge L_h.
\end{equation}
The layer-cake formula gives
\begin{equation}\label{eq:uniform-tail-moments}
 \begin{aligned}
 m_j:=\sup_{x\in\mathcal D}\E h(x,X)^j
 &\le j\int_0^{L_h}L^{j-1}\dd L
      +jC_h\int_{L_h}^{\infty}L^{j-p-1}\dd L\\
 &=L_h^j+\frac{jC_h}{p-j}L_h^{j-p}<\infty,
 \qquad j=1,2.
 \end{aligned}
\end{equation}
In the conditioning estimates, probabilities involving $X_1,\ldots,X_N$
are taken under $\rho^{\otimes N}$.  We set
\begin{equation}\label{eq:interaction-sum}
 D_i(X_N):=\sum_{j\ne i}h(X_i,X_j),\qquad
 \mathcal G_N(R):=\{\max_iD_i(X_N)\le RN\},
\end{equation}
and, for $\kappa>0$, we define
\begin{equation}\label{eq:conditioning-event-pair-bound}
 \mathcal B_N(\kappa):=
 \{h(X_i,X_j)\le\kappa N\ \text{for all }i<j\}.
\end{equation}
Empty sums and products have their usual values, so both events are the
whole space when $N=1$.

\begin{lemma}\label{lem:abstract-interaction-sums}
For every $R>m_1$ and $A>0$, there exist $\kappa\in(0,1)$ and $C<\infty$,
independent of $N$, such that
\begin{equation}\label{eq:truncated-sum-cost}
 \Pp\bigl(\mathcal G_N(R)^c\cap\mathcal B_N(\kappa)\bigr)
 \le CN^{-A},\qquad N\ge2.
\end{equation}
The probability of violating a pair constraint satisfies
\begin{equation}\label{eq:pair-exclusion-cost}
 \Pp(\mathcal B_N(\kappa)^c)
 \le\binom N2\Psi_h(\kappa N)\le C N^{2-p}.
\end{equation}
The constants depend only on $R,A$ and the constants in
\eqref{eq:poly-tail}--\eqref{eq:uniform-tail-moments}.
\end{lemma}

\begin{proof}
If $m_2=0$, then $D_i=0$ almost surely and both bounds hold.
Suppose $m_2>0$.  Put $\Delta=R-m_1>0$ and choose $\kappa\in(0,1)$ with
$\Delta/(2\kappa)>A+1$.  Conditional on $X_1=x$, we set
\[
 Y_{j,N}^x=h(x,X_j)\ind_{\{h(x,X_j)\le\kappa N\}},
 \qquad 2\le j\le N.
\]
The variables $Y_{j,N}^x$ are independent and satisfy
\[
 0\le Y_{j,N}^x\le\kappa N,\qquad
 \sum_{j=2}^N\E Y_{j,N}^x\le(N-1)m_1,\qquad
 \sum_{j=2}^N\operatorname{Var}(Y_{j,N}^x)\le Nm_2.
\]
For $b(u)=(1+u)\log(1+u)-u$, we have
$b(u)\ge(u/2)\log(1+u)$ for $u\ge0$.  Bennett's inequality
\cite[Theorem~2.9]{BoucheronLugosiMassart2013} therefore gives
\begin{align*}
 \Pp\left(D_1>RN,\ \max_{2\le j\le N}h(x,X_j)\le\kappa N
          \,\middle|\,X_1=x\right)
 &\le
 \Pp\left(\sum_{j=2}^N(Y_{j,N}^x-\E Y_{j,N}^x)>\Delta N\right)\\
 &\le
 \exp\left[-\frac{m_2}{\kappa^2N}
 b\left(\frac{\kappa\Delta N}{m_2}\right)\right]
 \le\left(1+\frac{\kappa\Delta N}{m_2}\right)^{-\Delta/(2\kappa)},
\end{align*}
Integration in $x$ and the union bound yield
\[
 \Pp\bigl(\mathcal G_N(R)^c\cap\mathcal B_N(\kappa)\bigr)
 \le N\left(1+\frac{\kappa\Delta N}{m_2}\right)^{-\Delta/(2\kappa)}
 \le CN^{-A}.
\]
The union bound over the $\binom N2$ pairs and \eqref{eq:poly-tail}
give \eqref{eq:pair-exclusion-cost}, after increasing $C$ for the
finitely many $N$ with $\kappa N<L_h$.
\end{proof}

\subsection{Conditioned marginals}\label{subsec:abstract-conditioned-hierarchy}

Let $f_0$ be a probability density on $\Omega_{\mathcal D}$ with spatial
density $\rho$.

\begin{remark}[Exponential interaction weights for product data]
\label{rem:product-exponential-divergence}
For a Riesz potential with $s>0$ and a continuous nonzero density $f_0$,
we have
\[
 \int(f_0^{\otimes k})^q e^{\lambda e_{N,k}}\dd Z_k=+\infty,
 \qquad N\ge k\ge2,\ q\ge1,\ \lambda>0.
\]
Indeed, continuity and nontriviality give a positive lower bound for
$f_0$ on a product of spatial and velocity balls.  Restrict all variables
to these balls and $x_2-x_1$ to a smaller ball on which
$\phi(x_2-x_1)\ge c_0|x_2-x_1|^{-s}$.  The two ordered pairs $(1,2)$ and
$(2,1)$ then give the lower bound
$C\int_0^{r_0}r^{d-1}\exp(2\lambda c_0/(Nr^s))\dd r=+\infty$.
\end{remark}

\begin{proposition}[Conditioning with marginal control]\label{prop:conditioning-marginals}
Let $h$ and $\rho$ be as above and satisfy \eqref{eq:poly-tail}.  There
exists $R_*>m_1$ such that, for every
$R\ge R_*$, one can choose symmetric closed events
$\mathcal G_N^\sharp(R)\subset\mathcal G_N(R)$ for which
\begin{equation}\label{eq:conditioning-full-N}
 p_N^\sharp:=\Pp(\mathcal G_N^\sharp(R))\ge\frac12,
 \qquad \eta_N^\sharp:=1-p_N^\sharp\le C_RN^{2-p},\qquad N\ge1.
\end{equation}
Define
\begin{equation}\label{eq:conditioning-density}
 G_N^{0,\sharp}:=(p_N^\sharp)^{-1}
       \ind_{\mathcal G_N^\sharp(R)}f_0^{\otimes N}.
\end{equation}
For $1\le k\le N$, its $k$-particle marginal satisfies
\begin{equation}\label{eq:conditioning-marginal-sharp-rate}
 \|G_{N,k}^{0,\sharp}-f_0^{\otimes k}\|_1\le C_RkN^{1-p},
 \qquad 0\le G_{N,k}^{0,\sharp}\le2f_0^{\otimes k},
\end{equation}
and
\begin{equation}\label{eq:conditioning-subset-energy}
 \frac1N\sum_{\substack{1\le i,j\le k\\i\ne j}}h(x_i,x_j)\le Rk
 \quad\text{on }\esssupp G_{N,k}^{0,\sharp}.
\end{equation}
Furthermore,
$\|G_N^{0,\sharp}-f_0^{\otimes N}\|_{\TV}=2\eta_N^\sharp$.
Markov evolution and coordinate projection contract this distance by
\eqref{eq:TV-contraction-conventions}.  All constants are independent
of $N$ and $k$.
\end{proposition}

\begin{proof}
\emph{Normalization and choice of the event.}
If $m_2=0$, take $\mathcal G_N^\sharp(R)=\mathcal G_N(R)$, which has
probability one, and any $R_*>0$.  Suppose $m_2>0$ and choose
$R_*=m_1+\sqrt{2m_2}$.  Conditional independence gives
\[
 \E[D_i\mid X_i=x]\le(N-1)m_1,\qquad
 \operatorname{Var}(D_i\mid X_i=x)\le(N-1)m_2.
\]
Chebyshev's inequality and the union bound give, for $N\ge2$ and
$R\ge R_*$,
\[
 \begin{aligned}
 \Pp(D_i>RN\mid X_i=x)
 &\le\frac{(N-1)m_2}{(R-m_1)^2N^2},\\
 \Pp(\mathcal G_N(R)^c)
 &\le\frac{(N-1)m_2}{(R-m_1)^2N}\le\frac12.
 \end{aligned}
\]
For $N=1$ the event is the whole space.

Fix $A>p+1$ and take $\kappa$ from
\Cref{lem:abstract-interaction-sums}.  For all sufficiently large $N$, set
\begin{equation}\label{eq:conditioning-refined-event}
 \mathcal G_N^\sharp(R)=\mathcal G_N(R)\cap\mathcal B_N(\kappa).
\end{equation}
By \eqref{eq:truncated-sum-cost}--\eqref{eq:pair-exclusion-cost}, we have
\[
 1-p_N^\sharp\le CN^{2-p}+CN^{-A}\longrightarrow0.
\]
Choose $N_*$ such that $p_N^\sharp\ge1/2$ and $\kappa N\ge L_h$ for
$N\ge N_*$.  For $N<N_*$, set
$\mathcal G_N^\sharp(R)=\mathcal G_N(R)$.  The Chebyshev bound above
then proves \eqref{eq:conditioning-full-N} after increasing $C_R$.

\emph{Marginal estimate for the pair constraints.}
Let $N\ge N_*$.  Conditioning on $\mathcal B_N(\kappa)$, we define
\[
 \mathcal Z_N=\Pp(\mathcal B_N(\kappa))\ge\tfrac12,
 \qquad Q_N^0=\mathcal Z_N^{-1}\ind_{\mathcal B_N(\kappa)}f_0^{\otimes N},
 \qquad b_N=\Psi_h(\kappa N)\le CN^{-p}.
\]
Fix $1\le k\le N$ and set $m=N-k$.  With the threshold $\kappa N$
fixed, define
\[
 \mathcal Z_m^{(N)}
 =\int\prod_{k<i<j\le N}\ind_{\{h(x_i,x_j)\le\kappa N\}}
       \prod_{j=k+1}^N\rho(x_j)\dd x_j,
 \qquad \mathcal Z_0^{(N)}=\mathcal Z_1^{(N)}=1,
\]
and
\[
 \mathfrak a_{N,k}(X_k)
 =\int\prod_{1\le i<j\le N}\ind_{\{h(x_i,x_j)\le\kappa N\}}
       \prod_{j=k+1}^N\rho(x_j)\dd x_j.
\]
Then
\begin{equation}\label{eq:pair-conditioned-density}
 Q_{N,k}^0=\frac{\mathfrak a_{N,k}}{\mathcal Z_N}f_0^{\otimes k},
 \qquad 0\le\mathfrak a_{N,k}\le\mathcal Z_m^{(N)},
 \qquad\int\mathfrak a_{N,k}\rho^{\otimes k}\dd X_k=\mathcal Z_N.
\end{equation}
The union bound gives
\begin{equation}\label{eq:pair-partition-L1}
 \begin{aligned}
 \mathcal Z_m^{(N)}-\mathcal Z_N
 &=\int(\mathcal Z_m^{(N)}-\mathfrak a_{N,k})
                 \rho^{\otimes k}\dd X_k\\
 &\le\sum_{\substack{1\le i<j\le N\\i\le k}}
       \Pp(h(X_i,X_j)>\kappa N)\\
 &\le\left(k(N-k)+\binom{k}{2}\right)b_N\le kNb_N.
 \end{aligned}
\end{equation}
By \eqref{eq:pair-conditioned-density}--\eqref{eq:pair-partition-L1}, we obtain
\begin{equation}\label{eq:pair-conditioned-marginal}
 \begin{aligned}
 \|Q_{N,k}^0-f_0^{\otimes k}\|_1
 &=\frac1{\mathcal Z_N}
    \int|\mathfrak a_{N,k}-\mathcal Z_N|\rho^{\otimes k}\dd X_k\\
 &\le\frac1{\mathcal Z_N}
    \int\bigl(\mathcal Z_m^{(N)}-\mathfrak a_{N,k}
                 +\mathcal Z_m^{(N)}-\mathcal Z_N\bigr)
                 \rho^{\otimes k}\dd X_k\\
 &=\frac{2(\mathcal Z_m^{(N)}-\mathcal Z_N)}{\mathcal Z_N}
 \le4kNb_N\le CkN^{1-p}.
 \end{aligned}
\end{equation}
\emph{Further conditioning and the energy bound.}
For a probability measure $\nu$ and an event $G$ of positive probability,
we have
\begin{equation}\label{eq:full-TV-conditioning}
 \|\nu(\cdot\mid G)-\nu\|_{\TV}
 =\int|\nu(G)^{-1}\ind_G-1|\dd\nu=2(1-\nu(G)).
\end{equation}
Conditioning $Q_N^0$ further on $\mathcal G_N(R)$ gives $G_N^{0,\sharp}$.
Hence \eqref{eq:truncated-sum-cost} implies
\[
 \|G_N^{0,\sharp}-Q_N^0\|_{\TV}
 =\frac{2\Pp(\mathcal G_N(R)^c\cap\mathcal B_N(\kappa))}{\mathcal Z_N}
 \le CN^{-A}.
\]
Projection onto the first $k$ particles and
\eqref{eq:pair-conditioned-marginal} yield
\[
 \|G_{N,k}^{0,\sharp}-f_0^{\otimes k}\|_1
 \le CkN^{1-p}+CN^{-A}\le C_RkN^{1-p}.
\]
For the finitely many $N<N_*$, the distance is at most two, so increasing
$C_R$ gives the same bound for every $1\le k\le N$.

By \eqref{eq:conditioning-density} and $p_N^\sharp\ge1/2$,
$G_{N,k}^{0,\sharp}\le2f_0^{\otimes k}$.  On
$\mathcal G_N^\sharp(R)\subset\mathcal G_N(R)$, nonnegativity of $h$ gives
\[
 \frac1N\sum_{\substack{1\le i,j\le k\\i\ne j}}h(x_i,x_j)
 \le\frac1N\sum_{i=1}^kD_i(X_N)\le Rk.
\]
The marginal vanishes almost everywhere outside this sublevel set,
which is closed by lower semicontinuity.  This proves
\eqref{eq:conditioning-subset-energy} on the essential support and also
shows that the events $\mathcal G_N^\sharp(R)$ are closed.
The remaining claims follow from \eqref{eq:full-TV-conditioning} and
\eqref{eq:TV-contraction-conventions}.
\end{proof}

\subsection{Application to Riesz interactions}\label{subsec:riesz-conditioning}

\begin{corollary}\label{prop:common-riesz-conditioning}
Let $d\ge2$, $0<s<d/2$, and let $f_0$ be a probability density on
$\Omega_{\mathcal D}$ with spatial density $\rho_0\in L^\infty$.
For $h(x,y)=\phi_{\mathcal D,s}(x-y)$, the construction of
\Cref{prop:conditioning-marginals} gives events $\mathcal G_N^\sharp(R)$,
$R\ge R_*$, and densities
\[
 G_N^0=p_N^{-1}\ind_{\mathcal G_N^\sharp(R)}f_0^{\otimes N},
 \qquad p_N=f_0^{\otimes N}(\mathcal G_N^\sharp(R)),\qquad \eta_N=1-p_N,
\]
whose marginals satisfy
\begin{gather}
 p_N\ge\tfrac12,\qquad \eta_N\le C_RN^{2-d/s},
 \qquad \|G_N^0-f_0^{\otimes N}\|_{\TV}=2\eta_N,
 \label{eq:common-riesz-conditioning-cost}\\
 G_{N,k}^0\le2f_0^{\otimes k},\qquad
 \|G_{N,k}^0-f_0^{\otimes k}\|_1\le C_RkN^{1-d/s},
 \label{eq:common-riesz-conditioning-marginals}\\
 \frac1N\sum_{\substack{1\le i,j\le k\\i\ne j}}\phi_{\mathcal D,s}(x_i-x_j)\le Rk
 \quad\text{on }\esssupp G_{N,k}^0.
 \label{eq:common-riesz-conditioned-support}
\end{gather}
The constants are independent of $N$ and $1\le k\le N$.
\end{corollary}

\begin{proof}
The potential is symmetric, lower semicontinuous, and bounded outside
any neighborhood of the origin.  For all sufficiently large $L$, its local form gives
\[
 \{\phi_{\mathcal D,s}(x-X)>L\}
 \subset\{X\in B_{CL^{-1/s}}(x)\},
 \qquad
 \sup_x\Pp(\phi_{\mathcal D,s}(x-X)>L)
 \le C\|\rho_0\|_\infty L^{-d/s},
\]
where $X$ has density $\rho_0$ and the balls are geodesic on $\T^d$.
Thus \eqref{eq:poly-tail} holds with $p=d/s>2$, and
\Cref{prop:conditioning-marginals} gives all the conclusions.
\end{proof}

If $f_0$ satisfies the pointwise bound in
\Cref{ass:limiting-solution}, then, for every $0<\beta<a_0$, there is
$A_\beta\ge1$, independent of $N,k$, such that
\begin{equation}\label{eq:common-conditioned-initial-exponential}
 G_{N,k}^0\le A_\beta^k e^{-\beta e_{N,k}}
       \prod_{i=1}^k\langle x_i\rangle^{-M}
\end{equation}
on $\R^d$, with the spatial factors omitted on $\T^d$.
Indeed, \eqref{eq:intro-particle-energy},
\eqref{eq:common-riesz-conditioning-marginals} and
\eqref{eq:common-riesz-conditioned-support} give, on
$\esssupp G_{N,k}^0$,
\[
 G_{N,k}^0e^{\beta e_{N,k}}\prod_i\langle x_i\rangle^M
 \le 2C_0^k
 \exp\left(-(a_0-\beta)\sum_i|v_i|^2+\beta(1+R)k\right)
 \le A_\beta^k,
\]
with $A_\beta:=\max\{1,2C_0e^{\beta(1+R)}\}$.
On the torus, the same calculation holds with the spatial factors omitted.

\section{Proof of the quantitative estimate}\label{sec:common-comparison}

We prove \Cref{thm:correlated-comparison} and apply it to the conditioned
product law from Section~\ref{sec:conditioning}.  Throughout this section,
$0<s<d-1$, $q$ satisfies \eqref{eq:intro-q-range}, and $f_0,f$ satisfy
\Cref{ass:limiting-solution} on $[0,T]$.  Write $\phi,K$ for the potential
and force, and use the energy $e_{N,k}$ defined in
\eqref{eq:intro-particle-energy}.

Let $G_N^0$ be a symmetric probability density whose marginals
satisfy the initial hypotheses
\eqref{eq:intro-correlated-marginals}--\eqref{eq:intro-correlated-whole-envelope}.
Here we use $G_N^0$ and $G_{N,k}^0$ in place of $Q_N^0$ and $Q_{N,k}^0$,
and write $\beta_G$ for $\beta_Q$ and $M_G$ for $M_Q$.
The constants $A_0\ge1$ and $\beta_G>0$, and, on $\R^d$, $M_G>d$,
are independent of $N,k$; the spatial factors are omitted on $\T^d$.
We retain $0\le\delta_N\le1$ from those hypotheses and set
\begin{equation}\label{eq:comparison-initial-error}
 a_{N,k}:=\|G_{N,k}^0-f_0^{\otimes k}\|_1,
 \qquad \Delta_N:=\delta_N+N^{-1}.
\end{equation}
The full initial law has finite expected energy and charges no collision
diagonal.  Hence \Cref{prop:common-finite-N} defines
$G_N(t)=P_t^NG_N^0$; write $G_{N,k}(t)$ for its marginals.  In this notation,
the estimate to prove is \eqref{eq:intro-correlated-TV} with
$Q_{N,k}=G_{N,k}$ and $\delta_N+N^{-1}=\Delta_N$.

The constants $C,C_T$ below may change from line to line.  They depend
only on $d,s,q,\sigma,T$, the fixed interaction, the bounds on $f_0,f$,
$A_0,\beta_G$, and, on $\R^d$, $M_G$ and the spatial exponents
$M_*,\alpha$ chosen below.  They are independent of $N,k$, the
approximation parameters, and $\delta_N$, unless a dependence on $N$ is
explicitly indicated.
Named constants such as $A_T,B_T,D_T,\beta_T,\lambda_*$ have the same
dependence.  We treat both geometries with the spatial weights specified
in Section~\ref{subsec:comparison-weights}.  Let
$\phi_\varepsilon,K_\varepsilon$ be the regularizations of
\Cref{lem:common-riesz-regularization}, with
$0<\varepsilon<\varepsilon_0$ as chosen there.

For the regularized interaction, $e_{N,k}^\varepsilon$ denotes
the energy \eqref{eq:intro-particle-energy} with $\phi$ replaced by
$\phi_\varepsilon$.
For $1\le k\le N$, we set
\begin{align}
 &L_{N,k}
 :=\sum_{i=1}^k v_i\cdot\nabla_{x_i}
 +\frac1N\sum_{\substack{1\le i,j\le k\\i\ne j}}K(x_i-x_j)\cdot\nabla_{v_i},
 \label{eq:common-internal-generator}\\
 &J_{i,k}[h](Z_k)
 :=\int_{\mathcal D\times\R^d}
 K(x_i-y)h(Z_k,y,w)\dd y\dd w.
 \label{eq:common-flux-definition}
\end{align}
The regularized operators $L_{N,k}^\varepsilon$ and
$J_{i,k}^\varepsilon$ are defined by
\eqref{eq:common-internal-generator}--\eqref{eq:common-flux-definition}
with $K$ replaced by $K_\varepsilon$.

Let $G_N^\varepsilon(t)$ be the regularized particle law with initial
density $G_N^0$, and denote its marginals by $G_{N,k}^\varepsilon$.
They satisfy the BBGKY hierarchy
\begin{equation}\label{eq:intro-BBGKY}
 \partial_tG_{N,k}^\varepsilon+L_{N,k}^\varepsilon G_{N,k}^\varepsilon
 -\sigma\Delta_{V_k}G_{N,k}^\varepsilon
 +\frac{N-k}{N}\sum_{i=1}^k\nabla_{v_i}\cdot
        J_{i,k}^\varepsilon[G_{N,k+1}^\varepsilon]=0,
 \qquad k<N,
\end{equation}
with the last term absent at $k=N$.
By \Cref{prop:common-singular-hierarchy}, the singular hierarchy is
obtained by deleting every superscript $\varepsilon$ in
\eqref{eq:intro-BBGKY}.

The energy calculation is first performed for smooth initial data and
$K_\varepsilon$.  In \Cref{lem:TV-fixed-N-passage}, the limits are taken
in the order
\begin{equation}\label{eq:comparison-order-of-limits}
 n\to\infty\quad\text{at fixed }N,\varepsilon,
 \qquad\varepsilon\downarrow0\quad\text{at fixed }N.
\end{equation}
The resulting coefficients $A,L,D_T$ in
\eqref{eq:TV-error-energy} are independent of $N,k$.

Since $\phi$ is even, nonnegative and $K=-\nabla\phi$, we have
\begin{equation}\label{eq:common-energy-identities}
 \begin{aligned}
 L_{N,k}e_{N,k}=0,\qquad
 e_{N,k+1} \ge e_{N,k}+1+|v_{k+1}|^2.
 \end{aligned}
\end{equation}
Direct differentiation gives
\begin{equation}\label{eq:comparison-energy-derivatives}
 \nabla_{x_i}e_{N,k}=-\frac2N\sum_{j\ne i}K(x_i-x_j),
 \qquad \nabla_{v_i}e_{N,k}=2v_i,
\end{equation}
and hence
\[
 L_{N,k}e_{N,k}
 =-\frac2N\sum_{i\ne j}v_i\cdot K(x_i-x_j)
   +\frac2N\sum_{i\ne j}K(x_i-x_j)\cdot v_i=0.
\]
The same identities hold for $\phi_\varepsilon$ and
$e_{N,k}^\varepsilon$.

\subsection{Pointwise bounds and weights}\label{subsec:comparison-weights}

The following supersolution propagates the initial energy bound.

\begin{lemma}[Propagation of pointwise energy bounds]\label{lem:common-exponential-envelope}
Let $1\le k\le N$, and let $b_i\in C([0,T];C_b^1(\mathcal D;\R^d))$
with $\max_i\sup_{t\le T}\|b_i(t)\|_\infty\le L$.
Let $u_k^\varepsilon\in C([0,T];L^1)$ be the forward solution of
\[
 \partial_tu_k^\varepsilon+L_{N,k}^\varepsilon u_k^\varepsilon
 +\sum_i b_i(t,x_i)\cdot\nabla_{v_i}u_k^\varepsilon
 =\sigma\Delta_{V_k}u_k^\varepsilon,
\]
with a nonnegative integrable initial datum.  Assume that, for some
$A_1,\beta_2>0$, this datum satisfies
\[
 u_k^\varepsilon(0,Z_k)\le A_1^k
 e^{-\beta_2e_{N,k}^\varepsilon}\prod_{i=1}^kw_m(x_i),
 \qquad
 w_m(x)=
 \begin{cases}
 \langle x\rangle^{-m},&\mathcal D=\R^d,\quad m>d,\\
 1,&\mathcal D=\T^d.
 \end{cases}
\]
There exist $\beta_T>0$ and $A_T\ge1$, depending only on
$d,\sigma,T,L,A_1,\beta_2$ and, on $\R^d$, $m$, such that
\[
 u_k^\varepsilon(t,Z_k)\le A_T^k
 e^{-\beta_Te_{N,k}^\varepsilon}\prod_{i=1}^kw_m(x_i),
 \qquad t\le T.
\]
These constants are independent of $N,k$ and $\varepsilon$.
\end{lemma}

\begin{proof}
Put $m=0$ in the torus case and set
\[
 \overline u_k^\varepsilon=A(t)^k e^{-\beta(t)e_{N,k}^\varepsilon}
 \prod_iw_m(x_i),\qquad
 -\beta'=(4\sigma+2)\beta^2,\qquad
 \frac{A'}A=\frac{(m/2+2\beta L)^2}{8\beta^2},
\]
with $\beta(0)=\beta_2$ and $A(0)=A_1$.
For $\mathcal P_k^\varepsilon
=\partial_t+L_{N,k}^\varepsilon+\sum_i b_i\cdot\nabla_{v_i}
-\sigma\Delta_{V_k}$, the energy identity gives
\begin{align*}
 \frac{\mathcal P_k^\varepsilon\overline u_k^\varepsilon}
 {\overline u_k^\varepsilon}
 =k\frac{A'}A-\beta'e_{N,k}^\varepsilon
 -m\sum_i v_i\cdot\frac{x_i}{1+|x_i|^2}
 -2\beta\sum_i b_i(t,x_i)\cdot v_i+2d\sigma\beta k-4\sigma\beta^2\sum_i|v_i|^2,
\end{align*}
where the term containing $m$ is absent on $\T^d$.
Since $e_{N,k}^\varepsilon\ge k+\sum_i|v_i|^2$ and
$|x|/(1+|x|^2)\le1/2$, the preceding identity yields
\begin{align*}
 \frac{\mathcal P_k^\varepsilon\overline u_k^\varepsilon}
 {\overline u_k^\varepsilon}
 &\ge k\left(\frac{A'}A+(4\sigma+2)\beta^2+2d\sigma\beta\right)
   +\sum_i\left(2\beta^2|v_i|^2-(m/2+2\beta L)|v_i|\right)\\
 &\ge k\left(\frac{A'}A+(4\sigma+2)\beta^2+2d\sigma\beta
       -\frac{(m/2+2\beta L)^2}{8\beta^2}\right)\\
 &=k\left((4\sigma+2)\beta^2+2d\sigma\beta\right)\ge0.
\end{align*}
For fixed $N,k,\varepsilon$, the Gaussian factor and $m>d$ on $\R^d$
give
$\overline u_k^\varepsilon\in C_tL^1$ and
$\mathcal P_k^\varepsilon\overline u_k^\varepsilon,
\nabla_{V_k}\overline u_k^\varepsilon\in L_t^1L^1$.
The comparison estimate \eqref{eq:common-L1-order}, applied to the
nonnegative source $\mathcal P_k^\varepsilon\overline u_k^\varepsilon$,
therefore yields $u_k^\varepsilon\le\overline u_k^\varepsilon$.
Finally, we take
\[
 \beta(t)=\frac{\beta_2}
 {1+(4\sigma+2)\beta_2t}>0,
 \qquad
 \beta_T=\beta(T),\qquad A_T=\max\{1,\max_{t\le T}A(t)\}.
\]
\end{proof}

For $k=1$, the energy is $1+|v|^2$.  With the prescribed field $E_f$,
\Cref{lem:common-L1-stability} and
\eqref{eq:intro-velocity-derivative} identify $f$ with the forward
evolution from $f_0$.  Apply \Cref{lem:common-exponential-envelope} with
$b_1=E_f$, $\beta_2=a_0$, and $A_1=C_0e^{a_0}$.
The initial bound in \Cref{ass:limiting-solution} gives constants
$a_f>0$ and $C_f<\infty$ such that
\begin{equation}\label{eq:common-derived-f-envelope}
 f(t,x,v)\le C_fe^{-a_f|v|^2}
 \begin{cases}
  \langle x\rangle^{-M},&\mathcal D=\R^d,\\
  1,&\mathcal D=\T^d,
 \end{cases}
 \qquad 0\le t\le T.
\end{equation}
Here $M$ is the initial spatial exponent, and the constants depend only
on the data in \Cref{ass:limiting-solution}.

On $\T^d$ no spatial weight is needed.  On $\R^d$, choose
\begin{equation}\label{eq:TV-alpha-choice}
 d<M_*<\min\{M,M_G\},\qquad d(q-1)<\alpha<(q-1)M_*,
\end{equation}
and define the spatial weights
\[
 \varpi(x)=\langle x\rangle^\alpha,
 \qquad
 \varpi_k(X_k)=\prod_{i=1}^k\langle x_i\rangle^\alpha.
\]
On $\T^d$ we set $\varpi_k\equiv1$ and omit all spatial decay factors.
The choice \eqref{eq:TV-alpha-choice} gives
\begin{equation}\label{eq:common-weight-compatibility}
 \langle x\rangle^{\alpha-(q-1)M_*}\le1,
 \qquad
 \int_{\R^d}\langle x\rangle^{\alpha-qM_*}\dd x<\infty,
 \qquad
 \int_{\R^d}\langle x\rangle^{-\alpha/(q-1)}\dd x<\infty.
\end{equation}
On $\T^d$, the corresponding integrals equal one.

For a positive spatial weight $\varpi$, we define
\begin{equation}\label{eq:common-kernel-norm}
 \|K\|_{q,\varpi}
 :=\left(\sup_{x\in\mathcal D}\int_{\mathcal D}|K(x-y)|^{q'}
          \varpi(y)^{-1/(q-1)}\dd y\right)^{1/q'}.
\end{equation}
Here $\varpi\equiv1$ on $\T^d$ and $\varpi(x)=\langle x\rangle^\alpha$ on
$\R^d$.  The kernel bounds used below are
\begin{equation}\label{eq:kernel-assumptions}
 \sup_{0<\varepsilon<\varepsilon_0}\|K_\varepsilon\|_{q,\varpi}<\infty,
 \qquad
 \|K_\varepsilon-K\|_{q,\varpi}\longrightarrow0,
 \qquad
 \|K_\varepsilon-K\|_{L^1(\mathcal D)}
 \le C\varepsilon^{d-s-1}.
\end{equation}
By \Cref{lem:common-riesz-regularization},
$(s+1)q'<d$ and, on $\R^d$, $\alpha>d(q-1)$ imply
\eqref{eq:kernel-assumptions}.
The bound \eqref{eq:common-derived-f-envelope} and conservation of mass give
\[
 \sup_{t\le T}\|\rho_f(t)\|_\infty<\infty,
 \qquad \|\rho_f(t)\|_1=1.
\]
Splitting the convolution at the radius $r_0$ of
\Cref{lem:common-riesz-regularization} gives
\[
 \begin{aligned}
 &\||K|*\rho_f(t)\|_\infty
   +\sup_{0<\varepsilon<\varepsilon_0}
       \||K_\varepsilon|*\rho_f(t)\|_\infty\\
 \le & C\|\rho_f(t)\|_\infty
       \int_0^{r_0}r^{d-s-2}\dd r+C\|\rho_f(t)\|_1
 \le C\bigl(1+\|\rho_f(t)\|_\infty\bigr).
 \end{aligned}
\]
We conclude that
\begin{equation}\label{eq:common-force-convolution-bound}
 \sup_{t\le T}\left(
 \||K|*\rho_f(t)\|_\infty+\|K*\rho_f(t)\|_\infty
 +\sup_{0<\varepsilon<\varepsilon_0}
      \||K_\varepsilon|*\rho_f(t)\|_\infty
 \right)<\infty.
\end{equation}

Choose $C_1\ge1$, independent of $N,k,\varepsilon$, as in
\Cref{lem:common-riesz-regularization}, and fix
\begin{equation}\label{eq:comparison-exponent-choices}
 \begin{gathered}
 e_{N,k}^\varepsilon\le C_1e_{N,k},\qquad
 \beta_1=\frac{\beta_G}{2C_1},\qquad
 \beta_0=\frac{\beta_1}{2},\qquad
 \beta_T=\frac{\beta_0}{1+(4\sigma+2)\beta_0T},\\
 \beta_*=\frac12\min\{\beta_T,a_f\},\qquad
 \lambda_* =\frac{(q-1)\beta_*}{2},\qquad
 \lambda_0=\frac18\min\{1,\lambda_*\}.
 \end{gathered}
\end{equation}

The pointwise hypothesis \eqref{eq:intro-correlated-periodic-envelope}
or \eqref{eq:intro-correlated-whole-envelope}, in the notation above, gives
\begin{equation}\label{eq:comparison-initial-regularized-envelope}
 G_{N,k}^0\le A_0^ke^{-\beta_1e_{N,k}^\varepsilon}
                  \prod_{i=1}^k\langle x_i\rangle^{-M_*},
\end{equation}
with the spatial factors omitted on $\T^d$.
The smooth approximations of \Cref{lem:common-smooth-initial-data}
satisfy this bound with exponent $\beta_0$ and a larger constant.
The supersolution then gives exponent $\beta_T$ on $[0,T]$.
All parameters in \eqref{eq:comparison-exponent-choices} are independent
of $N,k,n,\varepsilon$.

\subsection{Weighted estimate for the BBGKY interaction term}\label{subsec:common-weighted-tools}

We estimate the flux operator \eqref{eq:common-flux-definition}.

\begin{lemma}\label{lem:common-weighted-holder}
Let $\mathcal D$ be either $\T^d$ or $\R^d$, let $1<q<\infty$, fix $\lambda>0$, and let
$\varpi:\mathcal D\to(0,\infty)$ be measurable.  Suppose that
$\|K\|_{q,\varpi}<\infty$, with the norm defined in
\eqref{eq:common-kernel-norm}.
Let $\mathcal E_k(Z_k)$ and $\widehat{\mathcal E}_{k+1}(Z_{k+1})$
be measurable functions satisfying
$
 \widehat{\mathcal E}_{k+1}\ge
 \mathcal E_k+1+|v_{k+1}|^2,
$
and set $\varpi_k=\prod_{j=1}^k\varpi(x_j)$.  If $h$ is measurable and the
right-hand side of \eqref{eq:common-weighted-holder} is finite, then the
integral defining $J_{i,k}[h]$ in \eqref{eq:common-flux-definition} is
absolutely convergent for almost every $Z_k$ and satisfies
\begin{align}\label{eq:common-weighted-holder}
 \int |J_{i,k}[h]|^q e^{\lambda\mathcal E_k}\varpi_k\dd Z_k
 \le \Gamma_{d,q}(\lambda)\|K\|_{q,\varpi}^q
 \int |h|^q e^{\lambda\widehat{\mathcal E}_{k+1}}
       \varpi_{k+1}\dd Z_{k+1},
\end{align}
where
\begin{equation}\label{eq:common-weighted-holder-gamma}
 \Gamma_{d,q}(\lambda)
 :=e^{-\lambda}\left(\frac{\pi(q-1)}{\lambda}\right)^{d(q-1)/2}
 \le C_{d,q}\lambda^{-d(q-1)/2}.
\end{equation}
The constant is independent of the hierarchy level $k$ and, in the
applications below, of $N$.  Replacing $K$ by $K_\varepsilon$ gives
\eqref{eq:common-weighted-holder} with $\|K_\varepsilon\|_{q,\varpi}$
in place of $\|K\|_{q,\varpi}$; in particular the operator constants are
uniform when $\sup_\varepsilon\|K_\varepsilon\|_{q,\varpi}<\infty$.
\end{lemma}

\begin{proof}
Fix $Z_k$ and write $z_{k+1}=(y,w)$.  Weighted H\"older's inequality gives
\begin{align*}
 |J_{i,k}[h](Z_k)|^q
 \le \left(\int |h|^qe^{\lambda\widehat{\mathcal E}_{k+1}}
                  \varpi(y)\dd y\dd w\right)\left(\int |K(x_i-y)|^{q'}
 e^{-\lambda\widehat{\mathcal E}_{k+1}/(q-1)}
 \varpi(y)^{-1/(q-1)}\dd y\dd w\right)^{q-1}.
\end{align*}
Since $\widehat{\mathcal E}_{k+1}\ge\mathcal E_k+1+|w|^2$, the second
factor is at most
\[
 e^{-\lambda\mathcal E_k}\|K\|_{q,\varpi}^q
 \left(\int_{\R^d}e^{-\lambda(1+|w|^2)/(q-1)}\dd w\right)^{q-1}
 =e^{-\lambda\mathcal E_k}\|K\|_{q,\varpi}^q\Gamma_{d,q}(\lambda).
\]
The first factor is finite for almost every $Z_k$ by the assumed
integrability.  The same H\"older estimate applied to $|K||h|$ therefore
proves absolute convergence.  Multiplying by
$e^{\lambda\mathcal E_k}\varpi_k$ and integrating in $Z_k$ proves
\eqref{eq:common-weighted-holder}.  The argument applies to
$K_\varepsilon$ with its corresponding kernel norm.
\end{proof}

\subsection{Comparison with the limiting equation}

Let
\[
 E=K*\rho_f,\qquad E^\varepsilon=K_\varepsilon*\rho_f.
\]
For each $1\le k\le N$, let $U_{N,k}$ be the forward density of the
$k$-particle system with internal force $K$ and prescribed field
$((N-k)/N)E$.  Its existence follows from
\Cref{prop:common-finite-N,prop:common-regularization-removal}.  It solves
\begin{equation}\label{eq:TV-U-PDE}
 \partial_tU_{N,k}+L_{N,k}U_{N,k}
 +\frac{N-k}{N}\sum_{i=1}^kE(t,x_i)\cdot\nabla_{v_i}U_{N,k}
 =\sigma\Delta_{V_k}U_{N,k},
 \qquad U_{N,k}(0)=G_{N,k}^0.
\end{equation}
At the regularized level, $L_{N,k}$ is replaced by
$L_{N,k}^\varepsilon$, while the prescribed field remains $E=K*\rho_f$.
The regularized particle marginals $G_{N,k}^\varepsilon$ are those in
\eqref{eq:intro-BBGKY}.

The densities $U_{N,k}$ are defined separately at each level and need
not form a consistent family of marginals.

\begin{lemma}\label{lem:TV-U-pointwise}
Let $U_{N,k}^\varepsilon$ solve the regularized version of
\eqref{eq:TV-U-PDE}, with initial datum $G_{N,k}^0$ and internal force
$K_\varepsilon$.  On $\R^d$, there exist constants $\beta_T>0$ and $A_T\ge1$,
independent of $N,k$ and $\varepsilon$, such that
\begin{align}
 U_{N,k}(t,Z_k)
 &\le A_T^ke^{-\beta_Te_{N,k}(Z_k)}
       \prod_{j=1}^{k}\langle x_j\rangle^{-M_*},
 \label{eq:TV-U-bound}\\
 U_{N,k}^\varepsilon(t,Z_k)
 &\le A_T^ke^{-\beta_Te_{N,k}^\varepsilon(Z_k)}
       \prod_{j=1}^{k}\langle x_j\rangle^{-M_*},
 \qquad t\in[0,T].
 \label{eq:TV-U-bound-regularized}
\end{align}
On $\T^d$ the products in $x$ are omitted.  The singular density
$U_{N,k}$ belongs to $C([0,T];L^1)$ and satisfies \eqref{eq:TV-U-PDE} in
distributions with initial trace $G_{N,k}^0$.  At level $N$, the auxiliary
and particle densities coincide:
\begin{equation}\label{eq:TV-level-N-equality}
 U_{N,N}=G_N,\qquad U_{N,N}^\varepsilon=G_N^\varepsilon,
\end{equation}
and on $\R^d$ there exists $B_T\ge1$, independent of $N,k,\varepsilon$,
such that
\begin{align}
 G_{N,k}^\varepsilon(t,Z_k)
 &\le B_T^Ne^{-\beta_Te_{N,k}^\varepsilon(Z_k)}
       \prod_{j=1}^{k}\langle x_j\rangle^{-M_*},
 \label{eq:TV-GN-bound-regularized}\\
 G_{N,k}(t,Z_k)
 &\le B_T^Ne^{-\beta_Te_{N,k}(Z_k)}
       \prod_{j=1}^{k}\langle x_j\rangle^{-M_*}.
 \label{eq:TV-GN-bound}
\end{align}
Again, the products in $x$ are absent on $\T^d$.
\end{lemma}

The bound with $B_T^N$ is used for the approximation at fixed
$N$.  The auxiliary bound with $A_T^k$ supplies the uniform estimates
needed in the hierarchy comparison.

\begin{proof}
Apply \Cref{lem:common-exponential-envelope} with
$b_i=((N-k)/N)E$, $m=M_*$, and the initial bound
\eqref{eq:comparison-initial-regularized-envelope}, using the smaller
exponent $\beta_0$.  This gives \eqref{eq:TV-U-bound-regularized} with
constants uniform in $N,k,\varepsilon$.
At $k=N$, the auxiliary and particle equations coincide, which proves
\eqref{eq:TV-level-N-equality} for the regularized laws.  Integrating the
pointwise bound at this level and using
$e_{N,N}^\varepsilon\ge e_{N,k}^\varepsilon+
\sum_{i=k+1}^N(1+|v_i|^2)$ gives
\eqref{eq:TV-GN-bound-regularized}; the one-particle Gaussian and spatial
integrals are finite since $M_*>d$.

The singular auxiliary dynamics are given by
\Cref{prop:common-finite-N} with prescribed field $((N-k)/N)E$.
By \Cref{prop:common-regularization-removal},
$U_{N,k}^\varepsilon\to U_{N,k}$ in $C([0,T];L^1)$.
At any fixed $t\in[0,T]$, extract an almost everywhere convergent
subsequence from this $L^1$ convergence.  Since
$e_{N,k}^\varepsilon\to e_{N,k}$ off the collision diagonals, the
regularized bound passes to \eqref{eq:TV-U-bound} at that time.  Its
constants are independent of $t$.  At level $N$, the singular particle
systems also coincide, giving $U_{N,N}=G_N$; integration then proves
\eqref{eq:TV-GN-bound}.

For $0<\lambda<q\beta_T$, the weighted densities satisfy
\[
 |U_{N,k}^\varepsilon|^qe^{\lambda e_{N,k}^\varepsilon}\varpi_k
 \le A_T^{qk}e^{-(q\beta_T-\lambda)\sum_i(1+|v_i|^2)}
 \prod_i\langle x_i\rangle^{\alpha-qM_*}.
\]
This is integrable by \eqref{eq:common-weight-compatibility}.
Consequently \Cref{prop:common-singular-hierarchy} identifies the
limiting auxiliary equation and its initial trace.
\end{proof}

For the particle marginals, the singular BBGKY hierarchy is obtained by
letting $\varepsilon\downarrow0$ at fixed $N$.  The bound
\eqref{eq:TV-GN-bound-regularized}, together with the
$C([0,T];L^1)$ convergence from
\Cref{prop:common-regularization-removal}, verifies the hypotheses of
\Cref{prop:common-singular-hierarchy} and yields the BBGKY hierarchy
with the singular kernel $K$.

We next compare the auxiliary densities with $f_t^{\otimes k}$.

\begin{lemma}\label{lem:TV-U-product}
There exists $C_T\ge1$, independent of $N,k,\varepsilon$, such that
\begin{align}
 \sup_{t\le T}\|U_{N,k}(t)-f_t^{\otimes k}\|_1
 &\le a_{N,k}+\frac{C_Tk^{2-1/q}}N,
 \label{eq:TV-U-product}\\
 \sup_{t\le T}\|U_{N,k}^\varepsilon(t)-f_t^{\otimes k}\|_1
 &\le a_{N,k}+\frac{C_T}{N}
       \bigl(k^{2-1/q}+k^2\varepsilon^{d-s-1}\bigr).
 \label{eq:TV-U-product-regularized-sharp}
\end{align}
Both right-hand sides are bounded by $C_T^k\Delta_N$, after increasing
$C_T$.
\end{lemma}

\begin{proof}
Set $u_k=f_t^{\otimes k}$ and
\[
 \mathcal P_{N,k}^\varepsilon
 =\partial_t+L_{N,k}^\varepsilon
  +\frac{N-k}{N}\sum_{i=1}^kE(t,x_i)\cdot\nabla_{v_i}
  -\sigma\Delta_{V_k}.
\]
Tensorizing \eqref{eq:intro-riesz-vpfp} gives
\begin{equation}\label{eq:auxiliary-product-source}
 \mathcal P_{N,k}^\varepsilon u_k=r_{N,k}^\varepsilon,
 \qquad
 r_{N,k}^\varepsilon
 =\frac1N\sum_{i=1}^k
  \left(\sum_{j\ne i}K_\varepsilon(x_i-x_j)-kE(t,x_i)\right)
                  \cdot\nabla_{v_i}u_k.
\end{equation}
Since $(s+1)q'<d$ and $1<q'\le2$, we have
\begin{equation}\label{eq:auxiliary-centered-force-moment}
 \sup_{t\le T,x,\varepsilon}
 \int|K_\varepsilon(x-y)-E^\varepsilon(t,x)|^{q'}
          \rho_f(t,y)\dd y
 \le C_T\left(1+\int_0^{r_0}r^{d-1-(s+1)q'}\dd r\right)<\infty.
\end{equation}
Here $r_0$ is as in \Cref{lem:common-riesz-regularization}.
Fix $t,x$.  For independent $Y_j$ of density $\rho_f(t)$, let
$\xi_j=K_\varepsilon(x-Y_j)-E^\varepsilon(t,x)$, let $\xi_j'$ be
independent copies, and let $\theta_j$ be independent symmetric random
signs.  The identity $\mathbb E \xi_j=0$ gives
\begin{align}\label{eq:auxiliary-centered-sum}
 \mathbb E\left|\sum_{j=1}^{k-1}\xi_j\right|^{q'}
 &\le\mathbb E\left|\sum_{j=1}^{k-1}(\xi_j-\xi_j')\right|^{q'}
 =\mathbb E\mathbb E_\theta
       \left|\sum_{j=1}^{k-1}\theta_j(\xi_j-\xi_j')\right|^{q'}\notag\\
 &\le\mathbb E\left(\sum_{j=1}^{k-1}|\xi_j-\xi_j'|^2\right)^{q'/2}
 \le\sum_{j=1}^{k-1}\mathbb E|\xi_j-\xi_j'|^{q'}
 \le C_T(k-1).
\end{align}
We used Jensen's inequality, the symmetry of $\xi_j-\xi_j'$, and
$q'/2\le1$.  Thus
$\mathbb E|\sum_{j=1}^{k-1}\xi_j|\le C_T(k-1)^{1/q'}$.

For each $i$, we decompose the force as
\begin{align*}
 \sum_{j\ne i}K_\varepsilon(x_i-x_j)-kE(t,x_i)
 =\sum_{j\ne i}\bigl(K_\varepsilon(x_i-x_j)
                         -E^\varepsilon(t,x_i)\bigr)
 -E(t,x_i)+(k-1)(E^\varepsilon-E)(t,x_i).
\end{align*}
Since
$\nabla_{v_i}u_k=(\nabla_vf)(z_i)\prod_{j\ne i}f(z_j)$,
\eqref{eq:auxiliary-centered-sum} gives
\begin{align*}
 \|r_{N,k}^\varepsilon(t)\|_1
 &\le\frac1N\sum_{i=1}^k\int|\nabla_vf(t,z_i)|\biggl[
 \mathbb E\left|\sum_{j=1}^{k-1}
    \bigl(K_\varepsilon(x_i-Y_j)-E^\varepsilon(t,x_i)\bigr)\right|\\
 &\hspace{44mm}+\|E(t)\|_\infty
       +(k-1)\|E^\varepsilon(t)-E(t)\|_\infty\biggr]\dd z_i\\
 &\le\frac{k}{N}\left[C_T(k-1)^{1/q'}+\|E(t)\|_\infty
       +(k-1)\|E^\varepsilon(t)-E(t)\|_\infty\right]
       \|\nabla_vf(t)\|_1.
\end{align*}
We therefore obtain
\begin{equation}\label{eq:auxiliary-source-L1}
 \|r_{N,k}^\varepsilon(t)\|_1
 \le\frac{C_T}{N}
       \bigl(k^{2-1/q}+k^2\varepsilon^{d-s-1}\bigr)
       \|\nabla_vf(t)\|_1.
\end{equation}

The last line uses
\begin{equation}\label{eq:auxiliary-field-defect}
 1+\frac1{q'}=2-\frac1q,
 \qquad
 \|E^\varepsilon-E\|_\infty
 \le\|K_\varepsilon-K\|_1\|\rho_f\|_\infty
 \le C_T\varepsilon^{d-s-1}.
\end{equation}
These formulas also hold for $k=1$, with an empty centered sum.
The product density also satisfies
\[
 u_k\in C_tL^1,\qquad
 \int_0^T\|\nabla_{V_k}u_k(t)\|_1\dd t
 \le k\int_0^T\|\nabla_v f(t)\|_1\dd t<\infty.
\]
Applying \Cref{lem:common-L1-stability} to
\eqref{eq:auxiliary-product-source}, integrating
\eqref{eq:auxiliary-source-L1} in time, and using
\eqref{eq:intro-velocity-derivative} gives
\eqref{eq:TV-U-product-regularized-sharp}.
The convergence $U_{N,k}^\varepsilon\to U_{N,k}$ in $C_tL^1$ from
\Cref{prop:common-regularization-removal} proves
\eqref{eq:TV-U-product}.  Finally,
$a_{N,k}\le A_0^k\delta_N$ and $k^2\le4^k$ give the bound
$C_T^k\Delta_N$.
\end{proof}

\subsection{Estimate of the remainder flux}

For $k<N$, we define
\begin{equation}\label{eq:TV-remainder-weight}
 e_{N,k+1}^*=e_{N,k}+1+|v_{k+1}|^2,
 \qquad
 e_{N,k+1}^{\varepsilon,*}=e_{N,k}^{\varepsilon}+1+|v_{k+1}|^2.
\end{equation}
The identities
\[
 J_{i,k}[U_{N,k}f_t(z_{k+1})]=E(t,x_i)U_{N,k},
 \qquad e_{N,k+1}\ge e_{N,k+1}^*
\]
and the bounds in \Cref{lem:TV-U-pointwise} and
\eqref{eq:common-derived-f-envelope} give a common exponential weight for $U_{N,k+1}$ and
$U_{N,k}f_t(z_{k+1})$.

\begin{proposition}\label{prop:TV-small-residual}
For $1\le k<N$, set
\begin{align*}
 h_{N,k+1}=U_{N,k+1}-U_{N,k}f_t(z_{k+1}),\qquad
 R_{i,N,k}=J_{i,k}[h_{N,k+1}],
\end{align*}
and
\begin{align*}
 h_{N,k+1}^\varepsilon
 &=U_{N,k+1}^\varepsilon-U_{N,k}^\varepsilon f_t(z_{k+1}),\\
 R_{i,N,k}^\varepsilon
 &=J_{i,k}^\varepsilon[h_{N,k+1}^\varepsilon]
 +(E^\varepsilon-E)(t,x_i)U_{N,k}^\varepsilon\\
 &=J_{i,k}^\varepsilon[U_{N,k+1}^\varepsilon]
   -E(t,x_i)U_{N,k}^\varepsilon.
\end{align*}
Let $\lambda_*$ be fixed by \eqref{eq:comparison-exponent-choices}.
There exists $D_T\ge1$, depending only on the fixed data specified above
and independent of $N,k$ and of the force regularization, for which the
following estimates hold whenever $0<\lambda\le\lambda_*$ and $0\le t\le T$:
\begin{align}
 &\int|h_{N,k+1}(t)|^qe^{\lambda e_{N,k+1}^*}\varpi_{k+1}\dd Z_{k+1}
 \le D_T^k\Delta_N,
 \label{eq:TV-h-weighted}\\
& \sum_{i=1}^k
 \|R_{i,N,k}(t)\|_{L^q(e^{\lambda e_{N,k}}\varpi_k)}^q
 \le D_T^k\Delta_N,
 \label{eq:TV-residual-weighted}\\
 & \int|h_{N,k+1}^\varepsilon(t)|^q
 e^{\lambda e_{N,k+1}^{\varepsilon,*}}\varpi_{k+1}\dd Z_{k+1}
\le D_T^k\Delta_N,
 \label{eq:TV-h-weighted-regularized}\\
 &\sum_{i=1}^k
 \|R_{i,N,k}^\varepsilon(t)\|_{L^q(e^{\lambda e_{N,k}^\varepsilon}\varpi_k)}^q
 \le D_T^k\bigl(\Delta_N+\varepsilon^{q(d-s-1)}\bigr).
 \label{eq:TV-regularized-residual-bound}
\end{align}
\end{proposition}

\begin{proof}
The triangle inequality and \Cref{lem:TV-U-product} give
\begin{align*}
 \|h_{N,k+1}(t)\|_1
 &\le\|U_{N,k+1}(t)-f_t^{\otimes(k+1)}\|_1
       +\|U_{N,k}(t)-f_t^{\otimes k}\|_1\\
 &\le a_{N,k+1}+a_{N,k}
       +\frac{C_T(k+1)^{2-1/q}}N.
\end{align*}
The regularized estimate, the marginal hypothesis
\eqref{eq:intro-correlated-marginals}, and the notation
\eqref{eq:comparison-initial-error} give
\begin{equation}\label{eq:TV-h-L1-explicit}
 \|h_{N,k+1}(t)\|_1
 +\|h_{N,k+1}^\varepsilon(t)\|_1
 \le C_T^k\Delta_N.
\end{equation}
The bounds \eqref{eq:TV-U-bound}, \eqref{eq:common-derived-f-envelope} and
the relation $e_{N,k+1}\ge e_{N,k+1}^*$ imply
\begin{equation}\label{eq:TV-h-pointwise-bound}
 \begin{aligned}
 |h_{N,k+1}|
 \le U_{N,k+1}+U_{N,k}f_t(z_{k+1})\le C_T^{k+1}e^{-\beta_*e_{N,k+1}^*}
       \prod_{j=1}^{k+1}\langle x_j\rangle^{-M_*}.
 \end{aligned}
\end{equation}
The same bound holds with superscript $\varepsilon$.
For $0<\lambda\le\lambda_*$, this gives
\[
 \begin{aligned}
 |h_{N,k+1}|^{q-1}e^{\lambda e_{N,k+1}^*}\varpi_{k+1}
 \le C_T^{(q-1)(k+1)}
 e^{-((q-1)\beta_*-\lambda)e_{N,k+1}^*}
 \prod_j\langle x_j\rangle^{\alpha-(q-1)M_*}\le C_T^{(q-1)(k+1)}.
 \end{aligned}
\]
Multiplication by $|h_{N,k+1}|$ and integration give
\begin{align}\label{eq:TV-h-interpolation-explicit}
 \int|h_{N,k+1}|^qe^{\lambda e_{N,k+1}^*}\varpi_{k+1}\dd Z_{k+1}
 \le\bigl\||h_{N,k+1}|^{q-1}e^{\lambda e_{N,k+1}^*}
              \varpi_{k+1}\bigr\|_\infty\|h_{N,k+1}\|_1\le C_T^k\Delta_N.
\end{align}
This also proves \eqref{eq:TV-h-weighted-regularized}.
For $0<\lambda\le\lambda_*$, we use $e_{N,k}\ge0$ and
\Cref{lem:common-weighted-holder} at $\lambda_*$ to obtain
\[
 \begin{aligned}
 \sum_{i=1}^k\|J_{i,k}[h_{N,k+1}]\|_
        {L^q(e^{\lambda e_{N,k}}\varpi_k)}^q
 &\le\sum_{i=1}^k\|J_{i,k}[h_{N,k+1}]\|_
        {L^q(e^{\lambda_*e_{N,k}}\varpi_k)}^q\\
 &\le k\Gamma_{d,q}(\lambda_*)\|K\|_{q,\varpi}^qC_T^k\Delta_N
 \le D_T^k\Delta_N,
 \end{aligned}
\]
after increasing $D_T$.  By \eqref{eq:kernel-assumptions}, the same
estimate holds for
$J_{i,k}^\varepsilon[h_{N,k+1}^\varepsilon]$.

Finally, $\lambda_*<(q-1)\beta_T$,
$\alpha<(q-1)M_*$ and $\|U_{N,k}^\varepsilon(t)\|_1=1$ imply
\[
 \begin{aligned}
 \int|U_{N,k}^\varepsilon|^q
          e^{\lambda e_{N,k}^\varepsilon}\varpi_k
\le\bigl\||U_{N,k}^\varepsilon|^{q-1}
          e^{\lambda e_{N,k}^\varepsilon}\varpi_k\bigr\|_\infty
          \|U_{N,k}^\varepsilon\|_1\le A_T^{(q-1)k}.
 \end{aligned}
\]
Together with \eqref{eq:auxiliary-field-defect}, this gives
\[
 \begin{aligned}
 \sum_{i=1}^k
 \bigl\|(E^\varepsilon-E)(t,x_i)U_{N,k}^\varepsilon\bigr\|_
       {L^q(e^{\lambda e_{N,k}^\varepsilon}\varpi_k)}^q
 \le k\|E^\varepsilon-E\|_\infty^q A_T^{(q-1)k}\le D_T^k\varepsilon^{q(d-s-1)}.
 \end{aligned}
\]
Combining the two flux estimates with
$|a+b|^q\le2^{q-1}(|a|^q+|b|^q)$ proves
\eqref{eq:TV-regularized-residual-bound}.
\end{proof}

\subsection{Energy inequality for the difference}

We set
\begin{equation}\label{eq:TV-error-definition}
 H_{N,k}=G_{N,k}-U_{N,k}.
\end{equation}
For $k<N$, subtract \eqref{eq:TV-U-PDE} from the singular version of
\eqref{eq:intro-BBGKY}.  The definition of $R_{i,N,k}$ gives
\begin{align}\label{eq:TV-error-PDE}
 \partial_tH_{N,k}+L_{N,k}H_{N,k}-\sigma\Delta_{V_k}H_{N,k}
 +\frac{N-k}{N}\sum_i\nabla_{v_i}\cdot J_{i,k}[H_{N,k+1}]
 =-\frac{N-k}{N}\sum_i\nabla_{v_i}\cdot R_{i,N,k}.
\end{align}
The initial data agree at every level, and the evolution equations agree
when $k=N$.  It follows that
\begin{equation}\label{eq:TV-error-identities}
 H_{N,k}(0)=0\quad(1\le k\le N),\qquad H_{N,N}(t)=0.
\end{equation}

With $\lambda_0$ fixed by \eqref{eq:comparison-exponent-choices}, choose
$\Lambda\ge1$ sufficiently large, depending only on $(q,\sigma)$, and set
\begin{equation}\label{eq:TV-global-lambda}
 \lambda(t)=\frac{\lambda_0}{1+\Lambda\lambda_0t},\qquad
 \lambda'=-\Lambda\lambda^2.
\end{equation}
Define
\begin{equation}\label{eq:comparison-weighted-error}
 Y_{N,k}(t)=\int |H_{N,k}(t)|^q
 e^{\lambda(t)e_{N,k}}\varpi_k\dd Z_k.
\end{equation}
For fixed $N,k$, the bounds \eqref{eq:TV-U-bound} and
\eqref{eq:TV-GN-bound}, together with $4\lambda_0<q\beta_T$, give an
integrable majorant for $|H_{N,k}|^q e^{\lambda(t)e_{N,k}}\varpi_k$ on
$[0,T]$.  The iteration in Section~\ref{subsec:common-volterra} will give
$Y_{N,k}\le C^kt\Delta_N$ on an $N$-independent interval.

We first work with the regularized energy.  The cancellation
\eqref{eq:common-energy-identities} removes the internal-force contribution
to the weight, and $\lambda'=-\Lambda\lambda^2$ absorbs the velocity
quadratic terms.

\begin{lemma}\label{lem:common-weighted-flux-energy}
Let $2\le q<\infty$ and let $\lambda$ be given by
\eqref{eq:TV-global-lambda} with $0<\lambda_0\le1$.  Suppose that $H$ and
$\mathcal F_i$ are smooth and satisfy
\[
 \partial_tH+L_{N,k}^\varepsilon H-\sigma\Delta_{V_k}H
 +\sum_{i=1}^k\nabla_{v_i}\cdot\mathcal F_i=0,
\]
and set
\[
 W=e^{\lambda e_{N,k}^\varepsilon}\varpi_k,\qquad
 Y_k(t)=\int|H|^qW,\qquad
 \mathcal Q_k(t)=\sum_{i=1}^k\int|\mathcal F_i|^qW.
\]
Assume that $Y_k,\mathcal Q_k\in L^1(0,T)$ and
\[
 \int_0^T\!\int(1+|V_k|^2)|H|^qW\dd Z_k\dd t<\infty.
\]
If $\Lambda$ is sufficiently large, depending only on $(q,\sigma)$,
then the following inequality holds in the sense of distributions on
$(0,T)$:
\begin{equation}\label{eq:common-flux-energy}
 Y_k'(t)\le
 C_{d,q,\sigma}\bigl(1+\mathbf1_{\{\mathcal D=\R^d\}}\alpha^2\lambda(t)^{-2}\bigr)
 kY_k(t)
 +C_{d,q,\sigma}\lambda(t)^{-(q-2)}
 \sum_{i=1}^k\int|\mathcal F_i|^qW.
\end{equation}
For the smooth regularized approximations used in
Appendix~\ref{app:fixed-N-passage}, the corresponding integrated inequality
holds for every $t\in[0,T]$, with initial value $Y_k(0)$; see
Section~\ref{subsec:common-cutoff-testing}.
\end{lemma}

\begin{proof}
We use the weighted $L^q$ calculation of
\cite[Section~3.2]{BJS2025}.  Insert a compact phase-space cutoff and a
nonnegative test function in $C_c^\infty(0,T)$; the boundary terms are
estimated in Section~\ref{subsec:common-cutoff-testing}.  Write
$e=e_{N,k}^\varepsilon$ and set
\[
 \mathcal D_k=\sum_i\int |H|^{q-2}|\nabla_{v_i}H|^2W,
 \qquad
 \mathcal M_k=\sum_i\int |v_i|^2|H|^qW.
\]
The transport field in $L_{N,k}^\varepsilon$ is divergence free,
and the weight satisfies
\[
 \partial_tW=\lambda'eW,
 \qquad L_{N,k}^\varepsilon e=0,
 \qquad \nabla_{v_i}W=2\lambda v_iW,
\]
and
\[
 \Delta_{V_k}W
 =(2dk\lambda+4\lambda^2\textstyle\sum_i|v_i|^2)W.
\]
Integration by parts in the diffusion term gives
\[
 q\int |H|^{q-2}H\Delta_{V_k}H\,W
 =-q(q-1)\mathcal D_k+\int|H|^q\Delta_{V_k}W.
\]
For the flux term, integration by parts gives
\begin{align*}
 -q\int |H|^{q-2}H\,\nabla_{v_i}\cdot\mathcal F_i\,W
 =q(q-1)\int |H|^{q-2}\nabla_{v_i}H\cdot\mathcal F_iW+2q\lambda\int |H|^{q-2}H
 v_i\cdot\mathcal F_iW.
\end{align*}
Multiplying the equation by $q|H|^{q-2}H W$ and combining these identities,
we obtain
\begin{align}\label{eq:common-flux-energy-identity}
 Y_k'+q(q-1)\sigma\mathcal D_k
 &=\lambda'\int e|H|^qW
   +\int |H|^qW\frac{L_{N,k}^\varepsilon\varpi_k}{\varpi_k}
   +2d\sigma\lambda kY_k+4\sigma\lambda^2\mathcal M_k\notag\\
 &\quad+q(q-1)\sum_i\int |H|^{q-2}
       \nabla_{v_i}H\cdot\mathcal F_iW
   +2q\lambda\sum_i\int |H|^{q-2}H
       v_i\cdot\mathcal F_iW.
\end{align}
Young's inequality gives
\[
 q(q-1)\sum_i\int |H|^{q-2}|\nabla_{v_i}H|\,|\mathcal F_i|W
 \le\frac{q(q-1)\sigma}{4}\mathcal D_k
       +\frac{q(q-1)}\sigma
                  \sum_i\int |H|^{q-2}|\mathcal F_i|^2W
\]
and
\[
 2q\lambda |v_i|\,|H|^{q-1}|\mathcal F_i|
 \le\lambda^2|v_i|^2|H|^q
      +q^2|H|^{q-2}|\mathcal F_i|^2.
\]
For $q>2$, Young's inequality with exponents $q/(q-2)$ and $q/2$
yields
\begin{equation}\label{eq:common-flux-young-power}
 |H|^{q-2}|\mathcal F_i|^2
 \le\frac{q-2}{q}\lambda^2|H|^q
       +\frac2q\lambda^{-(q-2)}|\mathcal F_i|^q.
\end{equation}
For $q=2$, the left-hand side is $|\mathcal F_i|^2$ and
$\lambda^{-(q-2)}=1$.

Set $C_*=4\sigma+1+q(q-1)/\sigma+q^2$.
The two flux terms and $4\sigma\lambda^2\mathcal M_k$ are bounded by
\[
 \frac{q(q-1)\sigma}{4}\mathcal D_k
 +C_*\lambda^2(\mathcal M_k+kY_k)
 +C_*\lambda^{-(q-2)}\mathcal Q_k.
\]
Since $e\ge k+\sum_i|v_i|^2$ and
$\lambda'=-\Lambda\lambda^2$, we obtain
\[
 \lambda'\int e|H|^qW
 \le-\Lambda\lambda^2(kY_k+\mathcal M_k).
\]
Choose $\Lambda=2C_*$.  The spatial weight term vanishes on $\T^d$.
On $\R^d$, we use $\varpi(x)=\langle x\rangle^\alpha$ and
$|\nabla\log\varpi(x)|\le\alpha/2$ to obtain
\[
 \frac{|L_{N,k}^\varepsilon\varpi_k|}{\varpi_k}
 \le\frac\Lambda4\lambda^2\sum_i|v_i|^2
      +\frac{\alpha^2}{4\Lambda}\lambda^{-2}k.
\]
Substitution in \eqref{eq:common-flux-energy-identity} gives
\begin{equation}\label{eq:common-flux-energy-dissipation}
 \begin{aligned}
 Y_k'+\frac{3q(q-1)\sigma}{4}\mathcal D_k
 +\frac\Lambda4\lambda^2\mathcal M_k\le C_{d,q,\sigma}
 \bigl(1+\mathbf1_{\{\mathcal D=\R^d\}}\alpha^2\lambda^{-2}\bigr)kY_k
 +C_{q,\sigma}\lambda^{-(q-2)}\mathcal Q_k.
 \end{aligned}
\end{equation}
At fixed $\varepsilon$, the bounds on $\phi_\varepsilon,K_\varepsilon$
give $e_{N,k}^\varepsilon\le C_{N,k,\varepsilon}+|V_k|^2$.
The moment assumption controls the transport and diffusion cutoff terms;
H\"older's inequality and $Y_k,\mathcal Q_k\in L^1(0,T)$ control the flux
cutoff terms, as detailed in Section~\ref{subsec:common-cutoff-testing}.
Fatou's lemma preserves the nonnegative dissipation.  Removing the cutoff
and dropping these terms proves \eqref{eq:common-flux-energy} in
distributions.  For the regularized approximations, the endpoint
continuity established in that section gives its integrated form.
\end{proof}

We apply this estimate to the approximations constructed in
\Cref{lem:common-smooth-initial-data}.  For fixed $N$, $\varepsilon$, and $n$,
let $G_N^{0,n,\varepsilon}$ be the smooth initial approximation given there,
and denote by $G_{N,k}^{0,n,\varepsilon}$ its $k$-particle marginal.  Let
$G_N^{\varepsilon,n}(t)$ be the regularized $N$-particle law with initial
density $G_N^{0,n,\varepsilon}$, and denote by $G_{N,k}^{\varepsilon,n}(t)$
its $k$-particle marginal.  Let $U_{N,k}^{\varepsilon,n}$ solve the
regularized version of \eqref{eq:TV-U-PDE} with initial datum
$G_{N,k}^{0,n,\varepsilon}$.  Set
\begin{align*}
 H_{N,k}^{\varepsilon,n}
 &:=G_{N,k}^{\varepsilon,n}-U_{N,k}^{\varepsilon,n},
 &H_{N,k}^{\varepsilon}
 &:=G_{N,k}^{\varepsilon}-U_{N,k}^{\varepsilon},\\
 h_{N,k+1}^{\varepsilon,n}
 &:=U_{N,k+1}^{\varepsilon,n}
   -U_{N,k}^{\varepsilon,n}f_t(z_{k+1}),
 &\xi_{n,\varepsilon}
 &:=\|G_N^{0,n,\varepsilon}-G_N^0\|_1,
\end{align*}
and
\begin{align*}
 R_{i,N,k}^{\varepsilon,n}
 &:=J_{i,k}^{\varepsilon}[h_{N,k+1}^{\varepsilon,n}]
 +(E^\varepsilon-E)(t,x_i)U_{N,k}^{\varepsilon,n},\\
 Y_{N,k}^{\varepsilon,n}(t)
 &:=\int |H_{N,k}^{\varepsilon,n}(t)|^q
 e^{\lambda(t)e_{N,k}^{\varepsilon}}\varpi_k\dd Z_k,\\
 Y_{N,k}^{\varepsilon}(t)
 &:=\int |H_{N,k}^{\varepsilon}(t)|^q
 e^{\lambda(t)e_{N,k}^{\varepsilon}}\varpi_k\dd Z_k.
\end{align*}
The common initial data and
\eqref{eq:fixed-N-smooth-level-N} give
\begin{equation}\label{eq:fixed-N-error-identities}
 H_{N,k}^{\varepsilon,n}(0)=0\quad(1\le k\le N),
 \qquad H_{N,N}^{\varepsilon,n}(t)=0.
\end{equation}
Uniformly in $t$, Markov contraction and coordinate projection give
\begin{align}
 \|G_{N,k}^{\varepsilon,n}-G_{N,k}^{\varepsilon}\|_1
 &\le\xi_{n,\varepsilon},\qquad
 \|U_{N,k}^{\varepsilon,n}-U_{N,k}^{\varepsilon}\|_1
 \le\xi_{n,\varepsilon},\notag\\
 \|H_{N,k}^{\varepsilon,n}-H_{N,k}^{\varepsilon}\|_1
 &\le2\xi_{n,\varepsilon},\qquad
 \|h_{N,k+1}^{\varepsilon,n}-h_{N,k+1}^{\varepsilon}\|_1
 \le2\xi_{n,\varepsilon}.
 \label{eq:fixed-N-L1-approx}
\end{align}

Together with \eqref{eq:TV-h-L1-explicit}, this yields
$\|h_{N,k+1}^{\varepsilon,n}(t)\|_1
\le C_T^k(\Delta_N+\xi_{n,\varepsilon})$.
The initial bounds of \Cref{lem:common-smooth-initial-data} and the
supersolution in \Cref{lem:common-exponential-envelope} give
\eqref{eq:TV-U-bound-regularized} also for $U_{N,k}^{\varepsilon,n}$,
after increasing $A_T$.  Thus \eqref{eq:TV-h-pointwise-bound} also
holds for $h_{N,k+1}^{\varepsilon,n}$ with the regularized energy and
the same uniform constants.  Apply the proof of
\Cref{prop:TV-small-residual}, replacing $\Delta_N$ by
$\Delta_N+\xi_{n,\varepsilon}$ in
\eqref{eq:TV-h-interpolation-explicit} and applying
\Cref{lem:common-weighted-holder} at the fixed exponent $\lambda_*$.
Since the energy is nonnegative, for every $0<\lambda\le\lambda_*$
we obtain
\begin{equation}\label{eq:TV-regularized-residual-bound-smooth}
 \sum_{i=1}^k
 \|R_{i,N,k}^{\varepsilon,n}(t)\|_
 {L^q(e^{\lambda e_{N,k}^{\varepsilon}}\varpi_k)}^q
 \le D_T^k\bigl(\Delta_N+\xi_{n,\varepsilon}
 +\varepsilon^{q(d-s-1)}\bigr).
\end{equation}
The constants are independent of $N,k,n,\varepsilon$.

We combine the flux and remainder estimates and pass to the limits
in \eqref{eq:comparison-order-of-limits}.

\begin{lemma}\label{lem:TV-error-energy}
There exist nonnegative continuous functions $A,L$ on $[0,T]$ and a
constant $D_T\ge1$, depending only on the fixed data of this section and not
on $N,k,n$ or $\varepsilon$, such that
\begin{equation}\label{eq:TV-error-energy}
 Y_{N,k}(t)\le\int_0^t\left[
 kA(r)Y_{N,k}(r)+kL(r)\frac{N-k}{N}Y_{N,k+1}(r)
 +D_T^k\Delta_N\right]\dd r,
 \qquad k<N,
\end{equation}
with zero initial values and $Y_{N,N}=0$.  One may take
\begin{align*}
 A(t)=C_{d,q,\sigma}\bigl(
 1+\mathbf1_{\{\mathcal D=\R^d\}}\alpha^2\lambda(t)^{-2}\bigr),\qquad
 L(t)=C_{d,q,\sigma}\lambda(t)^{-(q-2)}
 \Gamma_{d,q}(\lambda(t))
 \sup_{0<\varepsilon<\varepsilon_0}\|K_\varepsilon\|_{q,\varpi}^q.
\end{align*}
\end{lemma}

\begin{proof}
Apply the estimate to the smooth regularized approximation constructed in
Appendix~\ref{app:fixed-N-passage}.  Its error equation has the form required
by \Cref{lem:common-weighted-flux-energy} with
\begin{equation}\label{eq:regularized-error-flux}
 \mathcal F_i=\frac{N-k}{N}\bigl(
 J_{i,k}^\varepsilon[H_{N,k+1}^{\varepsilon,n}]
 +R_{i,N,k}^{\varepsilon,n}\bigr).
\end{equation}
Because $0\le \frac{N-k}{N}\le1$ and $q\ge1$,
$(\frac{N-k}{N})^q\le \frac{N-k}{N}$.  Therefore
\begin{equation}\label{eq:error-flux-split}
 |\mathcal F_i|^q\le2^{q-1}\frac{N-k}{N}
 \left(|J_{i,k}^\varepsilon[H_{N,k+1}^{\varepsilon,n}]|^q
       +|R_{i,N,k}^{\varepsilon,n}|^q\right).
\end{equation}

For the first term, the energies satisfy
$e_{N,k+1}^\varepsilon\ge e_{N,k}^\varepsilon+1+|v_{k+1}|^2$.
For each $i$, \Cref{lem:common-weighted-holder} therefore gives
\[
 \int |J_{i,k}^\varepsilon[H_{N,k+1}^{\varepsilon,n}]|^q
 e^{\lambda e_{N,k}^\varepsilon}\varpi_k
 \le \Gamma_{d,q}(\lambda)
 \|K_\varepsilon\|_{q,\varpi}^q
 Y_{N,k+1}^{\varepsilon,n}.
\]
Summing over $i=1,\ldots,k$ and inserting
\eqref{eq:error-flux-split} into
\eqref{eq:common-flux-energy} gives
$
 kL(t)\frac{N-k}{N}Y_{N,k+1}^{\varepsilon,n}
$
with $L$ as stated in the lemma.  The coefficient of
$Y_{N,k}^{\varepsilon,n}$ is $kA(t)$.

For the remainder flux we use
\eqref{eq:TV-regularized-residual-bound-smooth}.  Since
$\lambda(t)$ stays in a compact subinterval of $(0,\lambda_*]$ on
$[0,T]$, the factor $\lambda(t)^{-(q-2)}$ in
\eqref{eq:common-flux-energy} is bounded and can be absorbed into $D_T^k$.
Integration in time therefore yields
\begin{align}\label{eq:fixed-N-regularized-integrated}
 Y_{N,k}^{\varepsilon,n}(t)
 \le\int_0^t\Bigl[
 kA(r)Y_{N,k}^{\varepsilon,n}(r)
 +kL(r)\frac{N-k}{N}Y_{N,k+1}^{\varepsilon,n}(r)
 +D_T^k\bigl(\Delta_N+\xi_{n,\varepsilon}
 +\varepsilon^{q(d-s-1)}\bigr)\Bigr]\dd r.
\end{align}
The functions $A,L$ and the base $D_T$ are independent of
$N,k,n,\varepsilon$.  Finally,
\Cref{lem:TV-fixed-N-passage} allows first $n\to\infty$ and then, for fixed
$N$, $\varepsilon\downarrow0$.  It also preserves the zero initial values
and the identity $Y_{N,N}=0$.  We obtain \eqref{eq:TV-error-energy}.
\end{proof}

\subsection{Iteration of the hierarchy}

\label{subsec:common-volterra}

\begin{lemma}
\label{lem:TV-iteration}
Let $A,L,\zeta\ge0$ belong to $L^1(0,T)$, and let $D\ge1$.
Suppose measurable, locally bounded functions
$Y_k:[0,T]\to[0,\infty)$, $1\le k\le N$, satisfy $Y_N=0$ and the following
inequality for some coefficients $0\le b_{N,k}\le1$:
\[
 Y_k(t)\le\int_0^t\left[
 kA(r)Y_k(r)+kL(r)b_{N,k}Y_{k+1}(r)+D^k\zeta(r)\right]\dd r,
 \qquad k<N.
\]
Define
\begin{equation}\label{eq:TV-iteration-quantities}
 I_A(t)=\int_0^tA(r)\dd r,\qquad
 I_L(t)=\int_0^tL(r)e^{I_A(r)}\dd r,
 \qquad \mathcal S(t)=\int_0^t\zeta(r)\dd r.
\end{equation}
If $D I_L(t)<1$, then we have
\begin{equation}\label{eq:TV-iteration-time-bound}
 Y_k(t)\le\mathcal S(t)
 \left(\frac{D e^{I_A(t)}}{1-D I_L(t)}\right)^k,
 \qquad1\le k\le N.
\end{equation}
\end{lemma}

\begin{proof}
Set $\widetilde Y_k(t)=e^{-kI_A(t)}Y_k(t)$ and
$\widetilde Y_N=0$.  Gronwall's inequality applied to the diagonal term
$kA(r)Y_k(r)$ gives
\begin{align*}
 \widetilde Y_k(t)
 &\le D^k\int_0^t e^{-kI_A(r)}\zeta(r)\dd r
 +k\int_0^t L(r)e^{I_A(r)}b_{N,k}\widetilde Y_{k+1}(r)\dd r.
\end{align*}
Since $A,\zeta\ge0$, the first term is bounded by $D^k\mathcal S(t)$.
We obtain the recursion
\begin{equation}\label{eq:TV-iteration-recursion}
 \widetilde Y_k(t)\le D^k\mathcal S(t)
 +k\int_0^t L(r)e^{I_A(r)}b_{N,k}\widetilde Y_{k+1}(r)\dd r.
\end{equation}
Fix $k<N$.  Repeated substitution terminates at level $N$, because
$\widetilde Y_N=0$.  The function $\mathcal S$ is nondecreasing, so its
value at every intermediate time is bounded by $\mathcal S(t)$.
The finite iteration therefore yields
\begin{align*}
 \widetilde Y_k(t)
 &\le D^k\mathcal S(t)
 \sum_{j=0}^{N-k-1}
 \frac{k(k+1)\cdots(k+j-1)}{j!}
 (D I_L(t))^j\prod_{\ell=0}^{j-1}b_{N,k+\ell}\\
 &\le D^k\mathcal S(t)
 \sum_{j=0}^{N-k-1}\binom{k+j-1}{j}(D I_L(t))^j.
\end{align*}
Empty products are one.  The factor $1/j!$ comes from integration over
the ordered time simplex:
\[
 \int_{0<r_j<\cdots<r_1<t}
 \prod_{\ell=1}^j L(r_\ell)e^{I_A(r_\ell)}\,
 \dd r_j\cdots\dd r_1
 =\frac{I_L(t)^j}{j!}.
\]
For $D I_L(t)<1$, the binomial identity
$\sum_{j\ge0}\binom{k+j-1}{j}x^j=(1-x)^{-k}$ therefore yields
\eqref{eq:TV-iteration-time-bound}.  The assertion at $k=N$ follows
from $Y_N=0$.
\end{proof}

\begin{proof}[Proof of \Cref{thm:correlated-comparison}]
Take $G_N^0=Q_N^0$, $\beta_G=\beta_Q$ and, on $\R^d$, $M_G=M_Q$.
Let $A,L,D_T$ be as in \Cref{lem:TV-error-energy}, and define $I_A,I_L$
by \eqref{eq:TV-iteration-quantities}.  Since $I_L(0)=0$, we may choose
$0<\tau\le T$, independent of $N,k$, such that
\begin{equation}\label{eq:TV-time-smallness}
 D_T I_L(\tau)\le\frac12.
\end{equation}
Apply \Cref{lem:TV-iteration} with $\zeta(r)=\Delta_N$, $D=D_T$,
and $b_{N,k}=(N-k)/N$.  This gives
\begin{equation}\label{eq:TV-comparison-weighted-error}
 Y_{N,k}(t)
 \le t\Delta_N\bigl(2D_Te^{I_A(\tau)}\bigr)^k,
 \qquad 0\le t\le\tau.
\end{equation}
Since $\phi\ge0$ and $\lambda(t)\ge\lambda(\tau)>0$ for
$t\le\tau$, H\"older's inequality gives
\[
 \begin{aligned}
 \|H_{N,k}(t)\|_1
 &\le Y_{N,k}(t)^{1/q}
 \left(\int e^{-\lambda(t)e_{N,k}/(q-1)}
                   \varpi_k^{-1/(q-1)}\dd Z_k\right)^{1/q'}\\
 &\le Y_{N,k}(t)^{1/q}
 \left[\int_{\Omega_{\mathcal D}}
 e^{-\lambda(\tau)(1+|v|^2)/(q-1)}
 \varpi(x)^{-1/(q-1)}\dd x\dd v\right]^{k/q'}.
 \end{aligned}
\]
The compatibility condition \eqref{eq:common-weight-compatibility} ensures that
\[
 \begin{aligned}
 &\int_{\Omega_{\mathcal D}}
 e^{-\lambda(\tau)(1+|v|^2)/(q-1)}
 \varpi(x)^{-1/(q-1)}\dd x\dd v\\
 =&e^{-\lambda(\tau)/(q-1)}
 \left(\frac{\pi(q-1)}{\lambda(\tau)}\right)^{d/2}
 \int_{\mathcal D}\varpi(x)^{-1/(q-1)}\dd x<\infty.
 \end{aligned}
\]
Thus \eqref{eq:TV-comparison-weighted-error} yields
\[
 \|H_{N,k}(t)\|_1
 \le C^k(t\Delta_N)^{1/q}
 \le C^k\Delta_N^{1/q},\qquad 0\le t\le\tau.
\]

Combining this bound with \Cref{lem:TV-U-product} proves
\eqref{eq:intro-correlated-TV}, since
$0<\Delta_N\le2$ implies
$\Delta_N\le2^{1-1/q}\Delta_N^{1/q}$.
The global strong particle dynamics and absence of collisions follow from
\Cref{prop:common-finite-N}.
\end{proof}

\subsection{Product initial data and the corollaries}

\begin{proof}[Proof of \Cref{thm:quantitative-comparison}]
The initial Gaussian bound gives
$\rho_0\in L^\infty$.
Take $R=R_*$ and the conditioned density $G_N^0$ from
\Cref{prop:common-riesz-conditioning}.  Set $G_N(t)=P_t^NG_N^0$ and
denote its marginals by $G_{N,k}(t)$.  By
\eqref{eq:common-conditioned-initial-exponential}, this law satisfies
\eqref{eq:intro-correlated-periodic-envelope} or
\eqref{eq:intro-correlated-whole-envelope}, with $Q_{N,k}^0$ replaced by
$G_{N,k}^0$, $\beta_Q=\beta_G$ for any fixed $0<\beta_G<a_0$,
and, on $\R^d$, $M_Q=M_G=M$.
Moreover, since $d/s>2$, we have
\[
 a_{N,k}\le C_RkN^{1-d/s}\le A_0^kN^{1-d/s}
\]
after increasing $A_0$.  Thus the marginal hypothesis \eqref{eq:intro-correlated-marginals}
holds with $\delta_N=N^{1-d/s}$, and
\eqref{eq:comparison-initial-error} gives $\Delta_N\le2/N$.
Applying \Cref{thm:correlated-comparison} gives
\[
 \sup_{t\le\tau}\|G_{N,k}(t)-f_t^{\otimes k}\|_{\TV}
 \le C^kN^{-1/q}
\]
on an interval independent of $N,k$.

The full laws start from $f_0^{\otimes N}$ and its conditioning, whose
total variation distance is $2\eta_N$.  Markov contraction and marginalization yield
\[
 \|F_{N,k}(t)-G_{N,k}(t)\|_{\TV}\le2\eta_N
 \le C_RN^{-(d/s-2)}.
\]
Combining the last two estimates proves \eqref{eq:intro-TV-general}.
The pathwise unique global strong particle dynamics and absence of
collisions follow from \Cref{prop:common-finite-N}.
\end{proof}

\begin{proof}[Proof of \Cref{cor:intro-quantitative-riesz}]
Choose the intermediate exponents as in
\Cref{prop:common-riesz-classical}.  On $\T^d$ we apply
\Cref{prop:common-riesz-classical} with $a_1=a_0/2$.  This gives the
pointwise Gaussian bound, the $C_b^1$ regularity of the field, and the
integrability condition \eqref{eq:intro-velocity-derivative} in
\Cref{ass:limiting-solution}.

On $\R^d$, choose
$M_1=(M+d)/2$, so that $d<M_1<M$, and again set $a_1=a_0/2$.
Then \Cref{prop:common-riesz-classical} yields
\eqref{eq:common-riesz-regularity} with these choices, and the same
field and derivative properties.  Hence
\Cref{ass:limiting-solution} holds in either geometry.  Applying
\Cref{thm:quantitative-comparison} gives \eqref{eq:intro-TV-general}, while
the dynamical properties for fixed $N$ follow from the same theorem.
\end{proof}

\begin{proof}[Proof of \Cref{cor:intro-quantitative-coulomb}]
We set $d=3$ and $s=1$.  By
\Cref{prop:gauss-coulomb-sobolev-core,prop:gaussian-coulomb-solution}, the
limiting VPFP equation has a nonnegative classical solution of unit mass
satisfying \eqref{eq:limiting-field-regularity} and
\eqref{eq:intro-velocity-derivative} with $T=T_f$, together with the
pointwise Gaussian bound
\eqref{eq:gauss-solution-pointwise-bound}.  On $\R^3$, the spatial
exponent in this bound is strictly larger than three; we take this exponent
as $M$ in \Cref{ass:limiting-solution}.  On $\T^3$, the Gaussian
velocity bound suffices.  Thus \Cref{ass:limiting-solution} holds.

For the Coulomb force in dimension three, the kernel condition is
$2q'<3$, equivalently $q>3$, and $d/s-2=1$.
Given $0<b<1/3$, take $q=1/b$ in
\Cref{thm:quantitative-comparison}.  Since
$N^{-1/q}+N^{-1}\le2N^{-b}$, the resulting estimate is
\eqref{eq:intro-TV-coulomb-b} after increasing the constant.
\end{proof}

\section{Riesz interactions on \texorpdfstring{$\T^d$ and $\R^d$}{the torus and Euclidean space}}
\label{sec:kernel-verification}

We verify the kernel estimates, construct the finite particle dynamics,
and remove the regularization in both geometries.  Throughout the
dynamical statements in this section, $\sigma>0$ is fixed.

\subsection{Regularization of the Riesz interaction and force estimates}
\label{subsec:whole-space-kernel-class}

The regularization preserves positivity and evenness and agrees with the
singular potential outside $B_{2\varepsilon}$.

\begin{lemma}[Local regularization of the interaction]
\label{lem:common-riesz-regularization}
Let $d\ge2$, $0<s<d-1$, and $\mathcal D\in\{\T^d,\R^d\}$.
Let $\phi=\phi_{\mathcal D,s}$,
$K=-\nabla\phi$ be as in the introduction.  There exist
$\varepsilon_0\in(0,1)$, $r_0>0$, and smooth even potentials
$\phi_\varepsilon\ge1$
such that
\begin{equation}\label{eq:common-riesz-regularization-properties}
 \phi_\varepsilon=\phi\ \text{when }d_{\mathcal D}(x,0)\ge2\varepsilon,
 \quad \phi_\varepsilon\le C\phi,
 \quad K_\varepsilon=-\nabla\phi_\varepsilon,
 \quad \operatorname{supp}(K_\varepsilon-K)
 \subset\{d_{\mathcal D}(x,0)\le2\varepsilon\}.
\end{equation}
Here $d_{\mathcal D}$ is geodesic distance on the torus and Euclidean
distance on $\R^d$.  The constants are independent of
$0<\varepsilon<\varepsilon_0$.  The forces satisfy
\begin{equation}\label{eq:whole-space-regularized-force-pointwise}
 |K_\varepsilon(x)|+|K(x)|\le C d_{\mathcal D}(x,0)^{-s-1}
 \quad(0<d_{\mathcal D}(x,0)<r_0).
\end{equation}
For $1\le p<d/(s+1)$, the force difference satisfies
\begin{equation}\label{eq:common-force-defect-Lp}
 \|K_\varepsilon-K\|_{L^p(\mathcal D)}
 \le C_p\varepsilon^{d/p-(s+1)}.
\end{equation}
On the torus,
$\sup_{0<\varepsilon<\varepsilon_0}\|K_\varepsilon\|_{L^p}<\infty$ in this range.
On $\R^d$, for $\alpha>0$ write
$\|\cdot\|_{q,\alpha}:=\|\cdot\|_{q,\varpi_\alpha}$ with
$\varpi_\alpha(x)=\langle x\rangle^\alpha$, using the norm in
\eqref{eq:common-kernel-norm}.  If $q>1$ satisfies $(s+1)q'<d$, the weighted norm satisfies
\begin{equation}\label{eq:whole-space-kernel-defect-weighted}
 \|K_\varepsilon-K\|_{q,\alpha}
 \le C\varepsilon^{d/q'-(s+1)}\longrightarrow0.
\end{equation}
If in addition $\alpha>d(q-1)$, we have
\begin{equation}\label{eq:whole-space-kernel-regularization}
 \sup_{0<\varepsilon<\varepsilon_0}\|K_\varepsilon\|_{q,\alpha}<\infty.
\end{equation}
For each fixed $\varepsilon$, $K_\varepsilon$ is smooth and globally
Lipschitz.
\end{lemma}

\begin{proof}
Choose $\chi\in C^\infty([0,\infty))$ with $0\le\chi\le1$,
$\chi'\ge0$, $\chi=0$ on $[0,1]$ and $\chi=1$ on $[2,\infty)$, and set
\[
 g_\varepsilon(r)=\chi(r/\varepsilon)r^{-s}
 +(1-\chi(r/\varepsilon))\varepsilon^{-s}.
\]
Then $g_\varepsilon$ is constant near zero, equals $r^{-s}$ for
$r\ge2\varepsilon$, and satisfies $0<g_\varepsilon(r)\le2^sr^{-s}$.
On the transition annulus, its derivative satisfies
\[
 g_\varepsilon'(r)
 =\varepsilon^{-1}\chi'(r/\varepsilon)(r^{-s}-\varepsilon^{-s})
 -s\chi(r/\varepsilon)r^{-s-1}\le0,
 \qquad |g_\varepsilon'(r)|\le Cr^{-s-1}.
\]
On $\R^d$ set $\phi_\varepsilon(x)=1+c_{d,s}g_\varepsilon(|x|)$.

On $\T^d$, choose an even cutoff $\zeta$ supported in a singular coordinate
chart and equal to one on $B_{r_*}$, and write
\[
 \phi(x)=c_{d,s}\zeta(x)|x|^{-s}+H(x),\qquad H\in C^\infty(\T^d).
\]
After decreasing $r_*$, $\phi(x)\simeq|x|^{-s}$ for $0<|x|<r_*$.  Choose
$\varepsilon_0$ so that $2\varepsilon_0<r_*$ and
$c_{d,s}(2\varepsilon_0)^{-s}+\min H\ge1$, and define
$\phi_\varepsilon=c_{d,s}\zeta g_\varepsilon+H$.  Since $\zeta=1$ on
$B_{2\varepsilon}$, this potential is smooth and even, agrees with $\phi$
outside $B_{2\varepsilon}$, is at least one, and satisfies
$\phi_\varepsilon\le C\phi$.  Indeed,
$g_\varepsilon(r)\ge(2\varepsilon)^{-s}$ for
$0\le r\le2\varepsilon$, and the potential is unchanged outside this ball.
Differentiation gives
\eqref{eq:whole-space-regularized-force-pointwise} and the support statement
in \eqref{eq:common-riesz-regularization-properties}.
For $1\le p<d/(s+1)$, integration gives
\[
 \|K_\varepsilon-K\|_p^p
 \le C\int_0^{2\varepsilon}r^{d-1-(s+1)p}\dd r
 \le C_p\varepsilon^{d-(s+1)p},
\]
which proves \eqref{eq:common-force-defect-Lp}.  On $\R^d$,
$\langle y\rangle^{-\alpha/(q-1)}\le1$ gives directly
\[
 \|K_\varepsilon-K\|_{q,\alpha}
 \le\|K_\varepsilon-K\|_{L^{q'}}
 \le C\varepsilon^{d/q'-(s+1)},
\]
proving \eqref{eq:whole-space-kernel-defect-weighted}.
On $\R^d$, put $a=(s+1)q'<d$ and $\nu=\alpha/(q-1)>d$.
Splitting the integral into $|x-y|<1$ and its complement gives
\[
 \sup_x\int_{\R^d}|x-y|^{-a}\langle y\rangle^{-\nu}\dd y
 \le\int_{|z|<1}|z|^{-a}\dd z
       +\int_{\R^d}\langle y\rangle^{-\nu}\dd y<\infty.
\]
This proves the required bound for $K$.  The triangle inequality and
\eqref{eq:whole-space-kernel-defect-weighted} give the uniform bound for
$K_\varepsilon$.  On $\T^d$ it follows directly from the local force
bound and compactness.
Evenness of the potentials makes the forces odd.  For fixed $\varepsilon$,
the second derivatives of $\phi_\varepsilon$ are bounded, so
$K_\varepsilon$ is globally Lipschitz.
\end{proof}

\subsection{Particle dynamics and removal of the regularization}
\label{sec:whole-space-dynamics}

The singular dynamics are defined on the open set of configurations with
distinct positions.  The extension $\phi(0)=+\infty$ is used in energy
and support arguments.  Initial laws with densities give zero mass to
the collision diagonals.

For repulsive potentials satisfying
$\phi_{\mathcal D,s}(x)\to+\infty$ as $x\to0$, the energy estimate up to the
first collision excludes collisions in the finite particle system; see
\cite[Section~2.4.2]{BJS2025}.  The following proposition includes a bounded
prescribed field for each particle, as required for the auxiliary equations.

\begin{proposition}[Global particle dynamics]
\label{prop:common-finite-N}
Let $d\ge2$ and $0<s<d-1$.  Fix integers $N\ge k\ge1$, $T>0$, and fields
$b_i\in C([0,T];C_b^1(\mathcal D;\R^d))$, $1\le i\le k$, independent of
velocity.  For every deterministic initial
configuration with distinct positions, the system
\begin{equation}\label{eq:common-forced-finite-system}
 \dd X_i=V_i\dd t,
 \qquad
 \dd V_i=\frac1N\sum_{\substack{1\le j\le k\\j\ne i}}
 K(X_i-X_j)\dd t+b_i(t,X_i)\dd t+\sqrt{2\sigma}\dd B_i,
 \qquad 1\le i\le k,
\end{equation}
admits a pathwise unique strong solution on $[0,T]$, and collisions do not
occur on this interval.  In particular, for $k=N$ and $b_i=0$, the
unregularized particle system \eqref{eq:intro-riesz-particle} admits a
pathwise unique global strong solution without collisions.  Its solutions
form a conservative strong Markov family on the set of configurations with
distinct particle positions.  We denote by $P_t^N$ the associated forward
Markov operator on laws; it contracts total variation.  The existence, pathwise uniqueness and noncollision conclusions also
hold for initial laws supported on this open set and independent of the
Brownian motions.  For time-dependent prescribed fields, the associated
transition kernels form a time-inhomogeneous Markov family.
\end{proposition}

\begin{proof}
Local existence and pathwise uniqueness up to the first collision or
explosion follow from the local Lipschitz continuity of the coefficients away
from the collision diagonals; see \cite[Section~5.2]{KaratzasShreve1991}.
The energy argument of \cite[Section~2.4.2]{BJS2025} also accommodates
the prescribed fields.  Set
\begin{equation}\label{eq:common-particle-energy}
 \mathcal H_{N,k}:=\tfrac12\bigl(e_{N,k}(Z_k)-k\bigr),
\end{equation}
where $e_{N,k}$ is defined in \eqref{eq:intro-particle-energy}.
The internal interaction terms cancel in It\^o's formula, which yields
\begin{equation}\label{eq:common-forced-energy-ito}
 \dd\mathcal H_{N,k}
 =d\sigma k\dd t
  +\sum_{i=1}^kV_i\cdot b_i(t,X_i)\dd t
  +\sqrt{2\sigma}\sum_{i=1}^kV_i\cdot\dd B_i.
\end{equation}
Set $\mathcal B_T:=\max_i\sup_{t\le T}\|b_i(t)\|_\infty$.  Since
$\phi_{\mathcal D,s}\ge0$, the additional drift satisfies
\[
 \left|\sum_{i=1}^kV_i\cdot b_i(t,X_i)\right|
 \le \mathcal H_{N,k}+\frac{k}{2}\mathcal B_T^2.
\]
For $L>0$, let $\tau_L:=\inf\{t:\mathcal H_{N,k}(t)\ge L\}$,
initially before the maximal existence time.  On an energy sublevel,
velocities are bounded and pair distances are bounded away from zero.
On $[0,T]$, bounded velocities also bound the displacement of each
position.  Thus the local solution continues until it leaves that
sublevel.  Since
\[
 \E\int_0^{T\wedge\tau_L}\sum_{i=1}^k|V_i(r)|^2\dd r
 \le2LT,
\]
the stopped stochastic integral has mean zero.  Taking expectations in
\eqref{eq:common-forced-energy-ito} gives
\[
 \E\mathcal H_{N,k}(t\wedge\tau_L)
 \le \E\mathcal H_{N,k}(0)
 +k\left(d\sigma+\frac12\mathcal B_T^2\right)t
 +\int_0^t\E\mathcal H_{N,k}(r\wedge\tau_L)\dd r.
\]
Gronwall's inequality and
$L\mathbf1_{\{\tau_L\le T\}}\le\mathcal H_{N,k}(T\wedge\tau_L)$ give
\begin{equation}\label{eq:common-stopped-energy}
 \begin{aligned}
 \sup_{t\le T}\E\mathcal H_{N,k}(t\wedge\tau_L)
 &\le C_T\bigl(\E\mathcal H_{N,k}(0)+k\bigr),\\
 \Pp(\tau_L\le T)
 &\le \frac{C_T}{L}\bigl(\E\mathcal H_{N,k}(0)+k\bigr),
 \end{aligned}
\end{equation}
whenever the initial energy is integrable.  Here $C_T$ depends on
$d,\sigma,T,\mathcal B_T$ and is independent of $N,k,L$.  For deterministic initial
data the energy is finite.  Letting $L\to\infty$ excludes finite-time
explosion of $\mathcal H_{N,k}$.  Since $\phi_{\mathcal D,s}\ge0$, it follows that
$\sum_{i=1}^k|V_i|^2$ cannot diverge in finite time.  In the whole-space case, the position equation gives
\[
 X_i(t)=X_i(0)+\int_0^tV_i(r)\dd r,
\]
so the positions cannot escape to infinity on a finite time interval.  Since
$\phi_{\mathcal D,s}(x)\to+\infty$ as $x\to0$, a collision would force
$\mathcal H_{N,k}$ to diverge.  Hence no collision occurs on $[0,T]$.
As $T$ is arbitrary, this gives global existence for $k=N$ and
$b_i=0$.  The strong Markov property follows from pathwise uniqueness
and nonexplosion; see \cite[Section~5.4.C]{KaratzasShreve1991}.  Random
initial data are treated by conditioning on the initial configuration.
\end{proof}

\begin{proposition}
\label{prop:common-regularization-removal}
Let $d\ge2$ and $0<s<d-1$.  Fix integers $N\ge k\ge1$, $T>0$, and fields
$b_i\in C([0,T];C_b^1(\mathcal D;\R^d))$, $1\le i\le k$, independent of
velocity.  Let $\mu_k$ be a probability law supported on configurations
with distinct positions, independent of the Brownian motions.
Let $Z_k$ and $Z_k^\varepsilon$ be the singular and regularized solutions
of \eqref{eq:common-forced-finite-system}, driven by the same initial
configuration and Brownian motions, with $K$ replaced by $K_\varepsilon$
in the regularized system.  Then
\begin{equation}\label{eq:whole-space-path-TV}
 \|\operatorname{Law}(Z_k^\varepsilon|_{[0,T]})
       -\operatorname{Law}(Z_k|_{[0,T]})\|_{\TV}
 \le2\Pp_{\mu_k}(d_T\le2\varepsilon)\longrightarrow0,
\end{equation}
where
\[
 d_T=\min_{0\le t\le T}\min_{i<j}
       d_{\mathcal D}(X_i(t),X_j(t)),
 \qquad d_T=+\infty\quad\text{if }k=1.
\]
If $\mu_k$ has a density, the time laws $F_k^\varepsilon(t)$ and $F_k(t)$
have densities in $C([0,T];L^1)$ satisfying
\begin{equation}\label{eq:whole-space-regularization-removal-L1}
 \sup_{0\le t\le T}
 \|F_k^\varepsilon(t)-F_k(t)\|_1\longrightarrow0.
\end{equation}
All limits in \eqref{eq:whole-space-path-TV} and
\eqref{eq:whole-space-regularization-removal-L1} are taken at fixed
$N,k$.  Taking $k=N$ and $b_i=0$ gives the regularization removal
for the $N$-particle dynamics.
\end{proposition}

\begin{proof}
Under the synchronous coupling, we set
\[
 \mathcal O_{k,\varepsilon}
 :=\{Z_k:\min_{i<j}d_{\mathcal D}(x_i,x_j)>2\varepsilon\}
\]
and
\[
 \widehat\tau_\varepsilon
 :=\inf\{t\in[0,T]:Z_k(t)\notin\mathcal O_{k,\varepsilon}
 \text{ or }Z_k^\varepsilon(t)\notin\mathcal O_{k,\varepsilon}\}.
\]
The two systems have identical coefficients on $\mathcal O_{k,\varepsilon}$;
pathwise uniqueness therefore gives
$Z_k^\varepsilon=Z_k$ on $[0,\widehat\tau_\varepsilon)$.  Continuity then gives
\[
 \widehat\tau_\varepsilon
 =\inf\{t\in[0,T]:\min_{i<j}
 d_{\mathcal D}(X_i(t),X_j(t))\le2\varepsilon\}.
\]
By \Cref{prop:common-finite-N}, $d_T>0$ almost surely, so we obtain
\[
 \Pp_{\mu_k}(\widehat\tau_\varepsilon\le T)
 =\Pp_{\mu_k}(d_T\le2\varepsilon)\longrightarrow0,
\]
and the coupling inequality yields \eqref{eq:whole-space-path-TV}.

Suppose now that $\mu_k$ has density $g$.  For fixed $\varepsilon>0$,
subtracting the additive Brownian path from the velocity variables turns
the regularized system into a random ordinary differential equation with
globally Lipschitz coefficients, continuously differentiable in phase
space.  Its solution maps $\Phi_t^\varepsilon$ are $C^1$ diffeomorphisms;
see also \cite[Sections~4.5--4.6]{Kunita1990}.  Writing $\mathbf K^\varepsilon(X)$
for the vector of internal forces, the phase-space drift
$a^\varepsilon(t,X,V)=(V,\mathbf K^\varepsilon(X)+b(t,X))$ has zero divergence.
The additive noise has zero spatial derivative.  The Jacobian equation therefore gives the pathwise identity
\[
 \frac{\dd}{\dd t}\det D\Phi_t^\varepsilon
 =(\operatorname{div}a^\varepsilon)(t,\Phi_t^\varepsilon)
       \det D\Phi_t^\varepsilon=0,
 \qquad \det D\Phi_0^\varepsilon=1.
\]
Consequently the time law has density
\[
 F_k^\varepsilon(t)=\E\bigl[g\circ(\Phi_t^\varepsilon)^{-1}\bigr].
\]
For $g\in C_c$, continuity of the flow and its inverse gives pathwise
$L^1$ continuity of the compositions: over a compact time interval their
supports lie in a pathwise compact set.  Their $L^1$ norms equal
$\|g\|_1$, so dominated convergence also gives continuity after
expectation.  Approximation of an arbitrary $g\in L^1$ by $C_c$ functions,
together with this $L^1$ isometry, proves
$F_k^\varepsilon\in C([0,T];L^1)$.
Since evaluation contracts total variation, we obtain
\[
 \sup_{t\le T}\|F_k^\varepsilon(t)-F_k^\delta(t)\|_1
 \le 2\Pp_{\mu_k}(d_T\le2\varepsilon)
     +2\Pp_{\mu_k}(d_T\le2\delta)
 \longrightarrow0.
\]
Completeness of $C([0,T];L^1)$ gives a density-valued limit.  By
\eqref{eq:whole-space-path-TV}, its value at each time is the law of
$Z_k(t)$.  This proves \eqref{eq:whole-space-regularization-removal-L1}
and $F_k\in C([0,T];L^1)$.
\end{proof}

\appendix

\section{Approximation at fixed particle number}
\label{app:fixed-N-passage}

We approximate the symmetric initial densities of
Section~\ref{sec:common-comparison} and justify the $L^1$ stability and
weighted energy estimates.  In the energy calculation we first remove
the phase-space cutoff and then the prescribed-field mollification.
We subsequently pass $n\to\infty$ and then $\varepsilon\downarrow0$,
always at fixed $N$.  The full order of the approximation limits is
\[
 R\to\infty,\qquad \delta\downarrow0,\qquad
 n\to\infty,\qquad \varepsilon\downarrow0,
\]
where $R$ is the phase-space cutoff radius, $\delta$ the prescribed-field
mollification scale, $n$ the initial-density approximation index, and
$\varepsilon$ the force regularization scale.
The distributional hierarchy and the integrated
energy inequality are treated separately; both limits preserve
$H_{N,k}(0)=0$ and $H_{N,N}=0$.

\subsection{Smooth approximation of the initial density}
\label{subsec:common-smooth-data}

\begin{lemma}\label{lem:common-smooth-initial-data}
Let $G_N^0$ be a symmetric probability density satisfying the
pointwise hypothesis \eqref{eq:intro-correlated-periodic-envelope} or
\eqref{eq:intro-correlated-whole-envelope}, with the notation and
parameters $\beta_0<\beta_1$ of Section~\ref{sec:common-comparison}.
For every fixed $N$ and
$0<\varepsilon<\varepsilon_0$ there exist smooth symmetric probability
densities $G_N^{0,n,\varepsilon}$ and a constant $A_2\ge1$, independent of
$N,k,n,\varepsilon$, such that
\[
 \|G_N^{0,n,\varepsilon}-G_N^0\|_1
 =\xi_{n,\varepsilon}\longrightarrow0.
\]
Denote by $G_{N,k}^{0,n,\varepsilon}$ the $k$-particle marginal of
$G_N^{0,n,\varepsilon}$.  On $\R^d$, it satisfies
\begin{equation}\label{eq:smooth-initial-envelope}
 G_{N,k}^{0,n,\varepsilon}
 \le A_2^k e^{-\beta_0e_{N,k}^\varepsilon}
       \prod_{j=1}^{k}\langle x_j\rangle^{-M_*},
 \qquad 1\le k\le N,
\end{equation}
and on $\T^d$ the same estimate holds without the spatial product.
\end{lemma}

\begin{proof}
We use the initial bound \eqref{eq:comparison-initial-regularized-envelope}.

Fix $N$ and $\varepsilon$.  We first truncate the initial density with the
same one-particle cutoff in each variable.  Choose $R_n\uparrow\infty$.  On
$\R^d$, let $\chi\in C_c^\infty(\R^{2d})$ satisfy $0\le\chi\le1$ and
$\chi=1$ on $B_1$, and set $\chi_R(z)=\chi(z/R)$.  On $\T^d$, choose
$\chi\in C_c^\infty(\R^d)$ with the same properties and set
$\chi_R(x,v)=\chi(v/R)$.  Set
\[
 c_n=\int G_N^0\prod_{i=1}^N\chi_{R_n}(z_i)\dd Z_N,
 \qquad
 \widetilde G_N^{0,n}=c_n^{-1}G_N^0\prod_{i=1}^N\chi_{R_n}(z_i).
\]
Denote by $\widetilde G_{N,k}^{0,n}$ the $k$-particle marginal of
$\widetilde G_N^{0,n}$.  By dominated convergence, $c_n\to1$, and we may
assume $c_n\ge1/2$.  The truncated marginals therefore satisfy
\[
 \widetilde G_{N,k}^{0,n}
 \le2G_{N,k}^0.
\]
We next mollify $\widetilde G_N^{0,n}$ with a product mollifier
$\zeta_{\delta_n}^{\otimes N}$, where $\zeta\ge0$ has integral one and is
supported in the unit phase-space ball.  The convolution is periodic in $x$
on $\T^d$.  Denote the resulting density by $G_N^{0,n,\varepsilon}$.
Product convolution preserves symmetry and mass,
and commutes with taking marginals.  Choose $\delta_n\downarrow0$ so that
$\delta_n\le1$ and $2\|K_\varepsilon\|_\infty\delta_n\le1$.  Write
$Y_k=(y_1,\ldots,y_k)$ and $W_k=(w_1,\ldots,w_k)$ for the translation
variables in the product mollifier.  On its support, the potential difference satisfies
\[
 \left|\frac1N\sum_{\substack{1\le i,j\le k\\i\ne j}}\phi_\varepsilon(x_i-x_j)
 -\frac1N\sum_{\substack{1\le i,j\le k\\i\ne j}}\phi_\varepsilon(x_i-y_i-x_j+y_j)\right|
 \le 2\|K_\varepsilon\|_\infty\delta_n\frac{k(k-1)}{N}\le k,
\]
and, for $v\in\R^d$ and $|w|\le\delta_n\le1$, we have
\[
 \beta_0|v|^2-\beta_1|v-w|^2
 \le \frac{\beta_0\beta_1}{\beta_1-\beta_0}|w|^2
 \le \frac{\beta_0\beta_1}{\beta_1-\beta_0}.
\]
For $|y|\le1$, the spatial weight on $\R^d$ satisfies
\[
 \langle x-y\rangle^{-M_*}\le C\langle x\rangle^{-M_*},
\]
while no spatial factor is present on $\T^d$.  We therefore have
\[
 \beta_0e_{N,k}^\varepsilon(Z_k)
 -\beta_1e_{N,k}^\varepsilon(Z_k-(Y_k,W_k))
 \le C_{\beta_0,\beta_1}k,
\]
and hence the mollified marginals satisfy \eqref{eq:smooth-initial-envelope}.
For example, one may take
\[
 A_2=2A_0C_{M_*}
 \exp\!\left(\beta_0+\frac{\beta_0\beta_1}{\beta_1-\beta_0}\right),
\]
where $C_{M_*}$ is the constant in the spatial translation estimate and is
one on the torus.  Thus $A_2$ is independent of $N,k,n,\varepsilon$.
Finally, the $L^1$ contraction under convolution gives
\[
 \|G_N^{0,n,\varepsilon}-G_N^0\|_1
 \le\|\widetilde G_N^{0,n}-G_N^0\|_1
   +\|\zeta_{\delta_n}^{\otimes N}*G_N^0-G_N^0\|_1\longrightarrow0.
\]
This completes the proof.
\end{proof}

For the regularized laws and auxiliary solutions introduced above, the
coefficient $(N-k)/N$ vanishes at $k=N$, so uniqueness gives
\begin{equation}\label{eq:fixed-N-smooth-level-N}
 U_{N,N}^{\varepsilon,n}=G_N^{\varepsilon,n}.
\end{equation}
For these initial data, apply the supersolution construction of
\Cref{lem:TV-U-pointwise} with $A(0)=\max\{A_0,A_2\}$.
The identity \eqref{eq:fixed-N-smooth-level-N} and marginalization give
\eqref{eq:TV-U-bound-regularized} and
\eqref{eq:TV-GN-bound-regularized} for
$U_{N,k}^{\varepsilon,n}$ and $G_{N,k}^{\varepsilon,n}$, respectively,
after enlarging $A_T,B_T$.  The constants $A_T,B_T,\beta_T$ are independent
of $N,k,n,\varepsilon$, and the spatial factors are omitted on $\T^d$.

\subsection{\texorpdfstring{$L^1$}{L1} stability for the regularized equation}

\begin{lemma}\label{lem:common-L1-stability}
Fix $N\ge k\ge1$, $0<\varepsilon<\varepsilon_0$, and $T>0$.
Let $b_i\in C([0,T];C_b^1(\mathcal D;\R^d))$, and set
\[
 \mathcal P=\partial_t+L_{N,k}^\varepsilon
       +\sum_i b_i(t,x_i)\cdot\nabla_{v_i}-\sigma\Delta_{V_k}.
\]
Let $g\ge0$ belong to $L^1$, and let $U\in C([0,T];L^1)$ be its
forward evolution under $\mathcal P U=0$.
Suppose $u\in C([0,T];L^1)$ satisfies $\mathcal P u=r$ in distributions,
with its $L^1$ initial trace, and
$r,\nabla_{V_k}u\in L^1(0,T;L^1)$.
Then, for every $t\in[0,T]$, the following stability estimate holds:
\begin{equation}\label{eq:common-L1-stability}
 \sup_{0\le a\le t}\|U(a)-u(a)\|_1
 \le\|g-u(0)\|_1+\int_0^t\|r(s)\|_1\dd s.
\end{equation}
Writing $a_+=\max\{a,0\}$, one also has
\begin{equation}\label{eq:common-L1-order}
 \sup_{0\le a\le t}\|(U(a)-u(a))_+\|_1
 \le\|(g-u(0))_+\|_1+\int_0^t\|(-r(s))_+\|_1\dd s.
\end{equation}
\end{lemma}

\begin{proof}
Extend $b_i$ constantly beyond the time endpoints and mollify in time
and space, periodically in space on the torus.  The smooth fields
$b_i^\nu$ converge uniformly to $b_i$ and have a common $C_b^1$ bound.
Write $\mathcal P^\nu$ for the corresponding operator and $U^\nu$ for
the evolution from $g$.  Then $w^\nu=U^\nu-u$ satisfies
$\mathcal P^\nu w^\nu=-r^\nu$, where
\[
 r^\nu=r+\sum_i(b_i^\nu-b_i)(t,x_i)\cdot\nabla_{v_i}u.
\]

Fix $t\le T$ and a smooth compactly supported terminal function $\psi$
with $\|\psi\|_\infty\le1$; on the torus it is periodic in position.
Let $\varphi$ solve the backward Kolmogorov equation
\[
 \partial_s\varphi+L_{N,k}^\varepsilon\varphi
 +\sum_i b_i^\nu(s,x_i)\cdot\nabla_{v_i}\varphi
 +\sigma\Delta_{V_k}\varphi=0,
 \qquad \varphi(t)=\psi.
\]
The stochastic representation and differentiation of the smooth flow give
\[
 \sup_{0\le s\le t}\|\varphi(s)\|_\infty\le1,\qquad
 \sup_{0\le s\le t}\left(
 \|\nabla_{X_k,V_k}\varphi(s)\|_\infty+
 \|D_{V_k}^2\varphi(s)\|_\infty\right)
 \le C_2.
\]
Here $C_2$ depends on $N,k,\varepsilon,\nu,T$, the $C_b^2$ norm of $\psi$,
the bounds for $b^\nu$, and those for the first and second derivatives of
the regularized drift.  We choose $C_2$ large enough for the cutoff
estimate below as well.
To use this test function, choose a radial cutoff $\chi_R$, equal to
one on the ball of radius $R$ and zero outside the ball of radius $2R$,
in all noncompact phase-space variables.  Its first and second
derivatives are bounded by $C/R$ and $C/R^2$.  On the support
$\mathcal A_R$ of these derivatives, $|V_k|\le2R$; on the torus the
cutoff depends only on velocity.  The additional transport and diffusion
terms are therefore bounded by
\[
 C_2\int_0^t\!\int_{\mathcal A_R}
 |w^\nu(s,Z_k)|\left(1+R^{-1}+R^{-2}\right)\dd Z_k\dd s.
\]
This tends to zero since $w^\nu\in L^1((0,T)\times\Omega_{\mathcal D}^k)$
and $\mathcal A_R$ escapes every compact set.  The source terms converge
by their $L^1$ integrability.  Time cutoffs and the
$C_tL^1$ traces justify the initial and terminal evaluations.  We obtain
\begin{equation}\label{eq:auxiliary-duality-identity}
 \langle w^\nu(t),\psi\rangle
 =\langle g-u(0),\varphi(0)\rangle
 -\int_0^t\langle r^\nu(s),\varphi(s)\rangle\dd s.
\end{equation}
Taking the supremum over $\psi$ yields
\[
 \|w^\nu(t)\|_1\le\|g-u(0)\|_1+
                         \int_0^t\|r^\nu(s)\|_1\dd s.
\]
For $0\le\psi\le1$, the backward representation gives
$0\le\varphi\le1$.  The same identity then implies
\[
 \|(w^\nu(t))_+\|_1
 \le\|(g-u(0))_+\|_1+\int_0^t\|(-r^\nu(s))_+\|_1\dd s.
\]

Let $\Phi_s^\nu$ and $\Phi_s$ be the stochastic flows driven by the same
Brownian motions, with fields $b^\nu$ and $b$, respectively.  Gronwall's
inequality and the common Lipschitz bound give
\begin{align}\label{eq:common-flow-field-stability}
 \sup_{s\le T}\|\Phi_s^\nu-\Phi_s\|_\infty
 &\le C_3\sum_i\|b_i^\nu-b_i\|_\infty,\notag\\
 \sup_{s\le T}\|(\Phi_s^\nu)^{-1}-\Phi_s^{-1}\|_\infty
 &\le \sup_{s\le T}\|D(\Phi_s^\nu)^{-1}\|_\infty
       \sup_{s\le T}\|\Phi_s^\nu-\Phi_s\|_\infty\notag\\
 &\le C_3\sum_i\|b_i^\nu-b_i\|_\infty\longrightarrow0.
\end{align}
Here $C_3$ depends on $N,k,\varepsilon,T$ and the common $C_b^1$ bound
of the fields, and is independent of $\nu$.  Periodic lifts are used on
$\T^d$.
The Jacobian identity in \Cref{prop:common-regularization-removal} gives
\[
 \|g\circ(\Phi_s^\nu)^{-1}\|_1
 =\|g\circ\Phi_s^{-1}\|_1=\|g\|_1.
\]
For $g\in C_c$, the transported supports, for all $s\le T$ and all
sufficiently large $\nu$, lie in a pathwise compact set.  We therefore obtain
\[
 \sup_{s\le T}
 \|g\circ(\Phi_s^\nu)^{-1}-g\circ\Phi_s^{-1}\|_1\longrightarrow0
 \quad\text{almost surely},
\]
and the left-hand side is bounded by $2\|g\|_1$.  Dominated convergence gives
\[
 \sup_{s\le T}\|U^\nu(s)-U(s)\|_1
 \le\E\sup_{s\le T}
 \|g\circ(\Phi_s^\nu)^{-1}-g\circ\Phi_s^{-1}\|_1\longrightarrow0.
\]
Approximation in $L^1$ and the displayed isometries give the same limit
for every $g\in L^1$.  Together with
\[
 \int_0^T\|r^\nu-r\|_1\dd s
 \le\sum_i\|b_i^\nu-b_i\|_\infty
       \int_0^T\|\nabla_{v_i}u\|_1\dd s\longrightarrow0,
\]
this proves the two estimates at each time.  Taking the supremum over
$0\le a\le t$ yields \eqref{eq:common-L1-stability} and
\eqref{eq:common-L1-order}.
\end{proof}

\subsection{Distributional limit of the hierarchy}

\begin{proposition}[Passage to the singular hierarchy]\label{prop:common-singular-hierarchy}
Fix $N$ and $q$ as in \eqref{eq:intro-q-range}.  Suppose that, at every level, the
regularized particle marginals of Section~\ref{sec:common-comparison}
satisfy
\begin{equation}\label{eq:common-singular-hierarchy-convergence}
 G_{N,k}^\varepsilon\longrightarrow G_{N,k}
 \quad\text{in }C([0,T];L^1),
\end{equation}
and that, for some continuous $\lambda:[0,T]\to(0,\infty)$, they satisfy
\begin{equation}\label{eq:common-singular-hierarchy-weight}
 \operatorname*{ess\,sup}_{t\le T}
 \int|G_{N,k}^\varepsilon(t)|^q e^{\lambda(t)e_{N,k}^\varepsilon}
      \varpi_k\dd Z_k\le C_k,
\end{equation}
with $C_k$ independent of $\varepsilon$.  Then the limiting marginals satisfy
the singular version of \eqref{eq:intro-BBGKY}, with $K$ in place of
$K_\varepsilon$ and all $\varepsilon$ superscripts omitted, in distributions,
with initial trace $G_{N,k}^0$.  The last term is omitted at $k=N$.

If \eqref{eq:common-singular-hierarchy-convergence} and
\eqref{eq:common-singular-hierarchy-weight} hold instead for the auxiliary
densities $U_{N,k}^\varepsilon,U_{N,k}$, then the latter satisfy
\eqref{eq:TV-U-PDE}
in distributions, with the same initial trace $G_{N,k}^0$.
\end{proposition}

\begin{proof}
We first consider the particle marginals.  For
$\psi\in C_c^\infty([0,T)\times\Omega_{\mathcal D}^k)$, periodic in
position on the torus, the regularized hierarchy reads
\begin{align}
 &\int G_{N,k}^\varepsilon(0,Z_k)\psi(0,Z_k)\dd Z_k+\int_0^T\!\int G_{N,k}^\varepsilon
       \bigl(\partial_t\psi+L_{N,k}^\varepsilon\psi
                    +\sigma\Delta_{V_k}\psi\bigr)\dd Z_k\dd t\notag\\
 &\quad+\frac{N-k}{N}\sum_{i=1}^k\int_0^T\!\int
       J_{i,k}^\varepsilon[G_{N,k+1}^\varepsilon]
                    \cdot\nabla_{v_i}\psi\dd Z_k\dd t=0.
 \label{eq:common-regularized-hierarchy-weak}
\end{align}
At $k=N$ the last line is absent.  We pass to the limit in this identity.

Set $\underline\lambda=\min_{t\le T}\lambda(t)>0$ and
\[
 \widetilde W_k
 =\exp\left(\underline\lambda\sum_{i=1}^k(1+|v_i|^2)\right)\varpi_k.
\]
Since $e_{N,k}^\varepsilon\ge\sum_i(1+|v_i|^2)$,
\eqref{eq:common-singular-hierarchy-weight} bounds
$G_{N,k}^\varepsilon$ in
$L^q((0,T)\times\Omega_{\mathcal D}^k;\widetilde W_k)$.
Reflexivity gives weakly convergent subsequences.  On compact subsets,
the weight and its reciprocal are bounded, and the strong $L^1$
convergence identifies every weak limit with $G_{N,k}$.  Thus the whole
family satisfies
\[
 G_{N,k}^\varepsilon\rightharpoonup G_{N,k}
 \quad\text{in }L^q((0,T)\times\Omega_{\mathcal D}^k;\widetilde W_k).
\]
The time derivative, free transport and diffusion terms in
\eqref{eq:common-regularized-hierarchy-weak} pass to the limit by
\eqref{eq:common-singular-hierarchy-convergence}.  This convergence also
gives $G_{N,k}^\varepsilon(0)\to G_{N,k}^0$ in $L^1$, and hence the
initial term.

For the interaction with level $k+1$, \Cref{lem:common-weighted-holder},
applied to the kinetic energies alone, gives
\begin{equation}\label{eq:singular-passage-flux-map}
 \|J_{i,k}[h]\|_{L_t^qL^q(\widetilde W_k)}
 \le\Gamma_{d,q}(\underline\lambda)^{1/q}\|K\|_{q,\varpi}
      \|h\|_{L_t^qL^q(\widetilde W_{k+1})}.
\end{equation}
Decompose the flux difference as
\begin{align*}
 J_{i,k}^\varepsilon[G_{N,k+1}^\varepsilon]-J_{i,k}[G_{N,k+1}]
 =J_{i,k}[G_{N,k+1}^\varepsilon-G_{N,k+1}]+(J_{i,k}^\varepsilon-J_{i,k})[G_{N,k+1}^\varepsilon].
\end{align*}
The first term converges weakly to zero because
\eqref{eq:singular-passage-flux-map} defines a bounded linear operator
between the two weighted spaces.  For the second term, the same estimate
with $K_\varepsilon-K$ in place of $K$ yields
\[
 \|(J_{i,k}^\varepsilon-J_{i,k})[G_{N,k+1}^\varepsilon]\|_
       {L_t^qL^q(\widetilde W_k)}
 \le\Gamma_{d,q}(\underline\lambda)^{1/q}
       \|K_\varepsilon-K\|_{q,\varpi}
       \|G_{N,k+1}^\varepsilon\|_{L_t^qL^q(\widetilde W_{k+1})}
 \longrightarrow0.
\]
Pairing with $\nabla_{v_i}\psi$ therefore gives convergence of the last
line of \eqref{eq:common-regularized-hierarchy-weak}.

It remains to consider the internal force.  Since $(s+1)q'<d$ and
$\psi$ is compactly supported, we have
\[
 K(x_i-x_j)\cdot\nabla_{v_i}\psi
 \in L^{q'}\bigl(\widetilde W_k^{-1/(q-1)}\dd t\dd Z_k\bigr).
\]
Consequently, the weak convergence gives
\[
 \int_0^T\!\int(G_{N,k}^\varepsilon-G_{N,k})
             K(x_i-x_j)\cdot\nabla_{v_i}\psi\dd Z_k\dd t\longrightarrow0.
\]
The remaining kernel difference satisfies
\begin{align*}
 \left|\int_0^T\!\int G_{N,k}^\varepsilon(K_\varepsilon-K)(x_i-x_j)
                 \cdot\nabla_{v_i}\psi\dd Z_k\dd t\right|\le C_\psi\|G_{N,k}^\varepsilon\|_{L_t^qL^q(\widetilde W_k)}
                    \|K_\varepsilon-K\|_{L^{q'}(\mathcal D)}
 \longrightarrow0
\end{align*}
by \eqref{eq:common-force-defect-Lp}.  On $\R^d$ the kernel difference
is supported in $B_{2\varepsilon}$, so the displayed norm is finite.
All terms of \eqref{eq:common-regularized-hierarchy-weak} thus converge
to their singular counterparts.  This proves the singular version of
\eqref{eq:intro-BBGKY} with its initial trace.

For the auxiliary densities, the same weighted weak convergence and
internal-force estimates apply with $U_{N,k}^\varepsilon$ in place of
$G_{N,k}^\varepsilon$.  In the regularized weak formulation, the last
line of \eqref{eq:common-regularized-hierarchy-weak} is replaced by
\[
 \frac{N-k}{N}\sum_{i=1}^k\int_0^T\!\int
       U_{N,k}^\varepsilon E_f(t,x_i)\cdot\nabla_{v_i}\psi\dd Z_k\dd t.
\]
The field $E_f$ is bounded, so this term converges by the assumed
$C_tL^1$ convergence.  The initial densities are $G_{N,k}^0$ in both
systems.  We obtain \eqref{eq:TV-U-PDE} with
that initial trace.
\end{proof}

\subsection{Localization of the weighted energy estimate}
\label{subsec:common-cutoff-testing}

The internal force is smooth for fixed $\varepsilon$ and the initial density
is smooth after the approximation of Section~\ref{subsec:common-smooth-data}.
The prescribed field $E$ belongs to $C([0,T];C_b^1)$.  We approximate it
by $E^\delta$, where $\delta$ is the mollification parameter;
$E^\varepsilon=K_\varepsilon*\rho_f$ continues to denote the field obtained
by regularizing the interaction kernel.  Extend $E$ constantly beyond the
time endpoints and mollify in time and space, periodically in space on
$\T^d$.  The resulting smooth fields satisfy
\[
 E^\delta\longrightarrow E\quad\text{in }C([0,T];C_b),
 \qquad
 \sup_{\delta>0}\bigl(\|E^\delta\|_\infty
       +\|\nabla_xE^\delta\|_\infty\bigr)
 \le C\|E\|_{C_tC_b^1}.
\]
Let $U_{N,k}^{\varepsilon,n,\delta}$ denote the auxiliary solution obtained by
replacing $E$ with $E^\delta$.  For fixed $N,k,\varepsilon,\delta$, the phase-space derivatives of
positive order of the regularized drifts are bounded.  The stochastic
flows and their inverses therefore have bounded phase-space derivatives
of every order.  The chain rule and volume preservation give, for every
integer $r\ge0$,
\begin{equation}\label{eq:fixed-N-smooth-derivatives}
 \begin{aligned}
 \sup_{t\le T}\bigl(
 \|G_N^{\varepsilon,n}(t)\|_{W^{r,1}}
 +\|U_{N,k}^{\varepsilon,n,\delta}(t)\|_{W^{r,1}}\bigr)\le C_{r,N,k,\varepsilon,\delta,T}
       \|G_N^{0,n,\varepsilon}\|_{W^{r,1}}.
 \end{aligned}
\end{equation}
For $k<N$, marginalization and differentiation of the regularized kernel yield
\begin{equation}\label{eq:fixed-N-smooth-flux-derivatives}
 \begin{aligned}
 \|G_{N,k}^{\varepsilon,n}(t)\|_{W^{r,1}}
 &\le\|G_N^{\varepsilon,n}(t)\|_{W^{r,1}},\\
 \|J_{i,k}^{\varepsilon}[G_{N,k+1}^{\varepsilon,n}(t)]\|_{W^{r,1}}
 &\le C_{r,k,\varepsilon}
        \|G_{N,k+1}^{\varepsilon,n}(t)\|_{W^{r,1}}.
 \end{aligned}
\end{equation}
The initial densities are compactly supported in the noncompact
variables.  Applying the same argument with polynomial weights on both
sides, and using the finite moments of the Brownian maximum, gives finite
weighted derivative norms at fixed $n$.  The equations and Sobolev
embedding then give the local smoothness of the regularized errors and
complete fluxes.

We first perform the calculation for
$U_{N,k}^{\varepsilon,n,\delta}$.  The remainder flux contains the additional term
$(E-E^\delta)U_{N,k}^{\varepsilon,n,\delta}$.  Uniformly in $N,k,n,\varepsilon,\delta$ and $0<\lambda\le\lambda_*$,
the pointwise estimate of \Cref{lem:TV-U-pointwise} gives
\[
 \sum_{i=1}^k
 \|(E^\delta-E)(t,x_i)U_{N,k}^{\varepsilon,n,\delta}(t)\|_
 {L^q(e^{\lambda e_{N,k}^{\varepsilon}}\varpi_k)}^q
 \le C_T^k\|E^\delta-E\|_\infty^q.
\]
For fixed $N,k,\varepsilon,n$, the coefficients have a common Lipschitz
bound.  The flow and inverse-flow argument in the proof of
\Cref{lem:common-L1-stability}, together with the volume preservation
proved in \Cref{prop:common-regularization-removal}, gives
\[
 U_{N,k}^{\varepsilon,n,\delta}\longrightarrow
 U_{N,k}^{\varepsilon,n}
 \quad\text{in }C([0,T];L^1),
 \qquad \delta\downarrow0.
\]
Set $D_k^\delta:=U_{N,k}^{\varepsilon,n,\delta}
 -U_{N,k}^{\varepsilon,n}$ and
$W_k(t):=e^{\lambda(t)e_{N,k}^\varepsilon}\varpi_k$.
The pointwise bounds give
\[
 |D_k^\delta|\le2A_T^ke^{-\beta_Te_{N,k}^\varepsilon}
                 \prod_{j=1}^k\langle x_j\rangle^{-M_*}.
\]
Since $\lambda(t)\le\lambda_0<(q-1)\beta_T$ and, on $\R^d$,
$\alpha<(q-1)M_*$, interpolation yields
\begin{equation}\label{eq:fixed-N-field-weighted-convergence}
 \sup_{t\le T}\|D_k^\delta(t)\|_{L^q(W_k(t))}^q
 \le (2A_T^k)^{q-1}\sup_{t\le T}\|D_k^\delta(t)\|_1
 \longrightarrow0.
\end{equation}
The spatial factors are omitted on the torus.  The complete flux in the
error equation is
\[
 \mathcal F_i^\delta
 =\frac{N-k}{N}\left(
 J_{i,k}^\varepsilon[G_{N,k+1}^{\varepsilon,n}]
 -E^\delta(t,x_i)U_{N,k}^{\varepsilon,n,\delta}\right).
\]
For the flux $\mathcal F_i$ in \eqref{eq:regularized-error-flux},
\eqref{eq:fixed-N-field-weighted-convergence} gives
\begin{align}\label{eq:fixed-N-field-flux-convergence}
 \sup_{t\le T}\|\mathcal F_i^\delta(t)-\mathcal F_i(t)\|_{L^q(W_k(t))}
 &\le \|E^\delta\|_\infty
       \sup_{t\le T}\|D_k^\delta(t)\|_{L^q(W_k(t))}\notag\\
 &\quad+\|E^\delta-E\|_\infty
       \sup_{t\le T}\|U_{N,k}^{\varepsilon,n}(t)\|_{L^q(W_k(t))}
 \longrightarrow0.
\end{align}
For $H_k^\delta=G_{N,k}^{\varepsilon,n}-U_{N,k}^{\varepsilon,n,\delta}$
and $H_k=G_{N,k}^{\varepsilon,n}-U_{N,k}^{\varepsilon,n}$, the pointwise
bounds and $\lambda_0<(q-1)\beta_T$ give
\[
 \sup_{t\le T}\left|\int
       (|H_k^\delta|^q-|H_k|^q)W_k(t)\dd Z_k\right|
 \le C_{N,k,T}\sup_{t\le T}\|D_k^\delta(t)\|_1\longrightarrow0.
\]
Here $C_{N,k,T}$ is independent of $n,\varepsilon,\delta$; we used
$\bigl||a|^q-|b|^q\bigr|\le q(|a|+|b|)^{q-1}|a-b|$.
Consequently, we first pass $\delta\downarrow0$ in the integrated form
of \eqref{eq:common-flux-energy}, with the complete flux
$\mathcal F_i^\delta$ on its right-hand side.  After this limit, we use the flux decomposition
\eqref{eq:regularized-error-flux} and apply
\eqref{eq:TV-regularized-residual-bound-smooth} for the original
prescribed field $E$.
In the cutoff calculation we keep $\delta>0$ fixed and suppress it from the
notation; the limit $\delta\downarrow0$ is taken only after removing
that cutoff.

The remaining integrations by parts are justified by a compact
phase-space cutoff, as in the weighted BBGKY estimate of
\cite[Section~3.2]{BJS2025}.  Write $W=W_k(t)$, with $W_k$ as
defined above, and choose a cutoff $0\le\chi_R\le1$ in the noncompact phase-space
variables, equal to one on $B_R$, supported in $B_{2R}$, with
$|\nabla\chi_R|\le C/R$ and $|D^2\chi_R|\le C/R^2$.  On $\T^d$ this
cutoff depends only on $V_k$.  Testing with $q|H|^{q-2}HW\chi_R$ is
legitimate.  For noninteger $q\ge2$, we first use the derivative of
$\eta_a(H)=(H^2+a^2)^{q/2}-a^q$ in place of $q|H|^{q-2}H$ and let
$a\downarrow0$ at fixed cutoff.  The chain rule, local smoothness, and
dominated convergence give the same calculation.  With $\mathcal F_i$
denoting the fluxes in
\Cref{lem:common-weighted-flux-energy}, the additional terms are
\begin{align*}
 \int |H|^qW L_{N,k}^\varepsilon\chi_R,\quad \sigma\int |H|^q\bigl(W\Delta_{V_k}\chi_R
             +2\nabla_{V_k}W\cdot\nabla_{V_k}\chi_R\bigr),\quad q\sum_{i=1}^k\int |H|^{q-2}H W\mathcal F_i\cdot\nabla_{v_i}\chi_R.
\end{align*}
Let $\mathcal A_R$ be the support of the cutoff derivatives.  At fixed
$N,k,\varepsilon$, the first two integrands are bounded in absolute value by
\[
 C_{N,k,\varepsilon}\ind_{\mathcal A_R}
 |H|^qW\left(\frac{1+|V_k|}{R}+\frac1{R^2}\right).
\]
The pointwise estimates and \eqref{eq:common-weight-compatibility} give
\[
 \begin{split}
 (1+|V_k|^2)|H(t,Z_k)|^qW(t,Z_k)\le C_{N,k,T}(1+|V_k|^2)
 e^{-(q\beta_T-\lambda_0)\sum_{i=1}^k(1+|v_i|^2)}
 \prod_{i=1}^k\langle x_i\rangle^{\alpha-qM_*}
 \in L^1(\dd Z_k),
 \end{split}
\]
uniformly in $t$, with the spatial product omitted on $\T^d$.
This verifies the moment hypothesis of
\Cref{lem:common-weighted-flux-energy} and makes the transport and
diffusion boundary terms tend to zero after time integration.  For the flux terms,
weighted H\"older gives
\begin{align*}
 \int_0^t\!\int |H|^{q-1}W|\mathcal F_i|\,
                  |\nabla_{v_i}\chi_R|\le\frac CR
 \left(\int_0^t\!\int_{\mathcal A_R}|H|^qW\right)^{(q-1)/q}
 \left(\int_0^t\!\int_{\mathcal A_R}|\mathcal F_i|^qW\right)^{1/q}
 \longrightarrow0.
\end{align*}
At fixed $N$, the pointwise bounds and
\Cref{lem:common-weighted-holder}, applied to the complete flux
displayed above, give the finiteness of both weighted integrals.  The
nonnegative dissipation terms pass to the limit by Fatou's lemma.
For these approximations, $H(0)=0$, $H\in C([0,T];L^1)$, and
$\lambda_0<(q-1)\beta_T$ give
\[
 \int |H(t)|^qW(t)
 \le\bigl\||H(t)|^{q-1}W(t)\bigr\|_\infty\|H(t)\|_1
 \le C_{N,k,T}\|H(t)\|_1\longrightarrow0
 \qquad(t\downarrow0).
\]
For $s,t\in[0,T]$, the pointwise bounds also give
\begin{align*}
 \left|\int |H(t)|^qW(t)-\int |H(s)|^qW(s)\right|
 \le C_{N,k,T}\|H(t)-H(s)\|_1+\int |H(s)|^q\varpi_k
 \left|e^{\lambda(t)e_{N,k}^\varepsilon}
       -e^{\lambda(s)e_{N,k}^\varepsilon}\right|.
\end{align*}
The last term tends to zero as $s\to t$ by the same integrable
majorant and dominated convergence.  Hence the weighted energy is
continuous, and the time cutoffs can be removed at both endpoints.
Thus \eqref{eq:common-flux-energy} holds in integrated form for every
$t\in[0,T]$.

\subsection{Limit of the integrated energy inequality}

\begin{lemma}[Passage to the integrated energy inequality]
\label{lem:TV-fixed-N-passage}
Fix $N$ and use the approximations of \Cref{lem:common-smooth-initial-data},
with $\xi_{n,\varepsilon}$ as above and $\Delta_N$ defined in
\eqref{eq:comparison-initial-error}.
Suppose that, for every $1\le k<N$, the regularized errors satisfy
\eqref{eq:fixed-N-regularized-integrated}, with
$H_{N,k}^{\varepsilon,n}(0)=0$ for every $1\le k\le N$ and
$H_{N,N}^{\varepsilon,n}=0$.  If one first lets $n\to\infty$ for fixed
$\varepsilon$ and then lets $\varepsilon\downarrow0$, the limiting
functions satisfy \eqref{eq:TV-error-energy} for every $1\le k<N$ and
every $t\in[0,T]$, with $H_{N,k}(0)=0$ and $H_{N,N}=0$.
\end{lemma}

\begin{proof}
The identities \eqref{eq:fixed-N-error-identities} follow from the
common initial data and \eqref{eq:fixed-N-smooth-level-N}.
We use the contraction estimates \eqref{eq:fixed-N-L1-approx}.
In particular, these contraction estimates give
\begin{equation}\label{eq:fixed-N-n-limit}
 H_{N,k}^{\varepsilon,n}
 \longrightarrow H_{N,k}^{\varepsilon}
 \quad\text{in }C([0,T];L^1),
 \qquad n\to\infty.
\end{equation}
Apply \Cref{prop:common-regularization-removal} to $G_N$ and
$U_{N,k}$, and project the former onto its first $k$ coordinates.
Contraction of total variation under projection gives
\begin{equation}\label{eq:fixed-N-epsilon-limit}
 H_{N,k}^{\varepsilon}
 \longrightarrow H_{N,k}
 \quad\text{in }C([0,T];L^1),
 \qquad \varepsilon\downarrow0.
\end{equation}
For $0<\lambda\le\lambda_0$, the pointwise estimates give
\[
 \max\bigl\{|H_{N,k}^{\varepsilon,n}|^q,
                 |H_{N,k}^{\varepsilon}|^q\bigr\}
 e^{\lambda e_{N,k}^{\varepsilon}}\varpi_k
 \le C_{N,k,T}e^{-(q\beta_T-\lambda)e_{N,k}^{\varepsilon}}
       \prod_{j=1}^k\langle x_j\rangle^{\alpha-qM_*},
\]
where one may take $C_{N,k,T}=2^{q-1}(B_T^{qN}+A_T^{qk})$.
The same estimate holds for $|H_{N,k}|^qe^{\lambda e_{N,k}}\varpi_k$ with
the singular energy.  Both energies dominate
$\sum_{i=1}^k(1+|v_i|^2)$, so all three weighted densities are bounded by
\begin{equation}\label{eq:fixed-N-common-dominator}
 C_{N,k,T}
 e^{-(q\beta_T-\lambda_0)\sum_{i=1}^k(1+|v_i|^2)}
 \prod_{j=1}^{k}\langle x_j\rangle^{\alpha-qM_*}\in L^1(\dd Z_k).
\end{equation}
This bound is uniform in $t,n,\varepsilon$ at fixed $N$.
For each fixed regularization, the weighted energy is continuous
in time by the argument of Section~\ref{subsec:common-cutoff-testing},
using the same pointwise bounds and $C_tL^1$ continuity.
Since $\phi_\varepsilon=\phi$ outside $B_{2\varepsilon}$, the energies satisfy
\begin{equation}\label{eq:fixed-N-energy-convergence}
 e_{N,k}^{\varepsilon}(Z_k)\longrightarrow e_{N,k}(Z_k)
 \qquad\text{as }\varepsilon\downarrow0
\end{equation}
for every configuration away from the collision diagonals.

We now prove convergence of the weighted energies uniformly in time.
In the following estimates $C_{N,k,T}$ may increase, but remains
independent of $t,n,\varepsilon$.
We use the power-difference inequality from
Section~\ref{subsec:common-cutoff-testing}.
The parameter choices give $\lambda_0<(q-1)\beta_T$ and, on
$\R^d$, $\alpha<(q-1)M_*$.  Thus the pointwise bounds imply
\[
 (|H_{N,k}^{\varepsilon,n}|+|H_{N,k}^{\varepsilon}|)^{q-1}
 e^{\lambda(t)e_{N,k}^{\varepsilon}}\varpi_k\le C_{N,k,T}.
\]
Together with \eqref{eq:fixed-N-L1-approx}, this gives
\begin{equation}\label{eq:fixed-N-energy-n-passage}
 \sup_{t\le T}|Y_{N,k}^{\varepsilon,n}(t)-Y_{N,k}^{\varepsilon}(t)|
 \le C_{N,k,T}\xi_{n,\varepsilon}\longrightarrow0.
\end{equation}

For the force regularization, set
\[
 \mathcal C_\varepsilon
 =\{Z_k:\min_{i<j}d_{\mathcal D}(x_i,x_j)\le2\varepsilon\},
 \qquad \mathcal C_\varepsilon=\varnothing\quad\text{if }k=1.
\]
On its complement, $e_{N,k}^{\varepsilon}=e_{N,k}$, so the same
pointwise estimate gives
\[
 \int_{\mathcal C_\varepsilon^c}
 \bigl||H_{N,k}^{\varepsilon}(t)|^q-|H_{N,k}(t)|^q\bigr|
 e^{\lambda(t)e_{N,k}}\varpi_k
 \le C_{N,k,T}\|H_{N,k}^{\varepsilon}(t)-H_{N,k}(t)\|_1.
\]
Let $D_{N,k,T}(Z_k)$ denote the integrable function displayed in
\eqref{eq:fixed-N-common-dominator}.  By
\eqref{eq:common-weight-compatibility}, $qM_*-\alpha>d$.
For $i\ne j$ on $\R^d$, integration in $x_i$ gives
\begin{align*}
\int_{|x_i-x_j|\le2\varepsilon}
       \prod_{\ell=1}^k\langle x_\ell\rangle^{\alpha-qM_*}\dd X_k\le |B_{2\varepsilon}|
       \left(\int_{\R^d}\langle x\rangle^{\alpha-qM_*}\dd x\right)^{k-1}.
\end{align*}
On $\T^d$, the corresponding bound is $C_d\varepsilon^d$.
After integrating the Gaussian velocity factors and summing over pairs,
we obtain
\begin{equation}\label{eq:fixed-N-collision-neighborhood}
 \int_{\mathcal C_\varepsilon}D_{N,k,T}(Z_k)\dd Z_k
 \le C_{N,k,T}\varepsilon^d.
\end{equation}
Both weighted densities on $\mathcal C_\varepsilon$ are bounded by
$D_{N,k,T}(Z_k)$.  Combining the estimate on
$\mathcal C_\varepsilon^c$ with \eqref{eq:fixed-N-collision-neighborhood}
gives
\begin{equation}\label{eq:fixed-N-energy-epsilon-passage}
 \sup_{t\le T}|Y_{N,k}^\varepsilon(t)-Y_{N,k}(t)|
 \le C_{N,k,T}\left(
     \sup_{t\le T}\|H_{N,k}^\varepsilon(t)-H_{N,k}(t)\|_1
     +\varepsilon^d\right)\longrightarrow0
\end{equation}
by \eqref{eq:fixed-N-epsilon-limit}, for fixed $N,k$.
The constants in this convergence argument may depend on $N$: the
approximation parameters are eliminated before the hierarchy is
iterated.  The coefficients $A,L,D_T$ in
\eqref{eq:TV-error-energy} retain the uniformity in $N,k$
established before taking these limits.
Consequently, for every $a\in L^\infty(0,T)$, we have
\[
 \sup_{t\le T}\left|
 \int_0^ta(r)\bigl(Y_{N,k}^{\varepsilon,n}-Y_{N,k}^{\varepsilon}\bigr)(r)\dd r
 \right|\to0,
\]
and similarly for $Y_{N,k}^{\varepsilon}-Y_{N,k}$.  We use this with
$a=kA$ and $a=kL\frac{N-k}{N}$.

The source converges to $D_T^k\Delta_N$ because
$\xi_{n,\varepsilon}\to0$ for fixed $\varepsilon$ and
$\varepsilon^{q(d-s-1)}\to0$.  Passing first $n\to\infty$ and then
$\varepsilon\downarrow0$ in
\eqref{eq:fixed-N-regularized-integrated} proves
\eqref{eq:TV-error-energy} for every $t$.  The uniform energy
convergence also gives $Y_{N,k}\in C([0,T])$.  The $C_tL^1$ convergence
preserves both identities in \eqref{eq:fixed-N-error-identities}.
\end{proof}

\section[Weighted estimates for the Vlasov--Riesz--Fokker--Planck equation]{Weighted estimates for the
Vlasov--Riesz--Fokker--Planck equation}\label{app:whole-space-classical}

Choi and Jeong \cite{ChoiJeong2024} establish local well-posedness for
Vlasov--Riesz equations and discuss a corresponding model with a linear
Fokker--Planck operator in velocity.  Related well-posedness results are given in
\cite{WangZhang2025,ChoiJeongKang2024}.  For the equation \eqref{eq:intro-riesz-vpfp}, we establish the weighted
$C^m$ estimates and the local construction required in
\Cref{prop:common-riesz-classical}.  For a prescribed
field, the linear kinetic equation is represented by the stochastic flow of
the corresponding SDE; see \cite[Sections~4.5--4.6]{Kunita1990}.

\begin{proposition}[Local Riesz solutions with weighted bounds]
\label{prop:common-riesz-classical}
Let $d\ge2$, $0<s<d-1$, $\sigma>0$ and $m\ge3$ be an integer.  Assume
$f_0\ge0$ has unit mass.  On $\T^d$, suppose
$f_0\in C^m_{\T,a_0}$ for some $a_0>0$ and fix $0<a_1<a_0$.
On $\R^d$, suppose $f_0\in C^m_{M,a_0}$ with $M>d$ and fix
$d<M_1<M$, $0<a_1<a_0$.
There exist $T_f>0$ and a nonnegative classical solution $f$ of
\eqref{eq:intro-riesz-vpfp} on $[0,T_f]$ with unit mass such that
\begin{equation}\label{eq:common-riesz-regularity}
 f\in
 \begin{cases}
 C([0,T_f];C^{m-1}_{\T,a_1}),&\mathcal D=\T^d,\\
 C([0,T_f];C^{m-1}_{M_1,a_1}),&\mathcal D=\R^d.
 \end{cases}
\end{equation}
Moreover, the solution satisfies
\eqref{eq:limiting-field-regularity} and
\eqref{eq:intro-velocity-derivative} with $T=T_f$.
Its pointwise decay is given by
\eqref{eq:initial-pointwise-torus} or \eqref{eq:initial-pointwise-whole},
uniformly for $0\le t\le T_f$, with $f_0$ replaced by $f(t)$,
$C_0$ by a suitable constant $C$, $a_0$ by $a_1$ and, on $\R^d$,
$M$ by $M_1$.
The interval and constants depend only on the displayed exponents, the
initial norm, and the fixed model parameters.
\end{proposition}

\subsection{Stochastic flow for a prescribed field}

Fix an integer $r\ge1$ and $0<T\le1$, and assume
\begin{equation}\label{eq:prescribed-field-regularity}
 E\in L^\infty(0,T;C_b^r(\mathcal D;\R^d))
       \cap C([0,T];C_b^{r-1}(\mathcal D;\R^d)).
\end{equation}
For each Brownian path, consider
\[
 \frac{\dd X_t}{\dd t}=V_t,\qquad
 \frac{\dd}{\dd t}\bigl(V_t-\sqrt{2\sigma}B_t\bigr)=E(t,X_t),
 \qquad (X_0,V_0)=(x,v).
\]
After subtracting the Brownian path from the velocity, this is a
Carath\'eodory equation with globally Lipschitz spatial coefficients
and bounded spatial derivatives of orders $1,\ldots,r$.
The forward and backward equations define a stochastic flow
$\Phi^E_{0,t}$ of $C^r$-diffeomorphisms; cf.
\cite[Sections~4.5--4.6]{Kunita1990}.  Its Jacobian satisfies
\[
 \frac{\dd}{\dd t}\det D\Phi^E_{0,t}
 =\Div_{x,v}(v,E(t,x))\det D\Phi^E_{0,t}=0,
 \qquad \det D\Phi^E_{0,0}=1.
\]
Thus the flow preserves Lebesgue measure.  All flow derivatives below
are taken with respect to the phase-space argument.

\begin{lemma}\label{lem:stochastic-flow-estimates}
If $\|E\|_{L_t^\infty C_b^r}\le R_E$, then for $0\le t\le T$ we have
\begin{equation}\label{eq:flow-small-time-derivatives}
 \|D\Phi^E_{0,t}-I\|_\infty
 +\|D(\Phi^E_{0,t})^{-1}-I\|_\infty
 +\sum_{j=2}^r\left(
  \|D^j\Phi^E_{0,t}\|_\infty
  +\|D^j(\Phi^E_{0,t})^{-1}\|_\infty\right)\le Ct.
\end{equation}
For two fields with the same $C_b^r$ bound, the flows driven by the same
Brownian motion satisfy
\begin{equation}\label{eq:flow-field-stability}
 \sup_{t\le T}\sum_{j=0}^{r-1}
  \|D^j(\Phi^E_{0,t}-\Phi^{\widetilde E}_{0,t})\|_\infty
 \le CT\|E-\widetilde E\|_{L_t^\infty C_b^{r-1}},
\end{equation}
with the same estimate for the inverse flows.  The constants depend only
on $d,r,R_E$.  Periodic flows are compared through their lifts.
\end{lemma}

\begin{proof}
The first variational equation is
\[
 \frac{\dd}{\dd t}D\Phi^E_{0,t}
 =\begin{pmatrix}0&I_d\\ \nabla E(t,X_t)&0\end{pmatrix}D\Phi^E_{0,t},
 \qquad D\Phi^E_{0,0}=I.
\]
For $2\le j\le r$, differentiating the characteristic equations $j$
times gives
\begin{equation}\label{eq:flow-higher-derivatives}
 \begin{aligned}
 \partial_tD^jX_t&=D^jV_t,\\
 \partial_tD^jV_t
 &=\nabla E(t,X_t)D^jX_t
   +\sum_{\substack{2\le\ell\le j\\j_1+\cdots+j_\ell=j,\ j_a\ge1}}
       c_{j_1,\ldots,j_\ell}
       D^\ell E(t,X_t)[D^{j_1}X_t,\ldots,D^{j_\ell}X_t],\\
 D^jX_0&=D^jV_0=0.
 \end{aligned}
\end{equation}
The constants in this finite sum depend only on $j$; tensor products
are symmetrized.  Since $j_a<j$ in every summand, induction and
Gronwall's inequality give \eqref{eq:flow-small-time-derivatives}.
Differentiating $\Phi^E_{0,t}\circ(\Phi^E_{0,t})^{-1}=\mathrm{Id}$
gives the inverse estimates.
For two synchronously driven flows, put
$\delta X_t=X_t^E-X_t^{\widetilde E}$ and
$\delta V_t=V_t^E-V_t^{\widetilde E}$.  The Brownian terms cancel, so
\[
 |\delta X_t|+|\delta V_t|
 \le C\int_0^t(|\delta X_s|+|\delta V_s|)\dd s
       +\int_0^t\|E(s)-\widetilde E(s)\|_\infty\dd s.
\]
For $1\le j\le r-1$, subtraction of the variational equations gives
\[
 \|D^j(\Phi^E_{0,t}-\Phi^{\widetilde E}_{0,t})\|_\infty
 \le C\int_0^t\left[
 \sum_{\ell=0}^j
 \|D^\ell(\Phi^E_{0,s}-\Phi^{\widetilde E}_{0,s})\|_\infty
 +\|E(s)-\widetilde E(s)\|_{C_b^j}\right]\dd s.
\]
The change of evaluation point in $D^jE$ uses $D^{j+1}E$.
Gronwall's inequality proves \eqref{eq:flow-field-stability}.
Differentiating the inverse identities gives the same estimate for
the inverse flows.
\end{proof}

\subsection{Weighted estimates for a prescribed field}

We use the unified notation $\mathcal C^j_{0,a}=C^j_{\T,a}$ on
$\T^d$ and $\mathcal C^j_{M,a}=C^j_{M,a}$ on $\R^d$, with the norms
defined in \eqref{eq:periodic-classical-data-space} and
\eqref{eq:whole-space-classical-data-norm}.  The corresponding weight is
\[
 w^{\mathcal D}_{M,a}(x,v)=
 \begin{cases}
 e^{a|v|^2},&\mathcal D=\T^d,\ M=0,\\
 \langle x\rangle^Me^{a|v|^2},&\mathcal D=\R^d,\ M>d.
 \end{cases}
\]
We fix $0<a'<a$ and, on $\R^d$, $d<M'<M$; on $\T^d$ set $M'=M=0$.

\begin{lemma}[Transport of spatial and Gaussian weights]
\label{lem:whole-space-weight-flow}
Let $0<T\le1$ and assume \eqref{eq:prescribed-field-regularity}
with $r=1$ and $\|E\|_{L_t^\infty C_b^1}\le R_E$.
With $B_T^*=\sup_{r\le T}|B_r|$, there is a random variable $\Theta_T$, depending
only on $B_T^*$ and the fixed bounds, such that
\[
 \sup_{t\le T,z_0}
 \frac{w^{\mathcal D}_{M',a'}(\Phi^E_{0,t}z_0)}
      {w^{\mathcal D}_{M,a}(z_0)}
 \le \Theta_T.
\]
For every finite $p\ge1$, $\Theta_T$ has finite $p$th moment when $T$ is
sufficiently small, and $\Theta_T\to1$ almost surely and in $L^1$ as
$T\downarrow0$.
\end{lemma}

\begin{proof}
Integrating the characteristic equations, for $0\le t\le T$ we have
\begin{equation}\label{eq:flow-displacements}
 |V_t-v_0|\le\delta_T,
 \qquad
 |X_t-x_0-tv_0|\le\eta_T,
\end{equation}
where
\[
 \delta_T=R_ET+\sqrt{2\sigma}B_T^*,
 \qquad
 \eta_T=\frac12R_ET^2+\sqrt{2\sigma}TB_T^*.
\]
We first control the velocity weight.  Choose $\vartheta>0$ with
$a'(1+\vartheta)<a$ and set $c_a=a-a'(1+\vartheta)$.  Then
\[
 a'|V_t|^2-a|v_0|^2
 \le -c_a|v_0|^2+C\delta_T^2.
\]
On $\R^d$ we also estimate the spatial weight.  Since
$|\nabla\log\langle x\rangle|\le1$, we obtain
\[
 \log\frac{\langle X_t\rangle}{\langle x_0\rangle}
 \le T|v_0|+\eta_T,
 \qquad
 M'T|v_0|\le\frac{c_a}{2}|v_0|^2+CT^2.
\]
In both geometries, this gives
\[
 \frac{w^{\mathcal D}_{M',a'}(\Phi^E_{0,t}z_0)}
      {w^{\mathcal D}_{M,a}(z_0)}
 \le \exp\bigl(C(T+B_T^*)^2\bigr)=:\Theta_T,
\]
with constants depending only on the fixed exponents, $R_E,d$, and $\sigma$.
The reflection principle and Brownian scaling give
\[
 \Pp(B_T^*>R)\le C_d e^{-c_dR^2/T};
\]
see, for instance, \cite[Sections~2.6 and~2.8.A]{KaratzasShreve1991}.  For each finite $p\ge1$, choose $T_p>0$ so small that
$2pCT_p<c_d$.  Then
\[
 \E\Theta_{T_p}^{\,p}
 \le e^{2pCT_p^2}\E e^{2pC(B_{T_p}^*)^2}<\infty,
 \qquad 1\le\Theta_T\le\Theta_{T_p}\quad(0<T\le T_p).
\]
Since $B_T^*\to0$ almost surely, dominated convergence gives
$\Theta_T\to1$ in $L^p$, and in particular in $L^1$.
\end{proof}

\begin{lemma}
\label{lem:whole-space-prescribed-linear}
Let $r\ge3$ be an integer and suppose
\eqref{eq:prescribed-field-regularity} holds, with
$\|E\|_{L_t^\infty C_b^r}\le R_E$.
If $g_0\ge0$ belongs to $\mathcal C^r_{M,a}$, then, for some
$0<T_0\le T$ sufficiently small, the equation
\begin{equation}\label{eq:whole-space-prescribed-linear}
 \partial_tg+v\cdot\nabla_xg+E\cdot\nabla_vg=\sigma\Delta_vg,
 \qquad g(0)=g_0,
\end{equation}
has a nonnegative classical solution that preserves mass and satisfies
\begin{align}
 \sup_{t\le\tau}\|g(t)\|_{\mathcal C^r_{M',a'}}
 &\le(1+\omega(\tau;R_E))\|g_0\|_{\mathcal C^r_{M,a}},
 \qquad 0<\tau\le T_0,
 \label{eq:whole-space-prescribed-small-time-bound}\\
 g&\in C([0,T_0];\mathcal C^{r-1}_{M',a'}),
 \qquad \partial_tg\in C([0,T_0];C_b^{r-3}),
 \label{eq:whole-space-prescribed-time-regularity}
\end{align}
where $\omega(\tau;R_E)\to0$ as $\tau\downarrow0$.
For $T_0$ sufficiently small, the solutions with common initial datum
$g_0$ and fields $E,\widetilde E$ having the same $C_b^r$ bound satisfy
\begin{equation}\label{eq:whole-space-interval-stability}
 \|g_E-g_{\widetilde E}\|_{C([0,T_0];\mathcal C^{r-1}_{M',a'})}
 \le CT_0\|g_0\|_{\mathcal C^r_{M,a}}
           \|E-\widetilde E\|_{L^\infty(0,T_0;C_b^{r-1})}.
\end{equation}
The constants and the choice of $T_0$ depend only on the field bound,
$d,r,\sigma$ and the weight exponents.
\end{lemma}

\begin{proof}
\emph{Representation and weighted bounds.}
Let $\Psi_t^E=(\Phi^E_{0,t})^{-1}$ and define
\[
 g_E(t,z)=\E[g_0(\Psi_t^Ez)].
\]
Volume preservation gives positivity and conservation of mass, while
It\^o's formula for compactly supported test functions gives
\eqref{eq:whole-space-prescribed-linear} in distributions.
For $|\beta|\le r$, the Fa\`a di Bruno formula and
\eqref{eq:flow-small-time-derivatives} yield
\[
 D^\beta(g_0\circ\Psi_t^E)
 =(D^\beta g_0)(\Psi_t^E)+\mathcal R_\beta(t),
 \qquad
 |\mathcal R_\beta(t)|
 \le Ct\sum_{|\gamma|\le|\beta|}|D^\gamma g_0(\Psi_t^E)|.
\]
Using \Cref{lem:whole-space-weight-flow} and dominated convergence, we obtain
\[
 \sup_{t\le\tau}\|g_E(t)\|_{\mathcal C^r_{M',a'}}
 \le (1+C\tau)\E\Theta_\tau\,
      \|g_0\|_{\mathcal C^r_{M,a}}
 =:(1+\omega(\tau;R_E))\|g_0\|_{\mathcal C^r_{M,a}}.
\]
Here $C$ depends only on $d,r,R_E$, and
$\omega(\tau;R_E)\to0$ as $\tau\downarrow0$.

\emph{Continuity in time.}
Choose
$a'<a''<a$ and, on $\R^d$, $d<M'<M''<M$; on the torus set $M''=0$.
For every $|\gamma|\le r-1$, the initial bound gives
\[
 w^{\mathcal D}_{M'',a''}|D^\gamma g_0|
 \le\|g_0\|_{\mathcal C^r_{M,a}}
       e^{-(a-a'')|v|^2}
 \begin{cases}
  1,&\mathcal D=\T^d,\\
  \langle x\rangle^{-(M-M'')},&\mathcal D=\R^d.
 \end{cases}
\]
Cutoff and mollification, periodic in position on $\T^d$, give
nonnegative smooth functions $g_0^{(j)}$, compactly supported in the
noncompact variables, such that
\begin{equation}\label{eq:weighted-initial-approximation}
 \|g_0^{(j)}-g_0\|_{\mathcal C^{r-1}_{M'',a''}}\longrightarrow0.
\end{equation}
The displayed decay controls all cutoff terms; local uniform
continuity of the derivatives through order $r-1$ gives the
mollification limit.

Write $\mathcal S_E(t)h=\E[h\circ\Psi_t^E]$.
The chain rule and \Cref{lem:whole-space-weight-flow}, with input
exponents $M'',a''$ and $T_0$ sufficiently small, give
\begin{equation}\label{eq:weighted-linear-approximation-bound}
 \sup_{t\le T_0}\|\mathcal S_E(t)h\|_{\mathcal C^{r-1}_{M',a'}}
 \le C\|h\|_{\mathcal C^{r-1}_{M'',a''}}.
\end{equation}
This estimate also holds for signed $h$.  Thus
$g^{(j)}(t)=\mathcal S_E(t)g_0^{(j)}$ satisfies
\[
 \sup_{t\le T_0}\|g^{(j)}(t)-g_E(t)\|_{\mathcal C^{r-1}_{M',a'}}
 \le C\|g_0^{(j)}-g_0\|_{\mathcal C^{r-1}_{M'',a''}}\longrightarrow0.
\]
For each Brownian path, the transported supports lie in a common
compact set on $[0,T_0]$.  The compositions
$g_0^{(j)}\circ\Psi_t^E$ are therefore continuous in
$\mathcal C^{r-1}_{M',a'}$, and
\[
 \sup_{t\le T_0}
 \|g_0^{(j)}\circ\Psi_t^E\|_{\mathcal C^{r-1}_{M',a'}}
 \le C\Theta_{T_0}\|g_0^{(j)}\|_{\mathcal C^{r-1}_{M'',a''}},
 \qquad \E\Theta_{T_0}<\infty,
\]
where \Cref{lem:whole-space-weight-flow} is applied with input exponents
$M'',a''$.  It follows that
\[
 \|g^{(j)}(t)-g^{(j)}(s)\|_{\mathcal C^{r-1}_{M',a'}}
 \le\E\|g_0^{(j)}\circ\Psi_t^E-g_0^{(j)}\circ\Psi_s^E
          \|_{\mathcal C^{r-1}_{M',a'}}\longrightarrow0
 \quad\text{as }t\to s.
\]
The uniform approximation proves the same assertion for $g_E$.

The map
$-v\cdot\nabla_xg_E-E\cdot\nabla_vg_E+\sigma\Delta_vg_E$
is continuous with values in $C_b^{r-3}$: multiplication by $v$ is
controlled by the Gaussian weight, and the other terms use at most
$r-1$ derivatives of $g_E$.  Its distributional identity with
$\partial_tg_E$ gives the asserted time regularity.  The weighted bounds
also imply $g_E\in C_tL^1$ and
$\nabla_vg_E\in L^1(0,T_0;L^1)$.  Uniqueness in this class follows from
\Cref{lem:common-L1-stability} with $k=N=1$ and zero source.

\emph{Dependence on the field.}
Drive the two flows by the same Brownian motion and set
\[
 y_\theta=(1-\theta)\Psi_t^Ez+\theta\Psi_t^{\widetilde E}z,
 \qquad 0\le\theta\le1.
\]
On $\T^d$ use periodic lifts.  For $z=(x,v)$ and
$\Psi_t^Fz=(x_0^F,v_0^F)$, $F=E,\widetilde E$,
\eqref{eq:flow-displacements} holds with $T=T_0$.
The displacement vectors are affine in $(x_0^F,v_0^F)$, so the same
bounds hold at $y_\theta$.  The weight calculation in
\Cref{lem:whole-space-weight-flow} gives
\[
 \sup_{t\le T_0,z,\theta}
 \frac{w^{\mathcal D}_{M',a'}(z)}{w^{\mathcal D}_{M,a}(y_\theta)}
 \le\Theta_{T_0},\qquad \E\Theta_{T_0}<\infty.
\]
The mean-value theorem uses derivatives of $g_0$ through order $r$.
The Fa\`a di Bruno formula, \eqref{eq:flow-small-time-derivatives} and
\eqref{eq:flow-field-stability} then yield
\[
 \sup_{t\le T_0}
 \|g_0\circ\Psi_t^E-g_0\circ\Psi_t^{\widetilde E}\|_
       {\mathcal C^{r-1}_{M',a'}}
 \le CT_0\Theta_{T_0}\|g_0\|_{\mathcal C^r_{M,a}}
                  \|E-\widetilde E\|_{L_t^\infty C_b^{r-1}}.
\]
Taking expectations proves \eqref{eq:whole-space-interval-stability}.
\end{proof}

\subsection{Construction of the nonlinear solution}

\begin{proof}[Proof of \Cref{prop:common-riesz-classical}]
We set $r=m\ge3$.  On $\R^d$ we fix $d<M_1<M$ and $0<a_1<a_0$; on $\T^d$ set
$M=M_1=0$.  For $0\le j\le r$, the weighted bounds imply
\[
 \max_{|\xi|\le j}\left(
 \|\partial_x^\xi\rho_h\|_\infty
 +\mathbf1_{\{\mathcal D=\R^d\}}\|\partial_x^\xi\rho_h\|_1\right)
 \le C\|h\|_{\mathcal C^j_{M_1,a_1}}.
\]
Since $0<s<d-1$, one has $s+1<d$, hence
$K\in L^1_{\mathrm{loc}}$.  On $\T^d$ this gives $K\in L^1$.  On $\R^d$,
split $K=K_<+K_>$ with $K_<\in L_c^1$ and
$K_>\in C_b^\infty$.  For $|\xi|\le j$, convolution yields
\[
 \|\partial_x^\xi(K*\rho_h)\|_\infty
 \le\begin{cases}
 \|K\|_1\|\partial_x^\xi\rho_h\|_\infty,&\mathcal D=\T^d,\\
 \|K_<\|_1\|\partial_x^\xi\rho_h\|_\infty
 +\|K_>\|_\infty\|\partial_x^\xi\rho_h\|_1,&\mathcal D=\R^d.
 \end{cases}
\]
For $h\in\mathcal C^j_{M_1,a_1}$, these convolution formulas give
$\partial_x^\xi(K*\rho_h)\in C_b$ for $|\xi|\le j$.
Consequently, also for differences, we obtain
\begin{equation}\label{eq:whole-space-force-derivative-bound}
 \|K*\rho_h\|_{W^{j,\infty}}
 \le C\|h\|_{\mathcal C^j_{M_1,a_1}},
 \qquad 0\le j\le r.
\end{equation}
We now construct the nonlinear solution by Picard iteration.  Let
$\mathcal S_E(t)$ denote the solution operator in
\Cref{lem:whole-space-prescribed-linear}, and set
\[
 f^{(0)}=f_0,\qquad E^{(n)}=K*\rho_{f^{(n)}},\qquad
 f^{(n+1)}=\mathcal S_{E^{(n)}}(t)f_0.
\]
All iterates use the input weights $(M,a_0)$ and output weights
$(M_1,a_1)$.  Choose $R>2\|f_0\|_{\mathcal C^r_{M,a_0}}$.
If $\|f^{(n)}\|_{L_t^\infty\mathcal C^r_{M_1,a_1}}\le R$, the field
estimate gives $\|E^{(n)}\|_{L_t^\infty C_b^r}\le C R$.
Choose $T_f$ so that
\[
 \bigl(1+\omega(T_f;CR)\bigr)
       \|f_0\|_{\mathcal C^r_{M,a_0}}\le R.
\]
Then \eqref{eq:whole-space-prescribed-small-time-bound} closes the induction,
starting from $f^{(0)}=f_0$, and yields
\[
 \sup_n\|f^{(n)}\|_{L_t^\infty\mathcal C^r_{M_1,a_1}}\le R,
 \qquad \sup_n\|E^{(n)}\|_{L_t^\infty C_b^r}\le CR.
\]
To obtain contraction, set
$\Delta_n=\|f^{(n)}-f^{(n-1)}\|_{C_t\mathcal C^{r-1}_{M_1,a_1}}$.
Decrease $T_f$ so that the small-time condition in
\Cref{lem:whole-space-prescribed-linear} holds for the common field bound
$CR$.  The field estimate and \eqref{eq:whole-space-interval-stability},
with $T_0=T_f$, give
\[
 \Delta_{n+1}\le C_RT_f\Delta_n.
\]
Taking $C_RT_f<1$ and using the weighted $L^1$ embedding, we obtain
\begin{equation}\label{eq:riesz-picard-strong-limit}
 f^{(n)}\to f
 \quad\hbox{in }C_t\mathcal C^{r-1}_{M_1,a_1}\cap C_tL^1,
 \qquad
 E^{(n)}\to E_f
 \quad\hbox{in }C_tC_b^{r-1}.
\end{equation}
The densities converge in $C_tW^{r-1,\infty}$ and, on $\R^d$, also in
$C_tW^{r-1,1}$, so $E_f=K*\rho_f$.

The highest derivatives also pass to the limit as weak derivatives.
For $|\beta|=r$, the sequence
$w^{\mathcal D}_{M_1,a_1}D^\beta f^{(n)}$ is bounded in $L^\infty$.
Weak-* compactness and integration by parts against compactly supported
test functions identify the limit with
$w^{\mathcal D}_{M_1,a_1}D^\beta f$.  Hence
\[
 \max_{|\beta|=r}
 \|w^{\mathcal D}_{M_1,a_1}D^\beta f\|_{L^\infty((0,T_f)\times\Omega_{\mathcal D})}
 <\infty.
\]
For $|\xi|=r$, integration in velocity gives, in distributions and
for almost every time,
\[
 \begin{aligned}
 \partial_x^\xi\rho_f&=\int_{\R^d}\partial_x^\xi f\dd v,\\
 \|\partial_x^\xi\rho_f\|_\infty
 +\mathbf1_{\{\mathcal D=\R^d\}}
       \|\partial_x^\xi\rho_f\|_1
 &\le C\|w^{\mathcal D}_{M_1,a_1}\partial_x^\xi f\|_\infty.
 \end{aligned}
\]
Hence
$\partial_x^\xi E_f=K*\partial_x^\xi\rho_f$ in distributions,
and the preceding decomposition of $K$ yields
\[
 f\in C_t\mathcal C^{r-1}_{M_1,a_1},\qquad
 E_f\in C_tC_b^{r-1}\cap L_t^\infty W^{r,\infty}.
\]
The Gaussian velocity weight and
\eqref{eq:riesz-picard-strong-limit} imply
\[
 -v\cdot\nabla_xf^{(n+1)}
 -E^{(n)}\cdot\nabla_vf^{(n+1)}+\sigma\Delta_vf^{(n+1)}
 \longrightarrow
 -v\cdot\nabla_xf-E_f\cdot\nabla_vf+\sigma\Delta_vf
 \quad\text{in }C_tC_b^{r-3}.
\]
Passing to the limit in the integrated equations therefore gives
\[
 f(t)-f_0=\int_0^t
 \bigl(-v\cdot\nabla_xf-E_f\cdot\nabla_vf
       +\sigma\Delta_vf\bigr)(s)\dd s
 \quad\text{in }C_b^{r-3}.
\]
Hence $f\in C^1([0,T_f];C_b^{r-3})$.  Since $r\ge3$, this and
$f\in C_t\mathcal C^{r-1}_{M_1,a_1}$ give the classical regularity in
\Cref{subsec:intro-model}.
The $C_tL^1$ convergence preserves nonnegativity and unit mass, while the
weighted derivative bounds give
$\int_0^{T_f}\|\nabla_v f(t)\|_1\dd t<\infty$.
\end{proof}

\section{Weighted Sobolev and Gaussian estimates for the Vlasov--Poisson--Fokker--Planck equation}\label{app:coulomb-estimates}

In this appendix, $K(x)=cx/|x|^3$ on $\R^3$ and
$K=-c\nabla G_{\T}$ on $\T^3$, with $G_{\T}$ as in the introduction.

For the three-dimensional VPFP equation on $\R^3$ with velocity diffusion
$\sigma\Delta_v f$, local classical solutions were obtained by Victory and
O'Dwyer
\cite{VictoryODwyer1990}, while Bouchut \cite{Bouchut1993} proved global
existence and uniqueness; see also \cite[Section~1.2]{HuangLiuPickl2020}.
The argument below constructs in both geometries a local solution satisfying
the weighted Sobolev estimates required by
\Cref{cor:intro-quantitative-coulomb}.  On $\R^3$ the spatial weight is
transported by the free flow.  Gaussian velocity moments, together with a
polynomial spatial moment on $\R^3$, are then converted into pointwise bounds
by \Cref{lem:gauss-integral-to-pointwise,lem:weighted-moment-pointwise}.

\subsection{Weighted spaces and local regularity}

For the stopping argument below, we set
\begin{equation}\label{eq:gauss-domain-confinement}
 r_{\T}(x,v):=1+|v|^2,\qquad
 r_{\R}(x,v):=1+|x|^2+|v|^2,
\end{equation}
and write $r_{\mathcal D}$ for the corresponding function.

\begin{lemma}\label{lem:gauss-integral-to-pointwise}
Let $m>3$ and $h\ge0$ on $\T^3\times\R^3$.  If
\[
 h\in H^m(\T^3\times\R^3),\qquad
 \int h(x,v)e^{a|v|^2}\dd x\dd v<\infty
\]
for some $a>0$, then $h(x,v)\le C_*e^{-a_*|v|^2}$ for suitable
$a_*>0$ and $C_*<\infty$, depending only on $m,a$ and the displayed norms.
\end{lemma}

\begin{proof}
For a function $g$ supported in a fixed coordinate ball in dimension six,
the Gagliardo--Nirenberg inequality \cite{AdamsFournier2003} gives
\begin{equation}\label{eq:local-GN-interpolation}
 \|g\|_\infty\le C_m\|g\|_{H^m}^{6/(m+3)}
                         \|g\|_1^{(m-3)/(m+3)},\qquad m>3.
\end{equation}
Let $\chi_{z_0}$ be a smooth cutoff equal to one near $z_0=(x_0,v_0)$
and supported in a coordinate ball of fixed radius.  Using a fixed finite
atlas on the torus, the cutoffs and their derivatives are uniformly
bounded.  On their support, $|v|^2\ge|v_0|^2/2-C$, so
\[
 \|\chi_{z_0}h\|_1\le C e^{-a|v_0|^2/2}\int h e^{a|v|^2},
 \qquad \|\chi_{z_0}h\|_{H^m}\le C\|h\|_{H^m}.
\]
Applying \eqref{eq:local-GN-interpolation} proves the claim with
\[
 a_* = \frac{a(m-3)}{2(m+3)}>0.
\]
The constant $C_*$ depends only on $m,a$ and the displayed norms.
\end{proof}

On $\R^3$ define
\begin{equation}\label{eq:whole-space-moving-sobolev-weight}
 \omega_{\mu,\ell,t}(x,v):=\langle x-tv\rangle^\mu\langle v\rangle^\ell,
\end{equation}
and
\begin{equation}\label{eq:whole-space-moving-sobolev-norm}
 \|h\|_{H^m_{\mu,\ell}(t)}^2
 :=\sum_{|\xi|+|\eta|\le m}\int_{\R^6}
 \omega_{\mu,\ell,t}^2
 |\partial_x^\xi\partial_v^\eta h|^2\dd x\dd v.
\end{equation}
We write $h\in C_tH^j_{\mu,\ell}(t)$ when
\begin{equation}\label{eq:moving-weight-continuity-definition}
 t\longmapsto\omega_{\mu,\ell,t}\partial_x^\xi\partial_v^\eta h(t)
 \quad\hbox{belongs to }C([0,T];L^2),\qquad |\xi|+|\eta|\le j.
\end{equation}
Multiplication by $\omega_{\mu,\ell,t}^{-1}$, which is bounded and
pointwise continuous in $t$, shows that
$h\in C_tH^j_{\mu,\ell}(t)$ implies $h\in C_tH^j$.
For $\mu,\ell>3/2$ and
$|\xi|\le j$, weighted Cauchy--Schwarz gives
\begin{align*}
 \|\partial_x^\xi\rho_h(t)\|_{L_x^2}
 &\le C\|h(t)\|_{H^j_{\mu,\ell}(t)},\\
 \|\partial_x^\xi\rho_h(t)\|_{L_x^1}
 &\le
 \|\omega_{\mu,\ell,t}\partial_x^\xi h(t)\|_2
 \left(\int_{\R^6}\langle y\rangle^{-2\mu}
                    \langle v\rangle^{-2\ell}\dd y\dd v\right)^{1/2}.
\end{align*}
The first estimate uses $\omega_{\mu,\ell,t}^{-1}\le\langle v\rangle^{-\ell}$;
the second uses $y=x-tv$.  Their constants are independent of $t$.
To see time continuity, define
$T_tF(x)=\int\omega_{\mu,\ell,t}(x,v)^{-1}F(x,v)\dd v$.
The two estimates give uniform bounds for $T_t$ from $L^2_{x,v}$ to
$L_x^2$ and to $L_x^1$.  For $F\in C_c$, continuity of the inverse
weight gives continuity of $T_tF$ in both spaces.  Density of $C_c$ in
$L^2$ extends this to every fixed $F\in L^2$.  Apply it to the
continuous curve $F(t)=\omega_{\mu,\ell,t}\partial_x^\xi h(t)$ to obtain
$\rho_h\in C_tH_x^j$ and $\partial_x^\xi\rho_h\in C_tL_x^1$.

\begin{lemma}\label{lem:weighted-moment-pointwise}
Let $m\ge4$ be an integer, $\mu,\ell,P\ge0$, $a>0$, and $0\le t\le T$.
Suppose $h\ge0$ and
\[
 h\in H^m_{\mu,\ell}(t),\qquad
 I_{P,a}(h):=\int_{\R^6}\langle x\rangle^Pe^{a|v|^2}h(x,v)\dd x\dd v<\infty.
\]
Then $h$ satisfies
\begin{equation}\label{eq:weighted-moment-pointwise-bound}
 h(x,v)\le C_*
 \langle x\rangle^{-\frac{6\mu+(m-3)P}{m+3}}e^{-a_*|v|^2},
 \qquad a_* =\frac{a(m-3)}{4(m+3)}.
\end{equation}
If $\mu>3$, then for every $P\ge0$ there are constants $M_*>3$,
$0<a_1<a_*$, and $C_*<\infty$ such that
$h(x,v)\le C_*\langle x\rangle^{-M_*}e^{-a_1|v|^2}$.  The constants depend on $m,\mu,\ell,P,a,T$ and
the two displayed norms.
\end{lemma}

\begin{proof}
Let $\chi_{z_0}$ be a translate of a fixed smooth cutoff, equal to one
on $B_1(z_0)$ and supported in $B_2(z_0)$, where $z_0=(x_0,v_0)$.
Uniformly for $0\le t\le T$ and $z\in B_2(z_0)$, we have
\[
 C_T^{-1}\le\frac{\omega_{\mu,\ell,t}(z)}
                      {\omega_{\mu,\ell,t}(z_0)}\le C_T,
 \qquad \langle x\rangle^P\ge c_P\langle x_0\rangle^P,
 \qquad |v|^2\ge\tfrac12|v_0|^2-4.
\]
These bounds imply
\begin{align}\label{eq:weighted-moment-local-norms}
 \|\chi_{z_0}h\|_{H^m}
 &\le C_T\langle x_0-tv_0\rangle^{-\mu}
              \langle v_0\rangle^{-\ell}\|h\|_{H^m_{\mu,\ell}(t)},\notag\\
 \|\chi_{z_0}h\|_1
 &\le C\langle x_0\rangle^{-P}e^{-a|v_0|^2/2}I_{P,a}(h).
\end{align}
Applying \eqref{eq:local-GN-interpolation} gives
\begin{align}\label{eq:weighted-moment-interpolation}
 h(x,v)
 &\le C_T\|h\|_{H^m_{\mu,\ell}(t)}^{\frac6{m+3}}
            I_{P,a}(h)^{\frac{m-3}{m+3}}
       \langle x-tv\rangle^{-\frac{6\mu}{m+3}}
       \langle v\rangle^{-\frac{6\ell}{m+3}}
       \langle x\rangle^{-\frac{(m-3)P}{m+3}}
       e^{-\frac{a(m-3)}{2(m+3)}|v|^2}\notag\\
 &\le C_*
       \langle x\rangle^{-\frac{6\mu+(m-3)P}{m+3}}
       \langle v\rangle^{\frac{6(\mu-\ell)}{m+3}}
       e^{-\frac{a(m-3)}{2(m+3)}|v|^2}.
\end{align}
The second inequality uses
$\langle x\rangle\le C_T\langle x-tv\rangle\langle v\rangle$.
Since a polynomial times $e^{-a_*|v|^2}$ is bounded,
\eqref{eq:weighted-moment-pointwise-bound} follows.

For $\mu>3$, the inequality $I_{0,a}(h)\le I_{P,a}(h)$ allows us to
apply \eqref{eq:weighted-moment-pointwise-bound} with $P=0$.
Also, $|D^\gamma\omega_{\mu,\ell,t}|\le C_T\omega_{\mu,\ell,t}$
for $|\gamma|\le m$, so the product rule gives
\[
 \|\omega_{\mu,\ell,t}h\|_{H^m}
 \le C_T\|h\|_{H^m_{\mu,\ell}(t)}.
\]
Sobolev embedding in dimension six, applied to
$\omega_{\mu,\ell,t}h$, therefore yields
\[
 h(x,v)\le C_T\omega_{\mu,\ell,t}^{-1}
 \le C_T\langle x\rangle^{-\mu}\langle v\rangle^{\mu-\ell}.
\]
Combining this Sobolev bound with
\eqref{eq:weighted-moment-pointwise-bound} at $P=0$, with powers
$\theta$ and $1-\theta$, respectively, gives, for each fixed
$0\le\theta<1$,
\begin{align}\label{eq:weighted-moment-spatial-exponent}
 h(x,v)
 &\le C_T
 \langle x\rangle^{-\mu\frac{6+\theta(m-3)}{m+3}}
 \langle v\rangle^{\theta(\mu-\ell)}
 e^{-a_*(1-\theta)|v|^2}\notag\\
 &\le C_{T,\theta}
 \langle x\rangle^{-\mu\frac{6+\theta(m-3)}{m+3}}
 e^{-\frac12a_*(1-\theta)|v|^2}.
\end{align}
Choose $0\le\theta<1$ so that
$M_*:=\mu(6+\theta(m-3))/(m+3)>3$ and set
$a_1=\tfrac12a_*(1-\theta)\in(0,a_*)$.
\end{proof}

\subsection{Density, field and commutator bounds}\label{subsec:coulomb-field-bounds}

We set
\[
 \omega_t=
 \begin{cases}
 \langle v\rangle^\ell,&\mathcal D=\T^3,\\
 \langle x-tv\rangle^\mu\langle v\rangle^\ell,&\mathcal D=\R^3,
 \end{cases}
 \qquad
 \mathcal H^j(t)=
 \begin{cases}H^j_\ell,&\mathcal D=\T^3,\\
 H^j_{\mu,\ell}(t),&\mathcal D=\R^3.
 \end{cases}
\]

The $L_x^1$ density bound controls the far-field part of the Coulomb kernel
on $\R^3$.

\begin{lemma}\label{lem:whole-space-coulomb-sobolev-field}
For an integer $j\ge0$, $\ell>3/2$ and, on $\R^3$, $\mu>3/2$, we have
\begin{equation}\label{eq:whole-space-coulomb-sobolev-density-field}
 \|\rho_h\|_{H_x^j}
 +\sum_{|\gamma|\le j}\|\partial_x^\gamma\rho_h\|_{L_x^1}
 +\|K*\rho_h\|_{H_x^{j+1}}
 \le C\|h\|_{\mathcal H^j(t)}.
\end{equation}
The constant depends only on $j,\ell,c$ and, on $\R^3$, $\mu$.
It is independent of $t$, and the estimate also holds for signed differences.
\end{lemma}

\begin{proof}
The density bounds follow from weighted Cauchy--Schwarz as above.  On $\T^3$, write
$\langle\rho_h\rangle:=\int_{\T^3}\rho_h(x)\dd x$.  The Poisson estimate gives
\begin{equation}\label{eq:torus-coulomb-elliptic-bound}
 \|K*\rho_h\|_{H_x^{j+1}}
 \le C\|\rho_h-\langle\rho_h\rangle\|_{H_x^j}.
\end{equation}
On $\R^3$, the Coulomb multiplier satisfies
$|\widehat{K*\rho_h}(\xi)|\le C|\xi|^{-1}|\widehat\rho_h(\xi)|$.
Plancherel's identity gives
\begin{align}\label{eq:whole-space-coulomb-elliptic-bound}
 \|K*\rho_h\|_{H_x^{j+1}}^2
 &\le C\int_{|\xi|\le1}|\xi|^{-2}|\widehat\rho_h(\xi)|^2\dd\xi
   +C\int_{|\xi|>1}(1+|\xi|^2)^j|\widehat\rho_h(\xi)|^2\dd\xi
   \notag\\
 &\le C\|\rho_h\|_1^2\int_{|\xi|\le1}|\xi|^{-2}\dd\xi
       +C\|\rho_h\|_{H_x^j}^2
 \le C\|h\|_{\mathcal H^j(t)}^2.
\end{align}
Here $\int_{|\xi|\le1}|\xi|^{-2}\dd\xi=4\pi$, proving
\eqref{eq:whole-space-coulomb-sobolev-density-field}.
\end{proof}

The following estimates control the field terms in the differentiated
kinetic equation.

\begin{lemma}\label{lem:whole-space-moving-weight-commutator}
Let $m\ge4$ be an integer and $\mu,\ell\ge0$.  In either geometry, if
$E\in H^m(\mathcal D;\R^3)$ and $h\in\mathcal H^m(t)$, then the commutator satisfies
\begin{equation}\label{eq:whole-space-moving-weight-commutator}
 \sum_{|\xi|+|\eta|\le m}
 \|\omega_t[\partial_x^\xi\partial_v^\eta,E\cdot\nabla_v]h\|_2
 \le C\|E\|_{H_x^m}\|h\|_{\mathcal H^m(t)}.
\end{equation}
Moreover, if $E_1,E_2:\mathcal D\to\R^3$ are vector fields such that
$E_1-E_2\in H^{m-1}(\mathcal D;\R^3)$, then we have
\begin{equation}\label{eq:whole-space-moving-weight-source}
 \|(E_1-E_2)\cdot\nabla_vh\|_{\mathcal H^{m-1}(t)}
 \le C\|E_1-E_2\|_{H_x^{m-1}}\|h\|_{\mathcal H^m(t)}.
\end{equation}
\end{lemma}

\begin{proof}
Since $E=E(x)$, the commutator expands as
\begin{equation}\label{eq:coulomb-commutator-expansion}
 [\partial_x^\xi\partial_v^\eta,E\cdot\nabla_v]h
 =\sum_{0<\gamma\le\xi}\binom{\xi}{\gamma}
 (\partial_x^\gamma E)\cdot\nabla_v
 \partial_x^{\xi-\gamma}\partial_v^\eta h.
\end{equation}
Set $G=\nabla_v\partial_x^{\xi-\gamma}\partial_v^\eta h$.  For
$1\le |\gamma|\le m-2$, $H_x^2\hookrightarrow L_x^\infty$ gives
\[
 \|\omega_t(\partial_x^\gamma E)\cdot G\|_{L^2_{x,v}}
 \le \|\partial_x^\gamma E\|_{L_x^\infty}
      \|\omega_tG\|_{L^2_{x,v}}
 \le C\|E\|_{H_x^m}\|h\|_{\mathcal H^m(t)}.
\]
For $|\gamma|\ge m-1$,
$H_x^2(L_v^2)\hookrightarrow L_x^\infty(L_v^2)$ gives
\[
 \|\omega_t(\partial_x^\gamma E)\cdot G\|_{L^2_{x,v}}
 \le \|\partial_x^\gamma E\|_{L_x^2}
      \|\omega_tG\|_{L_x^\infty L_v^2}.
\]
For $|\alpha|\le2$ one has
$|\partial_x^\alpha\omega_t|\le C_\alpha\omega_t$.  Hence
\[
 \|\omega_tG\|_{L_x^\infty L_v^2}
 \le C\sum_{|\alpha|\le2}
       \|\partial_x^\alpha(\omega_tG)\|_{L^2_{x,v}}
 \le C\|h\|_{\mathcal H^m(t)},
\]
because
$|\xi|+|\eta|-|\gamma|+1+|\alpha|\le m$ whenever
$|\gamma|\ge m-1$ and $m\ge4$.  Summing over the multi-indices proves
\eqref{eq:whole-space-moving-weight-commutator}.

For \eqref{eq:whole-space-moving-weight-source}, write $F=E_1-E_2$.  A typical
differentiated term is
\[
 (\partial_x^\gamma F)\cdot
 \nabla_v\partial_x^{\xi-\gamma}\partial_v^\eta h,
 \qquad |\xi|+|\eta|\le m-1.
\]
For $|\gamma|\le m-3$ we use $L_x^\infty$ on the field factor; for
$|\gamma|\ge m-2$ we use $L_x^2$ on that factor and
$L_x^\infty(L_v^2)$ on the other.  The derivative count is
\[
 (m-1)-|\gamma|+1+2=m-|\gamma|+2\le4\le m.
\]
The preceding derivative bound for $\omega_t$ then gives
\eqref{eq:whole-space-moving-weight-source}.
\end{proof}

\subsection{Weighted Sobolev estimates and local solutions}

\begin{proposition}[Local VPFP solutions in weighted Sobolev spaces]\label{prop:gauss-coulomb-sobolev-core}
Let $m\ge6$ be an integer, $\sigma,c>0$, and let $f_0\ge0$ have unit mass.
On $\T^3$, assume $\ell>3/2$ and
$f_0\in H^m_\ell(\T^3\times\R^3)$.
On $\R^3$, assume $\mu,\ell>3/2$ and
$f_0\in H^m_{\mu,\ell}(\R^6)$.
Then there exist $T_f>0$ and a nonnegative classical solution $f$ of the
VPFP equation on $[0,T_f]$ with unit mass, such that
$f\in C_t\mathcal H^{m-1}(t)$ and
\[
 \sup_{t\le T_f}\|f(t)\|_{\mathcal H^m(t)}
 +\left(\int_0^{T_f}\|\nabla_vf(t)\|_{\mathcal H^m(t)}^2\dd t\right)^{1/2}
 <\infty.
\]
In addition, we have
\begin{equation}\label{eq:gauss-sobolev-field-core}
 E_f\in C([0,T_f];C_b^1),\qquad
 \sup_{t\le T_f}
 \bigl(\|f(t)\|_{\mathcal H^m(t)}+\|E_f(t)\|_{H_x^m}\bigr)<\infty.
\end{equation}
Uniqueness holds among classical solutions in
$C_t\mathcal H^{m-1}(t)\cap L_t^\infty\mathcal H^m(t)$ with
$\nabla_v f\in L_t^2\mathcal H^m(t)$.
\end{proposition}

\begin{proof}
\emph{Linear weighted estimate.}
Let $E\in L_t^\infty H_x^m\cap C_tH_x^{m-1}$ and consider
\eqref{eq:whole-space-prescribed-linear} with $g=h$ and $g_0=h_0$.
We establish the weighted Sobolev estimate under these field assumptions.
For $\gamma=(\xi,\eta)$, write
$D^\gamma=\partial_x^\xi\partial_v^\eta$.  The free-transport commutator is
\[
 [D^\gamma,v\cdot\nabla_x]h
 =\sum_{\substack{1\le j\le3\\\eta_j\ge1}}\eta_j
 \partial_x^{\xi+e_j}\partial_v^{\eta-e_j}h,
\]
where $e_j$ is the $j$th coordinate vector.  Each term has total
differential order $|\gamma|$.  Put $\mu=0$ on $\T^3$.  Direct differentiation gives
\begin{equation}\label{eq:coulomb-moving-weight-derivatives}
 \begin{aligned}
 &(\partial_t+v\cdot\nabla_x)\omega_t=0,\qquad
 \nabla_v\log\omega_t
 =-\mu t\frac{x-tv}{1+|x-tv|^2}
       +\ell\frac{v}{1+|v|^2},\\
& \Delta_v\log\omega_t
 =\mu t^2\frac{3+|x-tv|^2}{(1+|x-tv|^2)^2}
       +\ell\frac{3+|v|^2}{(1+|v|^2)^2}, \qquad
 \frac{\Delta_v(\omega_t^2)}{\omega_t^2}
 =2\Delta_v\log\omega_t+4|\nabla_v\log\omega_t|^2.
 \end{aligned}
\end{equation}
In particular, the weight satisfies
\[
 \left\|\frac{\nabla_v\omega_t}{\omega_t}\right\|_\infty
 +\left\|\frac{\Delta_v(\omega_t^2)}{\omega_t^2}\right\|_\infty
 \le C_{T,\mu,\ell},\qquad 0\le t\le T.
\]
Write $h_\gamma=D^\gamma h$ and set
\[
 \mathcal R_\gamma
 =-[D^\gamma,v\cdot\nabla_x]h-[D^\gamma,E\cdot\nabla_v]h.
\]
For smooth fields and data, multiplying the differentiated equation by
$\omega_t^2h_\gamma$ and integrating gives
\begin{align}
 \frac12\frac{\dd}{\dd t}\int\omega_t^2|h_\gamma|^2
 +\sigma\int\omega_t^2|\nabla_vh_\gamma|^2
 =\int\omega_t^2h_\gamma\mathcal R_\gamma+\frac12\int |h_\gamma|^2
       \bigl(E\cdot\nabla_v\omega_t^2+\sigma\Delta_v\omega_t^2\bigr).
 \label{eq:coulomb-differentiated-weighted-identity}
\end{align}
The transport derivative of the weight has canceled.  Sum over
$|\gamma|\le m$, use the commutator estimate in
\Cref{lem:whole-space-moving-weight-commutator}, and apply
$\|E\|_\infty\le C\|E\|_{H_x^m}$.  This yields
\begin{equation}\label{eq:whole-space-moving-linear-energy}
 \frac{\dd}{\dd t}\|h\|_{\mathcal H^m(t)}^2
 +\sigma\|\nabla_vh\|_{\mathcal H^m(t)}^2
 \le C_T(1+\|E\|_{H_x^m})\|h\|_{\mathcal H^m(t)}^2.
\end{equation}
\emph{Approximation of the field and initial datum.}
Let $E^\nu$ be a spatial mollification of $E$ and choose smooth initial
data $h_0^\nu$, compactly supported in the noncompact variables, such that
\[
 E^\nu\longrightarrow E\quad\text{in }C_tH_x^{m-1},
 \qquad
 \sup_\nu\|E^\nu\|_{L_t^\infty H_x^m}
 \le \|E\|_{L_t^\infty H_x^m},
 \qquad
 h_0^\nu\longrightarrow h_0\quad\text{in }\mathcal H^m(0).
\]
The corresponding solutions $h^\nu$ are given by the stochastic flow.
Let $R$ bound the field and initial norms in these approximations.
The integration by parts above can be justified at this level as follows.
For a bounded field $E^\nu$, its characteristics satisfy
\[
 |V_t-v_0|+|X_t-tV_t-x_0|
 \le C_{T,R}(1+B_T^*),\qquad B_T^*=\sup_{s\le T}|B_s|.
\]
On $\R^3$, it follows that
\[
 \frac{\omega_t(\Phi^{E^\nu}_{0,t}(x_0,v_0))}{\omega_0(x_0,v_0)}
 \le C_{T,R}(1+B_T^*)^{\mu+\ell};
\]
on the torus only the velocity factor occurs.  For fixed $\nu$, the flow
and inverse-flow derivatives are bounded, and the transported supports of
$h_0^\nu$ are pathwise compact.  For $|\gamma|\le m$, volume preservation, the weight ratio and the
polynomial moments of $B_T^*$ give
\[
 \begin{aligned}
 \|\omega_tD^\gamma h^\nu(t)-\omega_sD^\gamma h^\nu(s)\|_2\le\E\bigl\|
  \omega_tD^\gamma[h_0^\nu\circ(\Phi^{E^\nu}_{0,t})^{-1}]
  -\omega_sD^\gamma[h_0^\nu\circ(\Phi^{E^\nu}_{0,s})^{-1}]
       \bigr\|_2\longrightarrow0
 \qquad(s\to t).
 \end{aligned}
\]
The same estimates at higher orders give finite weighted derivative norms
on $[0,T]$.  Inserting phase-space cutoffs in
\eqref{eq:coulomb-differentiated-weighted-identity} and then letting their
radius tend to infinity gives \eqref{eq:whole-space-moving-linear-energy},
uniformly in $\nu$.  More explicitly, choose a radial cutoff in
$(x,v)$ on $\R^6$ and in $v$ on $\T^3\times\R^3$.
On the support of its derivatives, $|v|\le2R$, so the free-transport
boundary term is bounded by a weighted energy tail.  The bounded force
and the diffusion terms are controlled by the weighted derivative
norms established above.  These tails vanish as $R\to\infty$.

For $q^{\nu,\kappa}:=h^\nu-h^\kappa$, the difference equation is
\[
 (\partial_t+v\cdot\nabla_x+E^\nu\cdot\nabla_v-\sigma\Delta_v)
 q^{\nu,\kappa}
 =(E^\kappa-E^\nu)\cdot\nabla_vh^\kappa,
 \qquad q^{\nu,\kappa}(0)=h_0^\nu-h_0^\kappa.
\]
Apply the order-$(m-1)$ energy estimate and
\eqref{eq:whole-space-moving-weight-source}, followed by Young's and
Gronwall's inequalities.  With the preceding uniform bounds, this gives
\begin{equation}\label{eq:coulomb-linear-approximation-difference}
 \begin{aligned}
 &\sup_{0\le r\le t}
   \|q^{\nu,\kappa}(r)\|_{\mathcal H^{m-1}(r)}^2
 +\sigma\int_0^t
   \|\nabla_vq^{\nu,\kappa}(r)\|_{\mathcal H^{m-1}(r)}^2\dd r\\
 \le & C_{T,R}
   \|h_0^\nu-h_0^\kappa\|_{\mathcal H^{m-1}(0)}^2+C_{T,R}\int_0^t\|E^\nu(r)-E^\kappa(r)\|_{H_x^{m-1}}^2
              \|h^\kappa(r)\|_{\mathcal H^m(r)}^2\dd r.
 \end{aligned}
\end{equation}
The right-hand side tends to zero.  Hence $h^\nu$ converges in
$C_t\mathcal H^{m-1}(t)$ to a solution $h$ of
\eqref{eq:whole-space-prescribed-linear}.  At each fixed time, this
convergence identifies any weak limit of the bounded highest weighted
derivatives with the corresponding derivatives of $h(t)$.
Multiply \eqref{eq:whole-space-moving-linear-energy} for $h^\nu$ by
\[
 \exp\!\left(-C_T\int_0^t(1+\|E(r)\|_{H_x^m})\dd r\right).
\]
The integral form, the bound $\|E^\nu(r)\|_{H_x^m}\le\|E(r)\|_{H_x^m}$,
and weak lower semicontinuity give
\begin{equation}\label{eq:coulomb-linear-integrated-energy}
 \begin{aligned}
 \|h(t)\|_{\mathcal H^m(t)}^2
 +\sigma\int_0^t\|\nabla_vh(r)\|_{\mathcal H^m(r)}^2\dd r\le\|h_0\|_{\mathcal H^m(0)}^2
 \exp\!\left(C_T\int_0^t(1+\|E(r)\|_{H_x^m})\dd r\right).
 \end{aligned}
\end{equation}
Let $h,\widetilde h\in C_t\mathcal H^{m-1}(t)\cap
L_t^\infty\mathcal H^m(t)$ solve
\eqref{eq:whole-space-prescribed-linear} with fields
$E,\widetilde E\in L_t^\infty H_x^m\cap C_tH_x^{m-1}$, and suppose
$\nabla_vh,\nabla_v\widetilde h\in L_t^2\mathcal H^m(t)$.
For $|\gamma|\le m-1$, their equations imply
$
 \partial_tD^\gamma(h-\widetilde h)\in L_t^2L^2_{\mathrm{loc}}.
$
Applying the preceding cutoff calculation at order $m-1$ and using
\eqref{eq:whole-space-moving-weight-source} gives
\begin{align}\label{eq:coulomb-linear-comparison-energy}
 &\frac{\mathrm d}{\mathrm dt}
   \|h-\widetilde h\|_{\mathcal H^{m-1}(t)}^2
 +\sigma\|\nabla_v(h-\widetilde h)\|_{\mathcal H^{m-1}(t)}^2
 \notag\\
 \le & C_T(1+\|E\|_{H_x^{m-1}})
             \|h-\widetilde h\|_{\mathcal H^{m-1}(t)}^2
+C_T\|E-\widetilde E\|_{H_x^{m-1}}
       \|\widetilde h\|_{\mathcal H^m(t)}
       \|h-\widetilde h\|_{\mathcal H^{m-1}(t)}.
\end{align}
The $C_t\mathcal H^{m-1}(t)$ continuity gives its integrated form
at every time.

For nonnegative $h_0$, choose
nonnegative approximations with the same mass.  The stochastic-flow
representation preserves positivity and mass for $h^\nu$; see
\cite[Sections~4.5--4.6]{Kunita1990}.  The weighted $L^1$ embedding gives
$h^\nu\to h$ in $C_tL^1$, so both properties pass to $h$.

\emph{Identification with the stochastic flow.}
Denote
by $\Phi^{E^\nu}_{0,t}$ and $\Phi^E_{0,t}$ the flows driven by the same
Brownian motion.  Sobolev embedding gives a common bound for
$\|\nabla E^\nu\|_\infty$ and $\|\nabla E\|_\infty$, while
$E^\nu\to E$ in $C_tL_x^\infty$.  The difference of the characteristic
equations and Gronwall's inequality yield
\begin{equation}\label{eq:coulomb-linear-flow-stability}
 \sup_{t\le T,z}
 |\Phi^{E^\nu}_{0,t}z-\Phi^E_{0,t}z|
 \le C_{T,R}\int_0^T\|E^\nu(s)-E(s)\|_\infty\dd s
 \longrightarrow0.
\end{equation}
On the torus we use lifts with the same initial position.  For a bounded
Lipschitz function $\psi$, volume preservation gives
\[
 \int\psi(z)h^\nu(t,z)\dd z
 =\E\int\psi(\Phi^{E^\nu}_{0,t}z)h_0^\nu(z)\dd z.
\]
The difference between the right-hand side and
$\E\int\psi(\Phi^E_{0,t}z)h_0(z)\dd z$ is bounded by
$\|\psi\|_\infty\|h_0^\nu-h_0\|_1$ plus
$\operatorname{Lip}(\psi)\|h_0\|_1$ times the supremum in
\eqref{eq:coulomb-linear-flow-stability}.  Passing to the limit gives
\begin{equation}\label{eq:coulomb-linear-flow-representation}
 \int\psi(z)h(t,z)\dd z
 =\E\int\psi(\Phi^E_{0,t}z)h_0(z)\dd z,
 \qquad 0\le t\le T.
\end{equation}
Bounded Lipschitz functions determine finite measures.  Thus, for a
probability density $h_0$, the measure $h(t,z)\dd z$ is the time-$t$ law
of the kinetic diffusion with prescribed field $E$.

\emph{Picard iteration.}
We set
\[
 f^{(0)}=f_0,\qquad E^{(n)}=K*\rho_{f^{(n)}},
\]
and let $f^{(n+1)}$ solve
\eqref{eq:whole-space-prescribed-linear} with field $E^{(n)}$ and initial
datum $f_0$.  The initial field is independent of time and satisfies
\[
 \|E^{(0)}\|_{H_x^m}
 \le C\|f_0\|_{\mathcal H^m(0)}
\]
by \Cref{lem:whole-space-coulomb-sobolev-field} at $t=0$.  We apply the
estimates with the moving weight to $f^{(n)}$ only for $n\ge1$.

Fix $R>2\|f_0\|_{\mathcal H^m(0)}$ and first restrict to $T_f\le1$.
If the preceding iterate has weighted norm at most $R$,
\Cref{lem:whole-space-coulomb-sobolev-field} bounds its field by $CR$.
Choose $T_f$ so that
\[
 \|f_0\|_{\mathcal H^m(0)}^2
       \exp\bigl(C(1+CR)T_f\bigr)\le R^2.
\]
The integrated estimate then preserves the bound $R$ at every iteration
and also controls the velocity dissipation.  We thus obtain
\begin{equation}\label{eq:coulomb-picard-uniform}
 \sup_{n\ge1}\sup_{t\le T_f}\|f^{(n)}(t)\|_{\mathcal H^m(t)}
 +\sup_{n\ge1}\left(
   \int_0^{T_f}\|\nabla_v f^{(n)}(t)\|_{\mathcal H^m(t)}^2\dd t
 \right)^{1/2}
 \le C_R,
\end{equation}
and $\sup_{n\ge0}\sup_{t\le T_f}\|E^{(n)}(t)\|_{H_x^m}\le C_R$.

To prove convergence of the iteration, set
$g^{(n)}=f^{(n+1)}-f^{(n)}$ for $n\ge1$.  Then
\[
 \partial_tg^{(n)}+v\cdot\nabla_xg^{(n)}+E^{(n)}\cdot\nabla_vg^{(n)}
 -\sigma\Delta_vg^{(n)}
 =(E^{(n-1)}-E^{(n)})\cdot\nabla_vf^{(n)}.
\]
The first difference $g^{(1)}$ is bounded in $C_t\mathcal H^{m-1}(t)$ by
\eqref{eq:coulomb-picard-uniform}.
For $n\ge2$, \eqref{eq:coulomb-linear-comparison-energy},
\Cref{lem:whole-space-coulomb-sobolev-field}, and Young's inequality give
\[
 \frac{\dd}{\dd t}\|g^{(n)}\|_{\mathcal H^{m-1}(t)}^2
 \le C_R\left(\|g^{(n)}\|_{\mathcal H^{m-1}(t)}^2
              +\|g^{(n-1)}\|_{\mathcal H^{m-1}(t)}^2\right).
\]
Since $g^{(n)}(0)=0$, Gronwall's inequality yields
\begin{equation}\label{eq:coulomb-picard-contraction}
 \sup_{0\le r\le t}
 \|g^{(n)}(r)\|_{\mathcal H^{m-1}(r)}^2
 \le C_R t\,
 \sup_{0\le r\le t}
 \|g^{(n-1)}(r)\|_{\mathcal H^{m-1}(r)}^2.
\end{equation}
After decreasing $T_f$ so that $C_RT_f<1$, the sequence
$(f^{(n)})_{n\ge1}$ is Cauchy and
\begin{align*}
 f^{(n)}&\to f
 \qquad\text{in }C_t\mathcal H^{m-1}(t)\cap C_tL^1,\\
 E^{(n)}&\to E_f:=K*\rho_f
 \qquad\text{in }C_tH_x^{m-1}.
\end{align*}
The field convergence follows from
\Cref{lem:whole-space-coulomb-sobolev-field}.  Passing to the limit in the
linear equations and using weak lower semicontinuity gives the VPFP equation and
\begin{equation}\label{eq:coulomb-local-weighted-bound}
 \operatorname*{ess\,sup}_{t\le T_f}\|f(t)\|_{\mathcal H^m(t)}
 +\left(\int_0^{T_f}\|\nabla_vf(t)\|_{\mathcal H^m(t)}^2\dd t\right)^{1/2}
 \le C_R.
\end{equation}
The $C_tL^1$ convergence preserves nonnegativity and unit mass.
Each $f^{(n+1)}$ satisfies
\eqref{eq:coulomb-linear-flow-representation} with field $E^{(n)}$
and initial density $f_0$.  Since $E^{(n)}\to E_f$ uniformly and the
fields have a common Lipschitz bound, the same flow estimate passes
this identity to the limit.  Thus
\eqref{eq:coulomb-linear-flow-representation} holds with
$(h,h_0,E)=(f,f_0,E_f)$ for every bounded Lipschitz $\psi$.
Hence $f(t,z)\dd z$ is the time-$t$
law of the diffusion with field $E_f$, started from a random variable
of density $f_0$ independent of the Brownian motion.

\emph{Regularity of the limit.}
To obtain the weighted bound at every time, fix $t\in[0,T_f]$ and choose
$t_j\to t$ from the full-measure set in
\eqref{eq:coulomb-local-weighted-bound}.  For each multi-index
$\gamma$ with $|\gamma|\le m$, the sequence
$\omega_{t_j}D^\gamma f(t_j)$ has a weakly convergent subsequence in $L^2$.
For $\varphi\in C_c^\infty(\Omega_{\mathcal D})$, integration by parts gives
\[
 \int\omega_{t_j}D^\gamma f(t_j)\varphi
 =(-1)^{|\gamma|}\int f(t_j)D^\gamma(\omega_{t_j}\varphi)
 \longrightarrow
 (-1)^{|\gamma|}\int f(t)D^\gamma(\omega_t\varphi).
\]
Here we used the local strong continuity of $f$ and the smooth dependence
of $\omega_t$ on $t$.  The weak limit is therefore
$\omega_tD^\gamma f(t)$, and weak lower semicontinuity gives the uniform
$\mathcal H^m(t)$ bound at every time.
\Cref{lem:whole-space-coulomb-sobolev-field} then gives the uniform
field bound in \eqref{eq:gauss-sobolev-field-core}.

For $0<\varepsilon<1$, interpolation now gives
\[
 \|f(t)-f(s)\|_{H^{m-\varepsilon}}
 \le (2C_R)^{1-\varepsilon}
      \|f(t)-f(s)\|_{H^{m-1}}^{\varepsilon}.
\]
Thus $f\in C_tH^{m-\varepsilon}$.  Since $m\ge6$, we may choose
$m-\varepsilon>5$, and Sobolev embedding yields $f\in C_tC_b^2$.
The field continuity in $H_x^{m-1}$ gives $E_f\in C_tC_b^1$.
The equation gives
\[
 \partial_tf=-v\cdot\nabla_xf-E_f\cdot\nabla_vf+\sigma\Delta_vf
 \in C([0,T_f];C_{\mathrm{loc}}),
\]
so $f$ is classical in the sense of \Cref{subsec:intro-model}.

\emph{Uniqueness.}
Let $f,g$ be two solutions in the stated class.  Their difference satisfies
\[
 (\partial_t+v\cdot\nabla_x+E_f\cdot\nabla_v-\sigma\Delta_v)(f-g)
 =(E_g-E_f)\cdot\nabla_vg.
\]
The field estimate in \Cref{lem:whole-space-coulomb-sobolev-field} and
\eqref{eq:coulomb-linear-comparison-energy} give
\begin{align*}
 \frac{\dd}{\dd t}\|f-g\|_{\mathcal H^{m-1}(t)}^2
 +\sigma\|\nabla_v(f-g)\|_{\mathcal H^{m-1}(t)}^2
 \le C_T\bigl(1+\|f\|_{\mathcal H^m(t)}
                       +\|g\|_{\mathcal H^m(t)}\bigr)
                 \|f-g\|_{\mathcal H^{m-1}(t)}^2.
\end{align*}
Gronwall's inequality therefore yields
\begin{align}\label{eq:coulomb-nonlinear-stability}
 \|f(t)-g(t)\|_{\mathcal H^{m-1}(t)}^2
 \le\|f(0)-g(0)\|_{\mathcal H^{m-1}(0)}^2\exp\!\left(C_T\int_0^t
  \bigl(1+\|f(r)\|_{\mathcal H^m(r)}
          +\|g(r)\|_{\mathcal H^m(r)}\bigr)\dd r\right).
\end{align}
For $f(0)=g(0)=f_0$, this gives $f=g$ on $[0,T_f]$.
\end{proof}

\subsection{Propagation of moments and pointwise bounds}

\begin{proposition}\label{prop:gaussian-coulomb-solution}
Assume \eqref{eq:intro-coulomb-torus-data}, or
\eqref{eq:intro-coulomb-spatial-condition}--\eqref{eq:intro-coulomb-whole-data},
with the Sobolev exponents of \Cref{cor:intro-quantitative-coulomb}.
Let $f$ be the solution of \Cref{prop:gauss-coulomb-sobolev-core}.
For every $T\le T_f$ there exist $a_T>0$, $C_T<\infty$, and, on $\R^3$,
$M_*>3$ such that
\begin{equation}\label{eq:gauss-solution-pointwise-bound}
 0\le f(t,x,v)\le
 \begin{cases}
 C_Te^{-a_T|v|^2},&\mathcal D=\T^3,\\
 C_T\langle x\rangle^{-M_*}e^{-a_T|v|^2},&\mathcal D=\R^3,
 \end{cases}
 \qquad0\le t\le T.
\end{equation}
On $\R^3$, if the exponents in \Cref{cor:intro-quantitative-coulomb}
satisfy $6\mu+(m-3)P>3(m+3)$, one may take
\begin{equation}\label{eq:coulomb-spatial-pointwise-exponent}
 M_*=\frac{6\mu+(m-3)P}{m+3}>3.
\end{equation}
If instead $\mu>3$, some $M_*>3$ is available for every $P\ge0$.
The velocity derivative also satisfies
\eqref{eq:intro-velocity-derivative}.
Together with \eqref{eq:gauss-sobolev-field-core}, these bounds verify
\eqref{eq:limiting-field-regularity}--\eqref{eq:initial-pointwise-whole}.
\end{proposition}

\begin{proof}
We first propagate the integral moments.  Fix $T\le T_f$, set
$L_T=\sup_{t\le T}\|E_f(t)\|_\infty$, and define
\[
 \mathcal A_t=v\cdot\nabla_x+E_f(t,x)\cdot\nabla_v+\sigma\Delta_v.
\]
Put $P=0$ on $\T^3$, and let $P$ be the exponent in the initial moment
assumption on $\R^3$.  For a positive function $a\in C^1([0,T])$, set
\[
 W_t(x,v)=
 \begin{cases}
 e^{a(t)(1+|v|^2)},&\mathcal D=\T^3,\\
 \langle x\rangle^P e^{a(t)(1+|v|^2)},&\mathcal D=\R^3.
 \end{cases}
\]
A direct calculation gives
\begin{align*}
 W_t^{-1}(\partial_t+\mathcal A_t)W_t
 &=a'(1+|v|^2)
 +\mathbf1_{\{\mathcal D=\R^3\}}
      P\frac{x\cdot v}{1+|x|^2}
 +2aE_f\cdot v+6\sigma a+4\sigma a^2|v|^2.
\end{align*}
Since $|x|/(1+|x|^2)\le1/2$, we have
\[
 P\frac{x\cdot v}{1+|x|^2}
 \le \frac P2|v|
 \le a^2|v|^2+\frac{P^2}{16a^2},
 \qquad
 2aE_f\cdot v\le a^2|v|^2+L_T^2,
\]
where the first term is absent on $\T^3$.  Choose $a$ to satisfy
\begin{equation}\label{eq:mixed-moment-a-ode}
 a'+\bigl(4\sigma+1+\mathbf1_{\{\mathcal D=\R^3\}}\bigr)a^2=0,
 \qquad a(0)=a_0/2.
\end{equation}
Solving this equation gives
\[
 a(t)=\frac{a(0)}{1+\bigl(4\sigma+1+
       \mathbf1_{\{\mathcal D=\R^3\}}\bigr)a(0)t}>0.
\]
Since $a'\le0$, the preceding inequalities yield
\[
 (\partial_t+\mathcal A_t)W_t\le C_4W_t,
 \qquad
 C_4:=\sup_{0\le t\le T}
 \left(6\sigma a(t)+L_T^2+\frac{P^2}{16a(t)^2}\right)<\infty.
\]
By \eqref{eq:coulomb-linear-flow-representation} with
$(h,h_0,E)=(f,f_0,E_f)$, $f(t,z)\dd z$ is
the law of $Z_t=\Phi^{E_f}_{0,t}(Z_0)$, where $Z_0$ has density $f_0$
and is independent of the Brownian motion.  The initial moment is finite
because $a(0)<a_0$.  We set
$
 \tau_R:=\inf\{t\ge0:r_{\mathcal D}(Z_t)\ge R\}.
$
Since
\[
 \mathbb E\int_0^{t\wedge\tau_R}
       |\nabla_vW_s(Z_s)|^2\dd s\le C_{T,R}t,
\]
the stopped stochastic integral has zero expectation.  It\^o's formula gives
\begin{align*}
 \mathbb E\,W_{t\wedge\tau_R}(Z_{t\wedge\tau_R})
 &=\mathbb E\,W_0(Z_0)
   +\mathbb E\int_0^t\mathbf1_{\{s<\tau_R\}}
       (\partial_s+\mathcal A_s)W_s(Z_s)\dd s\\
 &\le\mathbb E\,W_0(Z_0)
   +C_4\int_0^t\mathbb E\,W_{s\wedge\tau_R}
                       (Z_{s\wedge\tau_R})\dd s\le e^{C_4t}\mathbb E\,W_0(Z_0).
\end{align*}
The field is bounded on $[0,T]$, so the diffusion is nonexplosive.
As $R\to\infty$, Fatou's lemma yields
\begin{equation}\label{eq:mixed-moment-gronwall}
 \int f(t,z)W_t(z)\dd z
 \le e^{C_4t}\int f_0(z)W_0(z)\dd z,
 \qquad 0\le t\le T;
\end{equation}
see, for instance, \cite[Section~3.3]{KaratzasShreve1991}.  Replacing $a(t)$
by $\underline a_T:=\min_{0\le t\le T}a(t)>0$ gives a uniform bound for
the same integral moments on $[0,T]$.

Combine these moment bounds with the uniform weighted Sobolev bound in
\eqref{eq:gauss-sobolev-field-core}.  On $\T^3$,
\Cref{lem:gauss-integral-to-pointwise} gives the Gaussian velocity bound.
On $\R^3$, \Cref{lem:weighted-moment-pointwise} gives the exponent
\eqref{eq:coulomb-spatial-pointwise-exponent} when
$6\mu+(m-3)P>3(m+3)$; if $\mu>3$, choose the exponent in
\eqref{eq:weighted-moment-spatial-exponent} larger than three.
These are precisely the alternatives in
\eqref{eq:intro-coulomb-spatial-condition}.
Thus \eqref{eq:gauss-solution-pointwise-bound} holds uniformly for
$0\le t\le T$.
Weighted Cauchy--Schwarz, using the same integrable inverse weights as in
the density estimate, also gives
\[
 \int_0^T\|\nabla_v f(t)\|_1\dd t
 \le C\int_0^T\|f(t)\|_{\mathcal H^m(t)}\dd t<\infty
\]
by \eqref{eq:coulomb-local-weighted-bound}.
\end{proof}

\section*{Statements and Declarations}
\textbf{Funding.} This work was supported in part by the National Natural
Science Foundation of China under Grant No.~12371180.\par
\textbf{Competing interests.} The authors declare that they have no competing
financial or non-financial interests.\par
\noindent\textbf{Data availability.} No datasets were generated or analyzed in the
course of this study.\par
\noindent\textbf{Use of generative artificial intelligence.} During the preparation of this manuscript, the authors used GPT-6 to assist with parts of the existence and regularity arguments for the limiting Vlasov--Riesz--Fokker--Planck and Vlasov--Poisson--Fokker--Planck equations. The authors reviewed and revised all material produced with AI assistance and take responsibility for the mathematical content of the manuscript.

\end{document}